\documentclass[]{scrartcl}

\usepackage[utf8]{luainputenc}
\usepackage[USenglish]{babel}
\usepackage{csquotes}

\usepackage[a4paper,top=27mm,bottom=20mm,inner=25mm,outer=20mm]{geometry}

\usepackage{graphicx}
\usepackage{verbatim}
\usepackage{stmaryrd}
\usepackage{amsmath}
\usepackage{comment}
\usepackage{listings}
\usepackage[table,xcdraw]{xcolor}
\usepackage{color}
\usepackage{amsthm}
\usepackage{times}
\usepackage[hidelinks]{hyperref}
\usepackage{multirow}
\usepackage{makecell}
\usepackage{booktabs}
\usepackage{array}
\usepackage{enumitem}
\usepackage{amssymb}
\usepackage{subcaption}
\usepackage{siunitx}
\usepackage{bm}
\usepackage[normalem]{ulem}
\usepackage{adjustbox}

\usepackage[T1]{fontenc}
\usepackage{newpxtext,newpxmath}

\usepackage[%
  backend=bibtex,bibencoding=ascii,
  style=numeric-comp,
  giveninits=true, uniquename=init, 
  natbib=true,
  url=false,
  doi=true,
  isbn=false,
  backref=false,
  maxnames=99,
]{biblatex}
\definecolor{dkgreen}{rgb}{0,0.6,0}
\definecolor{gray}{rgb}{0.5,0.5,0.5}
\definecolor{mauve}{rgb}{0.58,0,0.82}

\newcommand*{\ldblbrace}{\{\mskip-5mu \left \{}
\newcommand*{\rdblbrace}{ \}\mskip-5mu \right \}}
\newcommand{\mean}[1]{\left\ldblbrace #1 \right\rdblbrace}
\newcommand*{\ldblparen}{(\mskip-3mu \left (}
\newcommand*{\rdblparen}{ )\mskip-3mu \right )}
\newcommand{\prodmean}[1]{\left\ldblparen #1 \right\rdblparen}
\newcommand{\mtx}[1]{\boldsymbol{#1}}
\renewcommand{\vec}[1]{\boldsymbol{#1}}
\newcommand{\vecs}[1]{\underline{#1}}
\newcommand{\imtx}[1]{\smash{\underline{\underline{#1}}}}
\newcommand{\mtxs}[1]{\underline{\underline{#1}}}

\newcommand{\jump}[1]{\left\llbracket #1  \right\rrbracket}
\newtheorem{theorem}{Theorem}[section]

\newtheorem{lemma}[theorem]{Lemma}
\newtheorem{remark}[theorem]{Remark}
\theoremstyle{definition}
\newtheorem{definition}[theorem]{Definition}
\newcommand{\todo}[1]{{\large\color{red}{#1}}}

\newcommand{\orcid}[1]{ORCID:~\href{https://orcid.org/#1}{#1}}
\usepackage{authblk}

\definecolor{dkgreen}{rgb}{0,0.6,0}
\definecolor{gray}{rgb}{0.5,0.5,0.5}
\definecolor{mauve}{rgb}{0.58,0,0.82}

\newenvironment{keywords}{\par\textbf{Key words.}}{\par}
\newenvironment{AMS}{\par\textbf{AMS subject classification.}}{\par}

\title{On Affordable High-Order Entropy-Conservative/Stable and Well-Balanced Methods for Nonconservative Hyperbolic Systems}

\author[1]{Marco Artiano\thanks{\orcid{0009-0009-5872-702X}}}
\affil[1]{Institute of Mathematics, Johannes Gutenberg University Mainz, Germany}

\author[1]{Hendrik Ranocha\thanks{\orcid{0000-0002-3456-2277}}}

\date{March 20, 2026}

\begin{document}
\maketitle

\begin{abstract}
Many entropy-conservative and entropy-stable (summarized as entropy-preserving) methods for hyperbolic conservation laws rely on Tadmor's theory for two-point entropy-preserving numerical fluxes and its higher-order extension via flux differencing using summation-by-parts (SBP) operators, e.g., in discontinuous Galerkin spectral element methods (DGSEMs).
The underlying two-point formulations have been extended to nonconservative systems using fluctuations by Castro et al.\ (2013, \href{https://doi.org/doi/10.1137/110845379}{doi:10.1137/110845379}) with follow-up generalizations to SBP methods.
We propose specific forms of entropy-preserving fluctuations for nonconservative hyperbolic systems that are simple to interpret and allow an algorithmic construction of entropy-preserving methods.
We analyze necessary and sufficient conditions, and obtain a full characterization of entropy-preserving three-point methods within the finite volume framework.
This formulation is extended to SBP methods in multiple space dimensions on Cartesian and curvilinear meshes.
Additional properties such as well-balancedness extend naturally from the underlying finite volume method to the SBP framework.
We use the algorithmic construction enabled by the chosen formulation to derive several new entropy-preserving schemes for nonconservative hyperbolic systems, e.g., the compressible Euler equations of an ideal gas using the internal energy equation and a dispersive shallow-water model.
Numerical experiments show the robustness and accuracy of the proposed schemes.
\end{abstract}


\newpage
\section{Introduction}
For hyperbolic conservation laws such as
\begin{equation}
	\partial_t \vec{u} + \partial_x \vec{f}(\vec{u}) = 0,
	\label{conservation_law}
\end{equation}
convex entropies are crucial for analysis and numerics.
From a theoretical point of view, they provide stability for classical solutions~\cite[Section~V]{dafermos2016hyperbolic} as well as selection criteria for weak~\cite[Sections~VIII and IX]{dafermos2016hyperbolic} or even more generalized solutions~\cite{feireisl2021numerical}.
From a numerical point of view, entropy-stable methods improve the robustness of under-resolved simulations significantly~\cite{gassner2016split,Chan2022OnTE,sjogreen2018high,rojas2021robustness,badrkhani2025entropy,aiello2026impact}.
Today, many efficient and entropy-conservative/stable (EC/ES, summarized as entropy-preserving) methods for conservation laws~\cite{ranocha2023efficient,winters2018comparative} are based on the seminal work of Tadmor~\cite{Tadmor1987,Tadmor2003}, who introduced the concept of entropy-preserving numerical fluxes for finite volume methods
\begin{equation}
	\partial_t \vec{u}_i
    + \frac{1}{\Delta x_i} \Bigl (
        \vec{f}^{\mathrm{num}}( \vec{u}_{i}, \vec{u}_{i+1} )
        - \vec{f}^{\mathrm{num}}( \vec{u}_{i-1}, \vec{u}_{i} )
    \Bigr ) = \vec{0}.
	\label{FV_conservation}
\end{equation}
The extension to high-order methods has been successively obtained in~\cite{LeFloch2006,FISHER2013518,chen2017entropy,Ranocha2018,chan2018discretely,chan2019skew}, relying on the general framework of summation-by-parts (SBP) operators~\cite{svard2014review,fernandez2014review}.
SBP operators are discrete derivative and integration operators that mimic integration by parts.
While they have originally been developed for finite differences (FDs)
\cite{kreiss1974finite,strand1994summation,carpenter1994time},
their general framework also includes finite volumes~\cite{nordstrom2001finite},
continuous Galerkin (CG) finite elements~\cite{Hicken2016_SBP,hicken2020entropy,abgrall2020analysisI},
discontinuous Galerkin~(DG) methods~\cite{gassner2013skew,carpenter2014entropy},
flux reconstruction~(FR)~\cite{huynh2007flux,vincent2011newclass,ranocha2016summation},
active flux methods~\cite{eymann2011active,barsukow2025stability},
meshless schemes~\cite{hicken2025constructing,kwan2025robust},
and cut-cell methods~\cite{petri2026kinetic,taylor2025entropy}.

In contrast to classical conservation laws such as \eqref{conservation_law}, nonconservative hyperbolic systems
\begin{equation}
\partial_t \vec{u} + \vec{A}(\vec{u}) \partial_x \vec{u} = \vec{0}
\label{special_nonconservative_system}
\end{equation}
are more difficult to handle.
For example, the nonconservative product needs to be defined carefully for weak solutions; the shock relations depend not only on the solution states but also on the path connecting these states~\cite{Maso1995DefinitionAW}.
This has led to the development of path-conservative numerical methods for nonconservative systems~\cite{Pares2006}; extensions to DG methods have been carried out by various authors~\cite{RHEBERGEN20081887,FRANQUET20124096}.
However, path conservation alone is not enough to ensure correct solutions~\cite{CASTRO20088107,ABGRALL20102759}.
Indeed, nonconservative systems are susceptible to the correct form of (artificial, effective) dissipation.
Thus, entropy-preserving schemes are important here as well~\cite{Castro2013,beljadid2017schemes}.
The extension of Tadmor's entropy framework to nonconservative systems has been initiated in~\cite{Castro2013} by considering finite volume methods in fluctuation form, i.e.,
\begin{equation}
	\partial_t \vec{u}_i
    + \frac{1}{\Delta x_i} \Bigl (
        \vec{D}^+( \vec{u}_{i-1}, \vec{u}_{i} ) +
        \vec{D}^-( \vec{u}_{i}, \vec{u}_{i+1} )
    \Bigr ) = \vec{0},
	\label{FV_fluctuations}
\end{equation}
where the fluctuations satisfy the consistency condition $\vec{D}^\pm(\vec{u}, \vec{u}) = \vec{0}$ and an additional path consistency condition involving an integral~\cite{Pares2006,Castro2013}.
Assuming that the nonconservative system \eqref{special_nonconservative_system} is equipped with a convex entropy $U$, i.e., the entropy inequality
\begin{equation}
    \partial_t U(\vec{u}) + \partial_x F(\vec{u}) \leq 0
\end{equation}
is demanded for weak solutions and holds as equality for smooth solutions, Castro et al.~\cite[Theorem~2.1]{Castro2013} derived the sufficient condition
\begin{equation}
    \forall \vec{u}_-, \vec{u}_+\colon \quad
    \vec{\omega}(\vec{u}_-) \cdot \vec{D}^-( \vec{u}_{-}, \vec{u}_{+} )
    + \vec{\omega}(\vec{u}_+) \cdot \vec{D}^+( \vec{u}_{-}, \vec{u}_{+} )
    = F(\vec{u}_+) - F(\vec{u}_-)
    \label{CastroEC}
\end{equation}
for an EC method, where $\vec{\omega} = U'(\vec{u})$ are the entropy variables.
These methods have been extended to 1D SBP DG methods by Renac~\cite{RENAC20191,coquel2021entropy} and to curved meshes in multiple dimensions by Waruszewski et al.~\cite{WARUSZEWSKI2022111507}, who also considered nonconservative terms of the form $\vec{H}(\vec{u}) \partial_x \vec{g}(\vec{u})$ instead of $\vec{A}(\vec{u}) \partial_x \vec{u}$.
Recently, Ersing and Winters~\cite{ersing2026newclassentropystable} proposed two particular strategies to compute EC fluctuations: one based on an integral formulation and another assuming a specific form of two-point fluxes, provided that a closed-form expression can be found.
However, their construction still relies on integrals.

Tadmor presented a general choice of EC fluxes for conservation laws~\cite{Tadmor1987,Tadmor2003} based on an integral form.
In practice, these integrals are typically hard or even impossible to evaluate (exactly and efficiently).
Thus, it took some time until affordable EC fluxes have been constructed without the need to compute integrals~\cite{ismail2009affordable,ECChandra,Ranocha2018}.
Moreover, Castro et al.~\cite[Remark~4.2]{Castro2013} observed:
\begin{quote}
    This example indicates that the choice of paths is not crucial in determining which solutions are approximated by the scheme.
    Instead, the numerical viscosity operator (that matches with the underlying viscosity) decides which weak solution the scheme will converge to.
\end{quote}
This is fortunate, since the correct choice of paths is often far from trivial, and many works have thus resorted to the default choice of a straight line path.
This motivates our current manuscript.
By focusing on the entropy, we first propose concrete forms of fluctuations that are easy to interpret and amenable to the algorithmic construction of affordable entropy-conserving methods, without the need to evaluate any integrals (including integrals related to choosing a path).
We provide both necessary and sufficient conditions for entropy conservation and stability, and demonstrate how to use them to construct entropy-preserving methods for several hyperbolic systems.
Moreover, we extend the results to high-order SBP methods on curved meshes in multiple dimensions.
Furthermore, we demonstrate how well-balanced three-point finite volume schemes extend to well-balanced high-order SBP methods on curved meshes.

In contrast to works on path-conservative methods for nonconservative systems, we consider the general form
\begin{equation}
    \partial_t \vec{u} + \partial_x \vec{f}(\vec{u}) + \vec{H}(\vec{u}) \partial_x \vec{g}(\vec{u}) = \vec{0},
    \label{nonconservative_system}
\end{equation}
separating conservative parts from nonconservative ones and allowing general derivatives in nonconservative terms.
This is important for many applications, e.g., to allow differentiating the pressure or the velocity for the compressible Euler equations, which are not conserved variables.
Moreover, we do not only provide sufficient conditions, but also study their necessity in a more general context without requiring convexity of the entropy.
In fact, while Tadmor's integral form of an EC flux requires a convex entropy (to obtain the entropy variables $\vec{\omega} = U'(\vec{u})$ as a bijective mapping from the conserved variables $\vec{u}$), his classical conditions for EC fluxes do not require convexity of the entropy, but are at the heart a set of algebraic relations.
We demonstrate this in Section~\ref{sec:conservation_laws} before extending the results to nonconservative systems \eqref{nonconservative_system} in Section~\ref{sec:nonconservative_systems}.
This is essential and has not been appreciated sufficiently, since the same conditions allow the derivation of general conservation relations, e.g., the conservation of the total energy when working with the potential temperature as thermodynamic state variable for atmospheric flows~\cite{artiano2025structurepreservinghighordermethodscompressible}.

Entropy-preserving methods for conservation laws are linked to split forms, e.g., kinetic energy-preserving methods \cite{gassner2016split} or alternative approaches for entropy-conserving methods \cite{DEMICHELE2025114262}.
We extend these works by pointing out explicitly how certain split forms can be recovered using the concrete forms of fluctuations considered in this work.

This paper is organized as follows: in Section~\ref{sec:conservation_laws}, we formulate Tadmor's condition as a purely algebraic condition and demonstrate that it is necessary and sufficient also for non-convex functionals.
In Section~\ref{sec:nonconservative_systems}, we formulate an algebraic characterization of entropy-preserving finite volume methods in fluctuation form. Specific semi-discrete forms of the nonconservative term are also introduced, for which the corresponding algebraic characterization and entropy-preserving conditions are derived.
In Section~\ref{sec:numerical_fluxes}, we demonstrate the systematic construction of entropy-preserving numerical fluxes based on the specific semi-discrete forms of the nonconservative terms for several systems.
The extension to high-order summation-by-parts (SBP) operators in both Cartesian and curvilinear coordinates, including a discussion about the generation of nonconservative split forms, is presented in Sections~\ref{sec:sbp_cartesian} and~\ref{sec:sbp_curvilinear}, including the extension of well-balanced schemes to curved meshes.
In Section~\ref{sec:numerical_experiments}, we present numerical experiments to assess the results of the previous sections.
Finally, concluding remarks and future perspectives are given in Section~\ref{sec:summary_conclusions}.
\section{Hyperbolic conservation laws}
\label{sec:conservation_laws}
We consider the hyperbolic conservation law
\begin{equation}
\tag{\ref{conservation_law}}
	\partial_t \vec{u} + \partial_x \vec{f}(\vec{u}) = \vec{0},
\end{equation}
where $\vec{u}(x,t) \in Y \subset \mathbb{R}^n$ are the conserved quantities, $Y$ is an appropriate open set, and $\vec{f}\colon Y \to \mathbb{R}^n$ is the flux, supplemented with appropriate initial and boundary conditions.
Tadmor's classical theory~\cite{Tadmor1987,Tadmor2003} builds on the existence of a convex entropy $U\colon Y \to \mathbb{R}$ with associated entropy flux $F\colon Y \to \mathbb{R}$ satisfying
\begin{equation}\label{chain_rule_ec}
    \forall \vec{u} \in Y\colon \quad U'(\vec{u}) \cdot \vec{f}'(\vec{u}) = F'(\vec{u}).
\end{equation}
Thus, smooth solutions of~\eqref{conservation_law} also satisfy the additional conservation law
\begin{equation}
    \partial_t U(\vec{u}) + \partial_x F(\vec{u}) = 0,
    \label{entropy_equality}
\end{equation}
and the entropy inequality
\begin{equation}
    \partial_t U(\vec{u}) + \partial_x F(\vec{u}) \leq 0
    \label{entropy_inequality}
\end{equation}
is used as admissibility criterion for (weak) solutions~\cite{dafermos2016hyperbolic}.

A three-point finite volume semi-discretization of~\eqref{conservation_law} can be written as
\begin{equation}
\tag{\ref{FV_conservation}}
	\partial_t \vec{u}_i + \frac{1}{\Delta x_i} \Bigl ( \vec{f}^{\mathrm{num}}\bigl( \vec{u}_{i}, \vec{u}_{i+1} \bigr)  - \vec{f}^{\mathrm{num}}\bigl( \vec{u}_{i-1}, \vec{u}_{i} \bigr) \Bigr ) = \vec{0},
\end{equation}
where $\vec{u}_i$ is the approximate cell-averaged solution in the $i$-th cell, $\Delta x_i$ is the cell-size, and $\vec{f}^{\mathrm{num}}\colon Y \times Y \to \mathbb{R}^n$ is a consistent numerical flux function, i.e.,
\begin{equation}
    \forall \vec{u} \in Y\colon \quad \vec{f}^{\mathrm{num}}(\vec{u},\vec{u}) = \vec{f}(\vec{u}).
\end{equation}
Following Tadmor~\cite{Tadmor1987,Tadmor2003}, we are interested in numerical schemes that mimic the entropy equality~\eqref{entropy_equality} or inequality~\eqref{entropy_inequality} discretely.

\begin{definition}\label{def_entropy_conservation}
A finite volume scheme like \eqref{FV_conservation} is entropy-conservative/stable if there is a numerical entropy flux $F^{\mathrm{num}}\colon Y \times Y \to \mathbb{R}$, with $F^{\mathrm{num}}(\vec{u},\vec{u}) = F(\vec{u})$, such that
\begin{equation}
	 U'(\vec{u}_i) \cdot \partial_t \vec{u}_i
     =
     \vec{\omega}_i \cdot \partial_t \vec{u}_i
     \leq
     - \frac{1}{\Delta x_i} \Bigl ( F^{\mathrm{num}}\bigl( \vec{u}_{i}, \vec{u}_{i+1} \bigr) - F^{\mathrm{num}}\bigl( \vec{u}_{i-1}, \vec{u}_{i} \bigr) \Bigr ),
\end{equation}
where $\vec{\omega} = U'$ are the entropy variables and equality has to hold for the entropy-conservative case.
\end{definition}

In his seminal work, Tadmor~\cite{Tadmor1987,Tadmor2003} has obtained local conditions on numerical fluxes ensuring entropy conservation/stability that have been used ever since as a systematic tool to construct entropy-preserving numerical fluxes~\cite{ismail2009affordable,ECChandra,Ranocha2018,thein2025computing,winters2020entropy,lefloch2021kinetic,artiano2025structurepreservinghighordermethodscompressible}.
Although not written out explicitly in detail, the condition for entropy conservation is also observed to be necessary~\cite[page~95]{Tadmor1987}.
However, the original derivation relies on the convexity of the entropy $U$, which guarantees the invertibility of the mapping between conserved variables $\vec{u}$ and the entropy variables $\vec{\omega} = U'(\vec{u})$.
As a preparation for the extension to nonconservative systems, we will next point out that convexity of the entropy is not required for the characterization of entropy-preserving methods and fluxes.

\subsection{Tadmor condition for non-convex functionals}{\label{tadmor_non_convex_entropies}}
In the following, we introduce some standard operators and notations that will be used throughout the paper. We first define the usual mean and jump operators as
\begin{equation}
\mean{a} = \frac{a_+ + a_-}{2}, \qquad \jump{a} = a_+ - a_-,
\end{equation}
where $a_\pm$ denote the values of $a$ at the right/left side of an interface.
For numerical fluxes at a specific interface, we adopt the notation
$\vec{f}^{\mathrm{num}}_{i+1/2} = \vec{f}^{\mathrm{num}}(\vec{u}_{i}, \vec{u}_{i+1})$,
where the first argument of $\vec{f}^{\mathrm{num}}(\vec{u}_-, \vec{u}_+)$ represents the value at the left of the interface and the second argument represents the value at the right.
To simplify notation further, we also write
\begin{equation}
g(\vec{u}_\pm) = g_\pm
\end{equation}
for any function $g$, e.g., the flux $\vec{f}$ or the entropy variables $\vec{\omega}$.

The characterization of entropy-preserving methods and numerical fluxes ultimately relies only on three key ingredients: the numerical flux $\vec{f}^{\mathrm{num}}$, the entropy variables $\vec{\omega}$, and the entropy flux $F$.
Tadmor's condition can be expressed as a purely algebraic condition summarized in the following lemma, where we still use the same notation as introduced earlier, but we do not assume any further relationship between the functions $\vec{\omega}$, $\vec{f}$, and $F$.

\begin{lemma}{\label{algebraic_lemma}}
    Let $Y \subset \mathbb{R}^n$ be open, and $\vec{\omega}\colon Y \to \mathbb{R}^n$, $\vec{f}\colon Y \to \mathbb{R}^n$, $F\colon Y \to \mathbb{R}$, and $\vec{f}^{\mathrm{num}}\colon Y \times Y \to \mathbb{R}^n$ satisfying
    \begin{equation}
    \forall \vec{u} \in Y\colon \quad \vec{f}^\mathrm{num}(\vec{u},\vec{u}) = \vec{f}(\vec{u})
    \end{equation}
    be given. Then,
    $\exists F^\mathrm{num}\colon Y \times Y \to \mathbb{R}$ satisfying $\forall \vec{u} \in Y\colon F^\mathrm{num}(\vec{u},\vec{u}) = F(\vec{u})$ and
    \begin{equation}
    \forall \vec{u}_-, \vec{u}_0, \vec{u}_+ \in Y\colon \quad \vec{\omega}(\vec{u}_0) \cdot \Bigl ( \vec{f}^\mathrm{num}(\vec{u}_0, \vec{u}_+) - \vec{f}^\mathrm{num}(\vec{u}_-, \vec{u}_0) \Bigr ) \geq F^\mathrm{num}(\vec{u}_0, \vec{u}_+) - F^\mathrm{num}(\vec{u}_-, \vec{u}_0)
    \label{sufficient_condition}
    \end{equation}
    if and only if
    \begin{equation}
    \forall \vec{u}_+, \vec{u}_- \in Y\colon \quad \left (\vec{\omega}_+ - \vec{\omega}_- \right ) \cdot \vec{f}^\mathrm{num}(\vec{u}_-, \vec{u}_+) \leq \left (\vec{\omega}_+ \cdot \vec{f}_+ - F_+ \right ) - \left (\vec{\omega}_- \cdot \vec{f}_- - F_- \right ).
    \label{necessary_condition}
    \end{equation}
    Moreover, if one of the conditions is satisfied with equality, then equality holds in the other condition as well, and $F^\mathrm{num}$ in~\eqref{sufficient_condition} is determined uniquely as
    \begin{equation}
    F^{\mathrm{num}} = \mean{F} + \mean{\vec{\omega}} \cdot \vec{f}^{\mathrm{num}}  -\mean{\vec{\omega} \cdot \vec{f}}.
    \label{F_num}
    \end{equation}
\end{lemma}
\begin{proof}
We first show that~\eqref{sufficient_condition} implies~\eqref{necessary_condition}.
First, we choose  $\vec{u}_0 = \vec{u}_-$ in~\eqref{sufficient_condition}, resulting in
\begin{equation}
\vec{\omega}_- \cdot \Bigl (\vec{f}^{\mathrm{num}}(\vec{u}_-, \vec{u}_+) - \vec{f}_- \Bigr ) \geq F^{\mathrm{num}}(\vec{u}_-, \vec{u}_+) - F_-.
\label{eqminus}
\end{equation}
Choosing $\vec{u}_0 = \vec{u}_+$ in~\eqref{sufficient_condition} yields
\begin{equation}
\vec{\omega}_+ \cdot \Bigl (\vec{f}_+ - \vec{f}^{\mathrm{num}}(\vec{u}_-, \vec{u}_+) \Bigr ) \geq F_+ - F^{\mathrm{num}}(\vec{u}_-, \vec{u}_+).
\label{eqplus}
\end{equation}
Adding the two inequalities, we obtain
\begin{equation}
\left ( \vec{\omega}_- - \vec{\omega}_+ \right ) \cdot \vec{f}^{\mathrm{num}}(\vec{u}_-, \vec{u}_+) \geq \left (\vec{\omega}_- \cdot \vec{f}_- - F_- \right ) - \left (\vec{\omega}_+ \cdot \vec{f}_+ - F_+ \right ),
\end{equation}
which is equivalent to~\eqref{necessary_condition}.

If equality holds in~\eqref{sufficient_condition}, it holds in~\eqref{necessary_condition} as well.
In this case, we can subtract~\eqref{eqplus} from~\eqref{eqminus} (both with equality) to obtain
\begin{equation}
F^{\mathrm{num}}(\vec{u}_-, \vec{u}_+) = \frac{\vec{\omega}_+ + \vec{\omega}_-}{2} \cdot \vec{f}^{\mathrm{num}}(\vec{u}_-, \vec{u}_+) - \frac{\vec{\omega}_+ \cdot \vec{f}_+ + \vec{\omega}_- \cdot \vec{f}_-}{2} + \frac{F_+ + F_-}{2},
\end{equation}
i.e., $F^\mathrm{num}$ is determined uniquely by~\eqref{F_num}.

Finally, we show that~\eqref{necessary_condition} implies~\eqref{sufficient_condition}.
Abbreviating $\psi = \vec{\omega} \cdot \vec{f} - F$, \eqref{necessary_condition} is equivalent to
\begin{equation}
-\left (\vec{\omega}_+ - \vec{\omega}_- \right ) \cdot \vec{f}^\mathrm{num}(\vec{u}_-, \vec{u}_+) \ge - \left(\psi_+ - \psi_-\right).
\label{necessary_condition_psi}
\end{equation}
Following Tadmor, we compute
\begin{equation}
\begin{aligned}
    &\qquad
    \vec{\omega}(\vec{u}_0) \cdot \Bigl ( \vec{f}^\mathrm{num}(\vec{u}_0, \vec{u}_+) - \vec{f}^\mathrm{num}(\vec{u}_-, \vec{u}_0) \Bigr )
    \\
    &=
    \frac{\vec{\omega}_0 + \vec{\omega}_+}{2} \cdot \vec{f}^\mathrm{num}(\vec{u}_0, \vec{u}_+)
    - \frac{\vec{\omega}_+ - \vec{\omega}_0}{2} \cdot \vec{f}^\mathrm{num}(\vec{u}_0, \vec{u}_+)
    \\
    &\quad
    - \frac{\vec{\omega}_- + \vec{\omega}_0}{2} \cdot \vec{f}^\mathrm{num}(\vec{u}_-, \vec{u}_0)
    - \frac{\vec{\omega}_0 - \vec{\omega}_-}{2} \cdot \vec{f}^\mathrm{num}(\vec{u}_-, \vec{u}_0)
    \\
    &\stackrel{\eqref{necessary_condition_psi}}{\ge}
    \frac{\vec{\omega}_0 + \vec{\omega}_+}{2} \cdot \vec{f}^\mathrm{num}(\vec{u}_0, \vec{u}_+)
    - \frac{\psi_+ - \psi_0}{2}
    - \frac{\vec{\omega}_- + \vec{\omega}_0}{2} \cdot \vec{f}^\mathrm{num}(\vec{u}_-, \vec{u}_0)
    - \frac{\psi_0 - \psi_-}{2}
    \\
    &=
    \frac{\vec{\omega}_0 + \vec{\omega}_+}{2} \cdot \vec{f}^\mathrm{num}(\vec{u}_0, \vec{u}_+)
    - \frac{\psi_+ + \psi_0}{2}
    - \frac{\vec{\omega}_- + \vec{\omega}_0}{2} \cdot \vec{f}^\mathrm{num}(\vec{u}_-, \vec{u}_0)
    + \frac{\psi_0 + \psi_-}{2}
    \\
    &=
    F^{\mathrm{num}}(\vec{u}_0, \vec{u}_+) - F^{\mathrm{num}}(\vec{u}_-, \vec{u}_0),
\end{aligned}
\end{equation}
where we have used the definition of $F^{\mathrm{num}}$ in~\eqref{F_num} in the last step.
\end{proof}

\begin{remark}
    The condition~\eqref{necessary_condition} (with equality) is often referred to as the Tadmor ``shuffle condition'' in the literature, and it is often used as a sufficient condition for entropy conservation/stability.
    However, it is also a necessary condition, and thus, it provides a complete characterization of entropy-preserving methods.
    Moreover, \eqref{F_num} determines the numerical entropy flux of an EC scheme.
    In particular, a numerical entropy flux such as $\mean{\vec{\omega}} \cdot \vec{f}^{\mathrm{num}}  - \psi(\mean{\vec{\omega}})$ \cite[Eq.~(33)]{abgrall2018general}, where $\psi = \vec{\omega} \cdot \vec{f} - F$, cannot be used in this setting.
\end{remark}

\begin{remark}
    The importance of the general Tadmor shuffle condition \eqref{necessary_condition} cannot be overstated, since it can be used for any functional.
    For example, it can be used for the derivation of pressure equilibrium preserving methods.
    Indeed, the condition to preserve pressure equilibria used in \cite{terashima2025approximately,fujiwara2023fully} reduced to a single species is
    \begin{equation}
        \left. \frac{\partial \varrho e}{\partial \varrho} \right|_{p, \bm{u}_i}
        ( \varrho^\mathrm{num}_{i + 1/2} - \varrho^\mathrm{num}_{i - 1/2} )
        =
        (\varrho e)^\mathrm{num}_{i + 1/2} - (\varrho e)^\mathrm{num}_{i - 1/2},
    \label{eq:pe_condition}
    \end{equation}
    where $\varrho^\mathrm{num}_{i \pm 1/2}$ and $(\varrho e)^\mathrm{num}_{i \pm 1/2}$ are consistent approximations (``numerical fluxes'') of the density and the internal energy density at the $i \pm 1/2$ interface, respectively.
    Adapted to this notation, Terashima et al.~\cite{terashima2025approximately} write
    \begin{quote}
        However, as recognized in Eq.~\eqref{eq:pe_condition}, the
        asymmetry associated with the partial derivative term
        $\left. \frac{\partial \varrho e}{\partial \varrho} \right|_{p, \bm{u}_i}$
        defined at $i$ means that no locally-defined half-point values
        satisfy Eq.~\eqref{eq:pe_condition} for all cells.
    \end{quote}
    However, this is not true since \eqref{eq:pe_condition} is of the general form \eqref{sufficient_condition} with $\vec{\omega} = \partial \varrho e / \partial \vec{u}$, $\vec{f}^\mathrm{num} = \varrho^\mathrm{num}$, and $F^\mathrm{num} = \varrho e^\mathrm{num}$.
    Thus, it can be reduced to the local condition \eqref{necessary_condition}.
    We will make use of this in an upcoming work.
\end{remark}

The following theorem captures the essential structure of Tadmor's condition while removing the assumption of entropy convexity.
Indeed, at a purely algebraic level, the invertibility of the mapping between conserved and entropy variables (which is guaranteed by the convexity of the entropy function) is not required.
Consequently, the same condition can also be used to ensure the conservation of non-convex functionals.

\begin{theorem}[Tadmor]{\label{tadmor_theorem}}
The FV scheme
\begin{equation}
\tag{\ref{FV_conservation}}
	\partial_t \vec{u}_i + \frac{1}{\Delta x_i} \Bigl ( \vec{f}^{\mathrm{num}}\bigl ( \vec{u}_{i}, \vec{u}_{i+1} \bigr )  - \vec{f}^{\mathrm{num}}\bigl ( \vec{u}_{i-1}, \vec{u}_{i} \bigr ) \Bigr ) = \vec{0}
\end{equation}
is entropy-conservative/stable for a given entropy pair $(U, F)$ (see Definition~\ref{def_entropy_conservation}) if and only if
\\
\begin{equation}
\label{TadmorDef}
	\jump{\vec{\omega}} \cdot \vec{f}^\mathrm{num} \leq \jump{\psi},
\end{equation}
where $\vec{\omega} = U'$ are the entropy variables and $\psi = \vec{\omega} \cdot \vec{f} - F$ is the entropy potential.
\end{theorem}
\begin{proof}
According to Definition~\ref{def_entropy_conservation}, entropy conservation/stability means
\begin{equation}
U'(\vec{u}_i) \cdot \partial_t \vec{u}_i = -\frac{\vec{\omega}(\vec{u}_i)}{\Delta x_i}\cdot  \Bigl ( \vec{f}^\mathrm{num}_{i+1/2} - \vec{f}^\mathrm{num}_{i-1/2} \bigg ) \leq -\frac{1}{\Delta x_i} \Bigl ( F^{\mathrm{num}}_{i+1/2} - F^{\mathrm{num}}_{i-1/2} \Bigr ),
\end{equation}
which simply boils down to
\begin{equation}
\vec{\omega}(\vec{u}_i) \cdot \bigg ( \vec{f}^\mathrm{num}_{i+1/2} - \vec{f}^\mathrm{num}_{i-1/2} \bigg ) \geq F^{\mathrm{num}}_{i+1/2} - F^{\mathrm{num}}_{i-1/2}.
\end{equation}
Thus, applying Lemma~\ref{algebraic_lemma} yields the desired result.
\end{proof}
\section{Nonconservative hyperbolic systems}
\label{sec:nonconservative_systems}
In this section, we first show that the entropy conservation condition for nonconservative systems~\eqref{CastroEC}, proposed by Castro et al.\ in~\cite[Theorem~2.1]{Castro2013}, is also necessary, generalizing to non-convex entropies.
Then, we propose four concrete semi-discretizations of the nonconservative terms $\vec{H}(\vec{u})\partial_x \vec{g}(\vec{u})$, composed of jumps and means of $\vec{H}$ and $\vec{g}$, and characterize their entropy-preserving properties.

In this section, we consider nonconservative hyperbolic systems of the form
\begin{equation}
\tag{\ref{nonconservative_system}}
    \partial_t \vec{u} + \partial_x \vec{f}(\vec{u}) + \vec{H}(\vec{u}) \partial_x \vec{g}(\vec{u}) = \vec{0},
\end{equation}
where $\vec{u}(x,t) \in Y \subset \mathbb{R}^n$ is the state vector, $\vec{f}\colon Y \to \mathbb{R}^n$ is the flux function, $\vec{H}\colon Y \to \mathbb{R}^{n \times m}$ is the nonconservative matrix, and $\vec{g}\colon Y \to \mathbb{R}^m$, paired with appropriate initial and boundary conditions.
We underline that here $n$ is the number of unknowns and $m$ accounts for the number of nonconservative terms.
As for the conservative case, we assume the existence of an entropy function $U\colon Y \to \mathbb{R}$ with associated entropy flux $F\colon Y \to \mathbb{R}$ satisfying
\begin{equation}\label{compatibility_condition_nc}
\forall \vec{u} \in Y\colon \quad U'(\vec{u}) \cdot \vec{f}'(\vec{u}) + U'(\vec{u}) \cdot \vec{H}(\vec{u}) \vec{g}'(\vec{u}) = F'(\vec{u}).
\end{equation}
Thus, smooth solutions of~\eqref{nonconservative_system} satisfy the additional conservation law
$\partial_t U(\vec{u}) + \partial_x F(\vec{u}) = 0$,
and weak solutions should satisfy the entropy inequality
$\partial_t U(\vec{u}) + \partial_x F(\vec{u}) \leq 0$.

\subsection{An algebraic characterization of the fluctuation form}

Next, we formulate the condition~\eqref{CastroEC} of Castro et al.\ \cite{Castro2013} as a purely algebraic condition, similarly to what we have just presented in Section~\ref{tadmor_non_convex_entropies} for the conservative case.
Note that we do not assume any further relationship between the functions $\vec{D}^\pm$, $\vec{\omega}$, and $F$.
\begin{lemma}{\label{algebraic_lemma_fluctuation}}
    Let $Y \subset \mathbb{R}^n$ be open, and $\vec{\omega}\colon Y \to \mathbb{R}^n$, $F\colon Y \to \mathbb{R}$, and $\vec{D}^{\pm}\colon Y \times Y \to \mathbb{R}^n$ satisfying
    \begin{equation}
    \forall \vec{u} \in Y\colon \quad \vec{D}^{\pm}(\vec{u},\vec{u}) = \vec{0}
    \end{equation}
    be given. Then,
    $\exists F^\mathrm{num}\colon Y \times Y \to \mathbb{R}$ satisfying $\forall \vec{u} \in Y\colon F^\mathrm{num}(\vec{u},\vec{u}) = F(\vec{u})$ and
    \begin{equation}
    \forall \vec{u}_-, \vec{u}_0, \vec{u}_+ \in Y\colon \quad \vec{\omega}(\vec{u}_0) \cdot \Bigl ( \vec{D}^+(\vec{u}_-, \vec{u}_0) + \vec{D}^{-}(\vec{u}_0, \vec{u}_+) \Bigr ) \geq F^\mathrm{num}(\vec{u}_0, \vec{u}_+) - F^\mathrm{num}(\vec{u}_-, \vec{u}_0)
    \label{sufficient_condition_fluctuation}
    \end{equation}
    if and only if
    \begin{equation}
    \forall \vec{u}_+, \vec{u}_- \in Y\colon \quad \vec{\omega}(\vec{u}_+) \cdot \vec{D}^+(\vec{u}_-, \vec{u}_+) + \vec{\omega}(\vec{u}_-) \cdot \vec{D}^-(\vec{u}_-, \vec{u}_+) \geq  F(\vec{u}_+) - F(\vec{u}_-).
    \label{necessary_condition_fluctuation}
    \end{equation}
    Moreover, if one of the conditions is satisfied with equality, then equality holds in the other condition as well, and $F^\mathrm{num}$ in~\eqref{sufficient_condition_fluctuation} is determined uniquely as
    \begin{equation}
    F^{\mathrm{num}}(\vec{u}_-, \vec{u}_+) = \frac{F(\vec{u}_+) + F(\vec{u}_-)}{2} + \frac{\vec{\omega}(\vec{u}_-)}{2} \cdot \vec{D}^-(\vec{u}_-, \vec{u}_+) - \frac{\vec{\omega}(\vec{u}_+)}{2} \cdot \vec{D}^+(\vec{u}_-,\vec{u}_+).
    \label{F_num_fluctuation}
    \end{equation}
\end{lemma}
\begin{proof}
We first show that~\eqref{sufficient_condition_fluctuation} implies~\eqref{necessary_condition_fluctuation}.
First, we choose $\vec{u}_0 = \vec{u}_-$ in~\eqref{sufficient_condition_fluctuation}, resulting in
\begin{equation}
\vec{\omega}_- \cdot \Bigl (\underbrace{\vec{D}^+(\vec{u}_-, \vec{u}_-)}_{= \vec{0}} + \vec{D}^-(\vec{u}_-, \vec{u}_+)\Bigr ) \geq F^{\mathrm{num}}(\vec{u}_-, \vec{u}_+) - F_-.
\label{eqminus_fluctuation}
\end{equation}
Choosing $\vec{u}_0 = \vec{u}_+$ in~\eqref{sufficient_condition_fluctuation} yields
\begin{equation}
\vec{\omega}_+ \cdot \Bigl (\vec{D}^+(\vec{u}_-, \vec{u}_+) + \underbrace{\vec{D}^-(\vec{u}_+, \vec{u}_+)}_{= \vec{0}} \Bigr ) \geq F_+ - F^{\mathrm{num}}(\vec{u}_-, \vec{u}_+).
\label{eqplus_fluctuation}
\end{equation}
Adding the two inequalities, we obtain~\eqref{necessary_condition_fluctuation}.

If equality holds in~\eqref{sufficient_condition_fluctuation}, it holds in~\eqref{necessary_condition_fluctuation} as well.
In this case, we can subtract~\eqref{eqplus_fluctuation} from~\eqref{eqminus_fluctuation} (both with equality) to obtain~\eqref{F_num_fluctuation}.

Finally, we show that~\eqref{necessary_condition_fluctuation} implies~\eqref{sufficient_condition_fluctuation}.
Similar to \cite{Castro2013}, we compute
\begin{equation}
    \begin{aligned}
        &\qquad
\vec{\omega}(\vec{u}_0) \cdot \Bigl ( \vec{D}^+(\vec{u}_-, \vec{u}_0) + \vec{D}^{-}(\vec{u}_0, \vec{u}_+) \Bigr )
\\
&=
\frac{\vec{\omega}_0}{2} \cdot \vec{D}^+(\vec{u}_-, \vec{u}_0) + \frac{\vec{\omega}_0}{2} \cdot \vec{D}^+(\vec{u}_-, \vec{u}_0)
+ \frac{\vec{\omega}_0}{2} \cdot \vec{D}^{-}(\vec{u}_0, \vec{u}_+) + \frac{\vec{\omega}_0}{2} \cdot \vec{D}^{-}(\vec{u}_0, \vec{u}_+)\\
&\stackrel{\eqref{necessary_condition_fluctuation}}{\ge}
\frac{\vec{\omega}_0}{2} \cdot \vec{D}^+(\vec{u}_-, \vec{u}_0) + \frac{F_0 - F_-}{2} - \frac{\vec{\omega}_-}{2} \cdot \vec{D}^-(\vec{u}_-,\vec{u}_0)
\\
&\quad
+ \frac{\vec{\omega}_0}{2} \cdot \vec{D}^{-}(\vec{u}_0, \vec{u}_+) + \frac{F_+ - F_0}{2} - \frac{\vec{\omega}_+}{2} \cdot \vec{D}^+(\vec{u}_0,\vec{u}_+)
\\
&=
\frac{\vec{\omega}_0}{2} \cdot \vec{D}^+(\vec{u}_-, \vec{u}_0) - \frac{F_0 + F_-}{2} - \frac{\vec{\omega}_-}{2} \cdot \vec{D}^-(\vec{u}_-,\vec{u}_0)
\\
&\quad
+ \frac{\vec{\omega}_0}{2} \cdot \vec{D}^{-}(\vec{u}_0, \vec{u}_+) + \frac{F_+ + F_0}{2} - \frac{\vec{\omega}_+}{2} \cdot \vec{D}^+(\vec{u}_0,\vec{u}_+)
\\
&= F^{\mathrm{num}}(\vec{u}_0, \vec{u}_+) - F^{\mathrm{num}}(\vec{u}_-, \vec{u}_0)
    \end{aligned}
\end{equation}
where we have used the definition of $F^{\mathrm{num}}$ in~\eqref{F_num_fluctuation} in the last step.
\end{proof}

Given this algebraic characterization, the necessary and sufficient condition for the FV scheme in fluctuation form follows immediately, cf.~\cite{Castro2013}.
\begin{theorem}{\label{fluctuation_theorem}}
The FV scheme in fluctuation form
\begin{equation}
\tag{\ref{FV_fluctuations}}
	\partial_t \vec{u}_i
    + \frac{1}{\Delta x_i} \Bigl (
        \vec{D}^+( \vec{u}_{i-1}, \vec{u}_{i} ) +
        \vec{D}^-( \vec{u}_{i}, \vec{u}_{i+1} )
    \Bigr ) = \vec{0}
\end{equation}
is entropy-conservative/stable for a given entropy pair $(U, F)$ (see Definition~\ref{def_entropy_conservation}) if and only if
\begin{equation}
\label{FluctuationDef}
 \vec{\omega}(\vec{u}_+) \cdot \vec{D}^+(\vec{u}_-, \vec{u}_+) + \vec{\omega}(\vec{u}_-) \cdot \vec{D}^-(\vec{u}_-, \vec{u}_+) \geq  \jump{F}
\end{equation}
where $\vec{\omega} = U'$ are the entropy variables.
\end{theorem}
\begin{proof}
The proof follows by applying Lemma~\ref{algebraic_lemma_fluctuation}.
\end{proof}

For us, the condition~\eqref{FluctuationDef} is too abstract to be useful for deriving entropy-preserving numerical schemes.
Thus, we focus on concrete forms of the fluctuations in the following.

\subsection{Specialized semi-discretizations for nonconservative systems}

The simplest discretization of the nonconservative term is the second-order central scheme
\begin{equation}
\vec{H}(\vec{u}) \partial_x \vec{g}(\vec{u})\Big |_{x = x_i} \approx \vec{H}_i \frac{\vec{g}_{i+1} - \vec{g}_{i-1}}{2 \Delta x_i}.
\label{FV_form0}
\end{equation}
This can be written more conveniently for the following analysis as
\begin{equation}
\vec{H}(\vec{u}) \partial_x \vec{g}(\vec{u})\Big |_{x = x_i} \approx \vec{H}_i \frac{\jump{\vec{g}}_{i+1/2} +\jump{\vec{g}}_{i-1/2}}{2 \Delta x_i},
\label{FV_form1}
\end{equation}
where $\jump{\vec{g}}_{i+1/2} = \vec{g}_{i+1} - \vec{g}_i$.
A second form can be obtained by writing the central discretization as
\begin{equation*}
    \vec{H}_i \frac{\vec{g}_{i+1} - \vec{g}_{i-1}}{2 \Delta x_i}
    =
    \vec{H}_i \frac{\mean{\vec{g}}_{i+1/2} - \mean{\vec{g}}_{i-1/2}}{ \Delta x_i}.
\end{equation*}
Generalizing the arithmetic mean to a numerical flux $g^\mathrm{num}$, we obtain the second form
\begin{equation}
\vec{H}(\vec{u}) \partial_x \vec{g}(\vec{u})\Big |_{x = x_i} \approx \vec{H}_i \frac{\vec{g}_{i+1/2}^{\mathrm{num}} - \vec{g}^{\mathrm{num}}_{i-1/2}}{\Delta x_i}.
\label{FV_form2}
\end{equation}

A third form can be obtained by interpreting the fraction in~\eqref{FV_form1} as an average of the discrete gradients at the left and right interface. Introducing a numerical flux $\vec{H}^{\mathrm{num}}$ at these interfaces, we obtain the third form
\begin{equation}
\vec{H}(\vec{u}) \partial_x \vec{g}(\vec{u})\Big |_{x = x_i} \approx \frac{\vec{H}^{\mathrm{num}}_{i+1/2} \jump{\vec{g}}_{i+1/2} + \vec{H}^{\mathrm{num}}_{i-1/2} \jump{\vec{g}}_{i-1/2}}{2 \Delta x_i}.
\label{FV_form3}
\end{equation}
Finally, we note that
\begin{equation*}
\frac{\vec{H}^{\mathrm{num}}_{i+1/2} \jump{\vec{g}}_{i+1/2} + \vec{H}^{\mathrm{num}}_{i-1/2} \jump{\vec{g}}_{i-1/2}}{2 \Delta x_i}
=
\frac{\vec{H}^{\mathrm{num}}_{i+1/2} \mean{\vec{g}}_{i+1/2} - \vec{H}^{\mathrm{num}}_{i-1/2} \mean{\vec{g}}_{i-1/2}}{\Delta x_i}
- \frac{\vec{H}^{\mathrm{num}}_{i+1/2} - \vec{H}^{\mathrm{num}}_{i-1/2}}{\Delta x_i} \vec{g}_i.
\end{equation*}
Thus, substituting the arithmetic mean $\mean{\vec{g}}$ with a generic two-point numerical flux and considering the single generic numerical flux of $\left (\vec{H} \vec{g} \right )^{\mathrm{num}}$, we obtain the fourth form
\begin{equation}
\vec{H}(\vec{u}) \partial_x \vec{g}(\vec{u})\Big |_{x = x_i}
\approx
\frac{\left (\vec{H} \vec{g} \right ) ^{\mathrm{num}}_{i+1/2} - \left (\vec{H} \vec{g} \right ) ^{\mathrm{num}}_{i-1/2}}{\Delta x_i}
- \frac{\vec{H}^{\mathrm{num}}_{i+1/2} - \vec{H}^{\mathrm{num}}_{i-1/2}}{\Delta x_i} \vec{g}_i.
\label{FV_form4}
\end{equation}
This last form is a generalization of the flux differencing form used by Waruszewski et al.~\cite{WARUSZEWSKI2022111507}, where they used the particular form $\vec{H}^{\mathrm{num}} \vec{g}^{\mathrm{num}} - \vec{H}^{\mathrm{num}} \vec{g}_i$ (at the right interface).
The form~\eqref{FV_form4} is particularly interesting because it allows a lot of flexibility; choosing $(\vec{H} \vec{g})^{\mathrm{num}}$ and $\vec{H}^{\mathrm{num}}$ as
\begin{equation}
(\vec{H} \vec{g})^{\mathrm{num}} = \prodmean{\vec{H} \cdot \vec{g}} := \frac{1}{2} \left ( \vec{H}_- \cdot \vec{g}_+ + \vec{H}_+ \cdot \vec{g}_- \right ),
\quad \vec{H}^{\mathrm{num}} = \mean{\vec{H}},
\end{equation}
we have that
\begin{equation}
\begin{aligned}
    \prodmean{\vec{H} \cdot \vec{g}}_{i+1/2} - \mean{\vec{H}}_{i+1/2} \vec{g}_i
    &=
    \frac{1}{2} \left ( \vec{H}_i \cdot \vec{g}_{i+1} + \vec{H}_{i+1} \cdot \vec{g}_i \right ) - \frac{\vec{H}_i + \vec{H}_{i+1}}{2} \cdot \vec{g}_i
    = \vec{H}_i \frac{\jump{\vec{g}}_{i+1/2}}{2}.
\end{aligned}
\end{equation}
Thus, we recover the first form~\eqref{FV_form1}.
Choosing
\begin{equation}
    (\vec{H} \vec{g})^{\mathrm{num}} = \vec{H}^\mathrm{num} \mean{\vec{g}},
    \quad \vec{H}^{\mathrm{num}} = \mean{\vec{H}},
\end{equation}
we obtain the third form~\eqref{FV_form3}.
\begin{remark}
Note that Definition~\ref{def_entropy_conservation} of entropy conservation/stability remains valid analogously for all the forms~\eqref{FV_form1}, \eqref{FV_form2}, \eqref{FV_form3}, and \eqref{FV_form4}.
\end{remark}

Next, we extend the algebraic characterization of Tadmor's condition discussed in Section~\ref{tadmor_non_convex_entropies} to the four semi-discretizations presented above.
For clarity, Table~\ref{table_formulation} summarizes the four forms of the semi-discretization of nonconservative product along with the associated entropy conservation/stability condition, as established by the lemmas presented below.

\begin{table}[htbp]
    \sisetup{
  output-exponent-marker=\text{e},
  round-mode=places,
  round-precision=2
}
\centering
  \caption{Semi-discretizations of nonconservative systems~\eqref{nonconservative_system} and associated entropy conservation/stability conditions.}
  \label{table_formulation}
  \centering
  \footnotesize
    \begin{tabular}{c ll}
      \toprule
      \multicolumn{1}{c}{Lem.}
      & \multicolumn{1}{c}{Semi-discretization}
      & \multicolumn{1}{c}{Entropy condition}\\
      \midrule
\ref{algebraic_lemma_nc_form1}
& $\vec{H}_i\frac{\jump{\vec{g}}_{i+1/2} + \jump{\vec{g}}_{i-1/2}}{2 \Delta x_i}$
& $\jump{\vec{\omega}} \cdot \vec{f}^\mathrm{num} - \mean{\vec{\omega} \cdot \vec{H}} \jump{\vec{g}} \leq \jump{\psi}$
\\
\ref{algebraic_lemma_nc_form2}
& $\vec{H}_i \frac{\vec{g}_{i+1/2}^{\mathrm{num}} - \vec{g}^{\mathrm{num}}_{i-1/2}}{\Delta x_i}$
& $\jump{\vec{\omega}} \cdot \vec{f}^\mathrm{num} +\jump{\vec{\omega} \cdot \vec{H}}\vec{g}^{\mathrm{num}} - \jump{\vec{\omega} \cdot \vec{H} \vec{g}} \leq \jump{\psi}$
\\
\ref{algebraic_lemma_nc_form3}
& $\frac{\vec{H}^{\mathrm{num}}_{i+1/2} \jump{\vec{g}}_{i+1/2} + \vec{H}^{\mathrm{num}}_{i-1/2} \jump{\vec{g}}_{i-1/2}}{2 \Delta x_i}$
& $\jump{\vec{\omega}} \cdot \vec{f}^\mathrm{num} - \mean{\vec{\omega}} \cdot \vec{H}^\mathrm{num} \jump{\vec{g}} \leq \jump{\psi}$
\\
\ref{algebraic_lemma_nc_form4}
& $\frac{\left (\vec{H} \vec{g} \right ) ^{\mathrm{num}}_{i+1/2} - \left (\vec{H} \vec{g} \right ) ^{\mathrm{num}}_{i-1/2}}{\Delta x_i} - \frac{\vec{H}^{\mathrm{num}}_{i+1/2} - \vec{H}^{\mathrm{num}}_{i-1/2}}{\Delta x_i} \vec{g}_i$
& $\jump{\vec{\omega}} \cdot \vec{f}^\mathrm{num} + \jump{\vec{\omega}} \cdot \left (\vec{H} \vec{g} \right )^{\mathrm{num}} - \Bigl (\vec{\omega}_+ \cdot \vec{H}^{\mathrm{num}} \vec{g}_+ - \vec{\omega}_- \cdot \vec{H}^{\mathrm{num}} \vec{g}_-\Bigr ) \leq \jump{\psi}$      \\
      \bottomrule
    \end{tabular}
\end{table}

\subsubsection{Algebraic characterization of the third form\texorpdfstring{~\eqref{FV_form3}}{}}
We consider the third form~\eqref{FV_form3}, since it is often flexible enough in practice.
We omit the proofs of the other lemmas, which follow essentially the same steps.

\begin{lemma}{\label{algebraic_lemma_nc_form3}(Third form)}
Let $Y \subset \mathbb{R}^n$ be open, and $\vec{\omega}\colon Y \to \mathbb{R}^n$, $\vec{f}\colon Y \to \mathbb{R}^n$, $F\colon Y \to \mathbb{R}$, $\vec{H}\colon Y \to \mathbb{R}^{n \times m}$,  $\vec{g}\colon Y \to \mathbb{R}^m$, $\vec{f}^{\mathrm{num}}\colon Y \times Y \to \mathbb{R}^{n}$ and $\vec{H}^{\mathrm{num}}\colon Y \times Y \to \mathbb{R}^{n \times m}$ satisfying
\begin{equation}
\forall \vec{u} \in Y\colon \quad \vec{f}^\mathrm{num}(\vec{u},\vec{u}) = \vec{f}(\vec{u}), \quad \vec{H}^{\mathrm{num}}(\vec{u},\vec{u}) = \vec{H}(\vec{u})
\end{equation}
be given.
Then, $\exists F^\mathrm{num}\colon Y \times Y \to \mathbb{R}$ satisfying  $\forall \vec{u} \in Y\colon F^\mathrm{num}(\vec{u},\vec{u}) = F(\vec{u})$ and
\begin{multline}
    \forall \vec{u}_-, \vec{u}_0, \vec{u}_+ \in Y\colon
    \quad
    \vec{\omega}(\vec{u}_0) \cdot \Bigl (
        \vec{f}^\mathrm{num}(\vec{u}_0, \vec{u}_+) - \vec{f}^\mathrm{num}(\vec{u}_-, \vec{u}_0)
        \\
        + \frac{1}{2}\vec{H}^\mathrm{num} (\vec{u}_0, \vec{u}_+) \Bigl ( \vec{g}(\vec{u}_+) - \vec{g}(\vec{u}_0) \Bigr )
        + \frac{1}{2}\vec{H}^\mathrm{num}(\vec{u}_-, \vec{u}_0) \Bigl (\vec{g}(\vec{u}_0) - \vec{g}(\vec{u}_-) \Bigr )
    \Bigr )
    \\
    \geq
    F^\mathrm{num}(\vec{u}_0, \vec{u}_+) - F^\mathrm{num}(\vec{u}_-, \vec{u}_0)
    \label{sufficient_condition_form3}
\end{multline}
if and only if
\begin{multline}
    \forall \vec{u}_-, \vec{u}_+ \in Y\colon
    \\
    \Bigl (\vec{\omega}(\vec{u}_+) - \vec{\omega}(\vec{u}_-) \Bigr )
    \cdot \vec{f}^\mathrm{num}(\vec{u}_-, \vec{u}_+)
    - \frac{\vec{\omega}(\vec{u}_+) + \vec{\omega}(\vec{u}_-)}{2}
    \cdot \vec{H}^{\mathrm{num}}(\vec{u}_-, \vec{u}_+)
    \Bigl ( \vec{g}(\vec{u}_+) - \vec{g}(\vec{u}_-) \Bigr )
    \\
    \leq
    \Bigl (\vec{\omega}(\vec{u}_+) \cdot \vec{f}(\vec{u}_+) - F(\vec{u}_+) \Bigr )
    - \Bigl (\vec{\omega}(\vec{u}_-) \cdot \vec{f}(\vec{u}_-) - F(\vec{u}_-) \Bigr ).
\label{necessary_condition_form3}
\end{multline}
Moreover, if one of the conditions is satisfied with equality, then equality holds in the other condition as well, and $F^{\mathrm{num}}$ in~\eqref{sufficient_condition_form3} is determined uniquely as
\begin{equation}
F^{\mathrm{num}} = \mean{F} + \mean{\vec{\omega}}\cdot \vec{f}^{\mathrm{num}}  -\mean{\vec{\omega} \cdot \vec{f}}   - \frac{1}{4} \jump{\vec{\omega}} \cdot \vec{H}^{\mathrm{num}}\jump{\vec{g}}.
\label{F_num_form3}
\end{equation}
\end{lemma}

\begin{proof}
The proof follows essentially the same steps as the proof of Lemma~\ref{algebraic_lemma}.
Thus, first we show that~\eqref{sufficient_condition_form3} implies~\eqref{necessary_condition_form3}.
We first choose $\vec{u}_0 = \vec{u}_-$ in~\eqref{sufficient_condition_form3}, resulting in
\begin{equation}
\vec{\omega}_- \cdot \Bigl (\vec{f}^{\mathrm{num}}(\vec{u}_-, \vec{u}_+) - \vec{f}_- \Bigr ) + \frac{\vec{\omega}_-}{2} \cdot \vec{H}^{\mathrm{num}}(\vec{u}_-, \vec{u}_+) \Bigl (\vec{g}_+ - \vec{g}_- \Bigr ) \geq F^{\mathrm{num}}(\vec{u}_-, \vec{u}_+) - F_-.
\label{eqminus_form3}
\end{equation}
Choosing $\vec{u}_0 = \vec{u}_+$ in~\eqref{sufficient_condition_form3} yields
\begin{equation}
\vec{\omega}_+ \cdot \Bigl (\vec{f}_+ - \vec{f}^{\mathrm{num}}(\vec{u}_-, \vec{u}_+) \Bigr )+ \frac{\vec{\omega}_+}{2} \cdot \vec{H}^{\mathrm{num}}(\vec{u}_-, \vec{u}_+) \Bigl (\vec{g}_+ - \vec{g}_- \Bigr ) \geq F_+ - F^{\mathrm{num}}(\vec{u}_-, \vec{u}_+) .
\label{eqplus_form3}
\end{equation}
Adding the two inequalities, we obtain
\begin{equation}
( \vec{\omega}_- - \vec{\omega}_+ ) \cdot \vec{f}^{\mathrm{num}}(\vec{u}_-, \vec{u}_+) + \frac{\vec{\omega}_- + \vec{\omega}_+}{2} \cdot \vec{H}^{\mathrm{num}}(\vec{u}_-, \vec{u}_+) \big (\vec{g}_+ - \vec{g}_- \big )  \geq (\vec{\omega}_- \cdot \vec{f}_- - F_-) - ( \vec{\omega}_+ \cdot \vec{f}_+ -F_+)
\end{equation}
which is equivalent to~\eqref{necessary_condition_form3}.

If equality holds in~\eqref{sufficient_condition_form3}, it holds in~\eqref{necessary_condition_form3} as well.
In this case, we can subtract~\eqref{eqplus_form3} from~\eqref{eqminus_form3} (both with equality) to obtain~\eqref{F_num_form3}.

Finally, we show that~\eqref{necessary_condition_form3} implies~\eqref{sufficient_condition_form3}. Abbreviating again $\psi = \vec{\omega} \cdot \vec{f} - F$,~\eqref{necessary_condition_form3} is equivalent to
\begin{equation}
  -(\vec{\omega}_+ - \vec{\omega}_-  ) \cdot \vec{f}^\mathrm{num}(\vec{u}_-, \vec{u}_+) + \frac{\vec{\omega}_+ + \vec{\omega}_-}{2} \cdot \vec{H}^{\mathrm{num}}(\vec{u}_-, \vec{u}_+) \Bigl ( \vec{g}_+ - \vec{g}_- \Bigr )\geq -(\psi_+ - \psi_-).
  \label{necessary_condition_psi_form3}
\end{equation}
Analogously to Lemma~\ref{algebraic_lemma}, we compute
\begin{equation}
\begin{aligned}
 &\qquad
 \vec{\omega}(\vec{u}_0) \cdot \Bigl ( \vec{f}^{\mathrm{num}}(\vec{u}_0, \vec{u}_+) - \vec{f}^{\mathrm{num}}(\vec{u}_-, \vec{u}_0) + \frac{1}{2}\vec{H}^\mathrm{num} (\vec{u}_0, \vec{u}_+) ( \vec{g}_+ - \vec{g}_0 ) + \frac{1}{2}\vec{H}^\mathrm{num}(\vec{u}_-, \vec{u}_0)  (\vec{g}_0 - \vec{g}_-  )  \Bigr )
 \\
 &= \frac{\vec{\omega}_0 + \vec{\omega}_+}{2} \cdot \vec{f}^{\mathrm{num}}(\vec{u}_0,\vec{u}_+) - \frac{\vec{\omega}_+ - \vec{\omega}_0}{2} \cdot \vec{f}^{\mathrm{num}} (\vec{u}_0,\vec{u}_+)
 \\
 &\quad - \frac{\vec{\omega}_- + \vec{\omega}_0}{2} \cdot \vec{f}^\mathrm{num}(\vec{u}_-, \vec{u}_0)
    - \frac{\vec{\omega}_0 - \vec{\omega}_-}{2} \cdot \vec{f}^\mathrm{num}(\vec{u}_-, \vec{u}_0)
    \\
    &\quad + \frac{\vec{\omega}_0 + \vec{\omega}_+}{4} \cdot \vec{H}^{\mathrm{num}}(\vec{u}_0, \vec{u}_+) (\vec{g}_+ - \vec{g}_0) - \frac{\vec{\omega}_+ - \vec{\omega}_0}{4} \vec{H}^{\mathrm{num}}(\vec{u}_0, \vec{u}_+)(\vec{g}_+ - \vec{g}_0)
    \\
    &\quad +\frac{\vec{\omega}_- + \vec{\omega}_0}{4} \cdot \vec{H}^{\mathrm{num}}(\vec{u}_-, \vec{u}_0) (\vec{g}_0 - \vec{g}_-) + \frac{\vec{\omega}_0 - \vec{\omega}_-}{4} \vec{H}^{\mathrm{num}}(\vec{u}_-, \vec{u}_0)(\vec{g}_0 - \vec{g}_-)
    \\
    &\stackrel{\eqref{necessary_condition_psi_form3}}{\ge}
    \frac{\vec{\omega}_0 + \vec{\omega}_+}{2} \cdot \vec{f}^\mathrm{num}(\vec{u}_0, \vec{u}_+)
    - \frac{\psi_+ - \psi_0}{2}
    - \frac{\vec{\omega}_- + \vec{\omega}_0}{2} \cdot \vec{f}^\mathrm{num}(\vec{u}_-, \vec{u}_0)
    - \frac{\psi_0 - \psi_-}{2}
    \\
     &\quad - \frac{\vec{\omega}_+ - \vec{\omega}_0}{4}\cdot \vec{H}^{\mathrm{num}}(\vec{u}_0, \vec{u}_+)(\vec{g}_+ - \vec{g}_0) + \frac{\vec{\omega}_0 - \vec{\omega}_-}{4} \cdot \vec{H}^{\mathrm{num}}(\vec{u}_-, \vec{u}_0)(\vec{g}_0 - \vec{g}_-)
     \\
     &= F^{\mathrm{num}}(\vec{u}_0, \vec{u}_+) - F^{\mathrm{num}}(\vec{u}_-, \vec{u}_0),
 \end{aligned}
\end{equation}
where we have used the definition of $F^{\mathrm{num}}$ in~\eqref{F_num_form3} in the last step.
\end{proof}

\subsubsection{Algebraic characterization of the first form\texorpdfstring{~\eqref{FV_form1}}{}}

\begin{lemma}{\label{algebraic_lemma_nc_form1}(First form)}
Let $Y \subset \mathbb{R}^n$ be open, and $\vec{\omega}\colon Y \to \mathbb{R}^n$, $\vec{f}\colon Y \to \mathbb{R}^n$, $F\colon Y \to \mathbb{R}$, $\vec{H}\colon Y \to \mathbb{R}^{n \times m}$, $\vec{g}\colon Y \to \mathbb{R}^m$ and $\vec{f}^{\mathrm{num}}\colon Y \times Y \to \mathbb{R}^n$  satisfying
\begin{equation}
\forall \vec{u} \in Y\colon \quad \vec{f}^\mathrm{num}(\vec{u},\vec{u}) = \vec{f}(\vec{u}),
\end{equation}
be given. Then,
$\exists F^\mathrm{num}\colon Y \times Y \to \mathbb{R}$ satisfying $\forall \vec{u}\in Y\colon F^\mathrm{num}(\vec{u},\vec{u}) = F(\vec{u})$ and
\begin{multline}
    \forall \vec{u}_-, \vec{u}_0, \vec{u}_+ \in Y\colon \quad \vec{\omega}(\vec{u}_0) \cdot \Bigl (  \vec{f}^\mathrm{num}(\vec{u}_0, \vec{u}_+) - \vec{f}^\mathrm{num}(\vec{u}_-, \vec{u}_0) \\+ \frac{\vec{H} (\vec{u}_0)  (\vec{g}(\vec{u}_+) - \vec{g}(\vec{u}_0) ) + \vec{H}(\vec{u}_0)  (\vec{g}(  \vec{u}_0) - \vec{g}(\vec{u}_-)  )}{2} \Bigr ) \geq F^\mathrm{num}(\vec{u}_0, \vec{u}_+) - F^\mathrm{num}(\vec{u}_-, \vec{u}_0)
    \label{sufficient_condition_form1}
\end{multline}
if and only if
\begin{multline}
    \forall \vec{u}_+, \vec{u}_- \in Y\colon
    \\
    \Bigl (\vec{\omega}(\vec{u}_+) - \vec{\omega}(\vec{u}_-) \Bigr )
    \cdot \vec{f}^\mathrm{num}(\vec{u}_-, \vec{u}_+)
    - \frac{\vec{\omega}(\vec{u}_+) \cdot \vec{H}(\vec{u}_+) + \vec{\omega}(\vec{u}_-) \cdot \vec{H}(\vec{u}_-)}{2}
    \cdot \Bigl ( \vec{g}(\vec{u}_+) - \vec{g}(\vec{u}_-) \Bigr )
    \\
    \leq  \Bigl ( \vec{\omega}(\vec{u}_+) \cdot \vec{f}(\vec{u}_+) - F(\vec{u}_+) \Bigr ) - \Bigl (\vec{\omega}(\vec{u}_-) \cdot \vec{f}(\vec{u}_-) - F(\vec{u}_-) \Bigr ).
\label{necessary_condition_form1}
\end{multline}
Moreover, if one of the conditions is satisfied with equality, then equality holds in the other condition as well, and $F^{\mathrm{num}}$ in~\eqref{sufficient_condition_form1} is determined uniquely as
\begin{equation}
F^{\mathrm{num}} = \mean{F} + \mean{\vec{\omega}} \cdot \vec{f}^{\mathrm{num}}-\mean{\vec{\omega} \cdot \vec{f}}  - \frac{1}{4} \jump{\vec{\omega} \cdot \vec{H}} \jump{\vec{g}}.
\label{F_num_form1}
\end{equation}
\end{lemma}

\subsubsection{Algebraic characterization of the second form\texorpdfstring{~\eqref{FV_form2}}{}}

\begin{lemma}{\label{algebraic_lemma_nc_form2}(Second form)}
Let $Y \subset \mathbb{R}^n$ be open, and $\vec{\omega}\colon Y \to \mathbb{R}^n$, $\vec{f}\colon Y \to \mathbb{R}^n$, $F\colon Y \to \mathbb{R}$, $\vec{H}\colon Y \to \mathbb{R}^{n \times m}$, $\vec{g}\colon Y \to \mathbb{R}^m$, $\vec{g}^{\mathrm{num}}\colon Y \times Y \to \mathbb{R}^m$ and $\vec{f}^{\mathrm{num}}\colon Y \times Y \to \mathbb{R}^n$  satisfying
\begin{equation}
\forall \vec{u} \in Y\colon \quad \vec{f}^\mathrm{num}(\vec{u},\vec{u}) = \vec{f}(\vec{u}), \quad \vec{g}^\mathrm{num}(\vec{u},\vec{u}) = \vec{g}(\vec{u}),
\end{equation}
be given. Then,
$\exists F^\mathrm{num}\colon Y \times Y \to \mathbb{R}$ satisfying $\forall \vec{u}\in Y\colon F^\mathrm{num}(\vec{u},\vec{u}) = F(\vec{u})$ and
\begin{multline}
    \forall \vec{u}_-, \vec{u}_0, \vec{u}_+ \in Y\colon \quad \vec{\omega}(\vec{u}_0) \cdot \Bigl (  \vec{f}^\mathrm{num}(\vec{u}_0, \vec{u}_+) - \vec{f}^\mathrm{num}(\vec{u}_-, \vec{u}_0) \\+ \vec{H} (\vec{u}_0)  \vec{g}^\mathrm{num}(\vec{u}_0, \vec{u}_+) - \vec{H}(\vec{u}_0)\vec{g}^\mathrm{num}(\vec{u}_-, \vec{u}_0)  \Bigr ) \geq F^\mathrm{num}(\vec{u}_0, \vec{u}_+) - F^\mathrm{num}(\vec{u}_-, \vec{u}_0)
    \label{sufficient_condition_form2}
\end{multline}
if and only if
\begin{multline}
    \forall \vec{u}_+, \vec{u}_- \in Y\colon
    \\
    \Bigl (\vec{\omega}(\vec{u}_+) - \vec{\omega}(\vec{u}_-) \Bigr )
    \cdot \vec{f}^\mathrm{num}(\vec{u}_-, \vec{u}_+)
    + \Bigl (\vec{\omega}(\vec{u}_+) \cdot \vec{H}(\vec{u}_+) - \vec{\omega}(\vec{u}_-) \cdot \vec{H}(\vec{u}_-)  \Bigr ) \vec{g}^{\mathrm{num}}(\vec{u}_-, \vec{u}_+)
    \\
    - \Bigl ( \vec{\omega}(\vec{u}_+) \cdot \vec{H}(\vec{u}_+) \vec{g}_+ - \vec{\omega}(\vec{u}_-) \cdot \vec{H}(\vec{u}_-) \vec{g}_- \Bigr )\leq  \Bigl ( \vec{\omega}(\vec{u}_+) \cdot \vec{f}(\vec{u}_+) - F(\vec{u}_+) \Bigr ) - \Bigl (\vec{\omega}(\vec{u}_-) \cdot \vec{f}(\vec{u}_-) - F(\vec{u}_-) \Bigr ).
\label{necessary_condition_form2}
\end{multline}
Moreover, if one of the conditions is satisfied with equality, then equality holds in the other condition as well, and $F^{\mathrm{num}}$ in~\eqref{sufficient_condition_form2} is determined uniquely as
\begin{equation}
F^{\mathrm{num}} = \mean{F} + \mean{\vec{\omega}} \cdot \vec{f}^{\mathrm{num}}-\mean{\vec{\omega} \cdot \vec{f}}  - \mean{\vec{\omega} \cdot \vec{H} \vec{g}} + \mean{\vec{\omega} \cdot \vec{H}} \vec{g}^{\mathrm{num}}.
\label{F_num_form2}
\end{equation}
\end{lemma}

\subsubsection{Algebraic characterization of the fourth form\texorpdfstring{~\eqref{FV_form4}}{}}

\begin{lemma}{\label{algebraic_lemma_nc_form4}(Fourth form)}
Let $Y \subset \mathbb{R}^n$ be open, and $\vec{\omega}\colon Y \to \mathbb{R}^n$, $\vec{f}\colon Y \to \mathbb{R}^n$, $F\colon Y \to \mathbb{R}$, $\vec{H}\colon Y \to \mathbb{R}^{n \times m}$, $\vec{g}\colon Y \to \mathbb{R}^m$, $\vec{f}^{\mathrm{num}}\colon Y \times Y \to \mathbb{R}^n$, $\vec{H}^{\mathrm{num}}\colon Y \times Y \to \mathbb{R}^{n \times m}$ and $(\vec{H} \vec{g})^{\mathrm{num}}\colon Y \times Y \to \mathbb{R}^n$ satisfying
\begin{equation}
\forall \vec{u} \in Y\colon \quad \vec{f}^\mathrm{num}(\vec{u},\vec{u}) = \vec{f}(\vec{u}), \quad \vec{H}^\mathrm{num}(\vec{u},\vec{u}) = \vec{H}(\vec{u}), \quad (\vec{H} \vec{g})^\mathrm{num}(\vec{u},\vec{u}) = \vec{H}(\vec{u}) \vec{g}(\vec{u}),
\end{equation}
be given. Then,
$\exists F^\mathrm{num}\colon Y \times Y \to \mathbb{R}$ satisfying $\forall \vec{u}\in Y\colon F^\mathrm{num}(\vec{u},\vec{u}) = F(\vec{u})$ and
\begin{multline}
    \forall \vec{u}_-, \vec{u}_0, \vec{u}_+ \in Y\colon \quad \vec{\omega}(\vec{u}_0) \cdot \Bigl (  \vec{f}^\mathrm{num}(\vec{u}_0, \vec{u}_+) - \vec{f}^\mathrm{num}(\vec{u}_-, \vec{u}_0) + \left ( \vec{H}\vec{g}\right )^\mathrm{num}(\vec{u}_0, \vec{u}_+) - \left (\vec{H}\vec{g}\right )^\mathrm{num}(\vec{u}_-, \vec{u}_0) \\- \vec{H}^{\mathrm{num}}(\vec{u}_0, \vec{u}_+)\vec{g}(\vec{u}_0) + \vec{H}^\mathrm{num}(\vec{u}_-,\vec{u}_0)\vec{g}(\vec{u}_0)  \Bigr ) \geq F^\mathrm{num}(\vec{u}_0, \vec{u}_+) - F^\mathrm{num}(\vec{u}_-, \vec{u}_0)
    \label{sufficient_condition_form4}
\end{multline}
if and only if
\begin{multline}
    \forall \vec{u}_+, \vec{u}_- \in Y\colon
    \\
    \Bigl (\vec{\omega}(\vec{u}_+) - \vec{\omega}(\vec{u}_-) \Bigr )
    \cdot \vec{f}^\mathrm{num}(\vec{u}_-, \vec{u}_+)
    + \Bigl (\vec{\omega}(\vec{u}_+) - \vec{\omega}(\vec{u}_-)  \Bigr ) \cdot \left ( \vec{H}\vec{g} \right )^{\mathrm{num}}(\vec{u}_-, \vec{u}_+)
    - \vec{\omega}(\vec{u}_+) \cdot \vec{H}^{\mathrm{num}}(\vec{u}_-, \vec{u}_+)  \vec{g}(\vec{u}_+)
    \\
    + \vec{\omega}(\vec{u}_-) \cdot \vec{H}^{\mathrm{num}}(\vec{u}_-, \vec{u}_+) \vec{g}(\vec{u}_-) \leq  \Bigl ( \vec{\omega}(\vec{u}_+) \cdot \vec{f}(\vec{u}_+) - F(\vec{u}_+) \Bigr ) - \Bigl (\vec{\omega}(\vec{u}_-) \cdot \vec{f}(\vec{u}_-) - F(\vec{u}_-) \Bigr ).
\label{necessary_condition_form4}
\end{multline}
Moreover, if one of the conditions is satisfied with equality, then equality holds in the other condition as well, and $F^{\mathrm{num}}$ in~\eqref{sufficient_condition_form4} is determined uniquely as
\begin{equation}
F^{\mathrm{num}} = \mean{F} + \mean{\vec{\omega}} \cdot \vec{f}^{\mathrm{num}}-\mean{\vec{\omega} \cdot \vec{f}}  +\mean{\vec{\omega}} \cdot \left ( \vec{H}\vec{g}\right )^{\mathrm{num}} - \frac{\vec{\omega}_+}{2} \cdot \vec{H}^{\mathrm{num}} \vec{g}_+ - \frac{\vec{\omega}_-}{2} \cdot \vec{H}^{\mathrm{num}} \vec{g}_-.
\label{F_num_form4}
\end{equation}
\end{lemma}

\subsubsection{Entropy conservation/stability condition for nonconservative systems}
Given these results, we can state four conditions extending Tadmor's condition in Theorem~\ref{tadmor_theorem} to nonconservative hyperbolic systems (summarized in Table~\ref{table_formulation}).

\begin{theorem}{(First form)}
The FV scheme
\begin{equation}
    \tag{\ref{FV_form1}}
\partial_t \vec{u}_i + \frac{1}{\Delta x_i} \left  ( \vec{f}^{\mathrm{num}}_{i+1/2} - \vec{f}^{\mathrm{num}}_{i-1/2}\right ) + \vec{H}_i \frac{ \jump{\vec{g}}_{i+1/2}+ \jump{\vec{g}}_{i-1/2}}{2 \Delta x_i}  = \vec{0}
\end{equation}
is entropy-conservative/stable for a given entropy pair $(U, F)$ (see Definition~\ref{def_entropy_conservation}) if and only if
\begin{equation}
\label{TadmorDef_form1}
\jump{\vec{\omega}} \cdot \vec{f}^\mathrm{num} -\mean{\vec{\omega} \cdot \vec{H}}\jump{\vec{g}} \leq \jump{\psi},
\end{equation}
where $\vec{\omega} = U'$ are the entropy variables and $\psi = \vec{\omega} \cdot \vec{f} - F$ is the entropy potential.
\end{theorem}
\begin{proof}
The proof follows the same steps as in Theorem~\ref{tadmor_theorem} by applying Lemma~\ref{algebraic_lemma_nc_form1}.
\end{proof}

Please note that the condition~\eqref{TadmorDef_form1} for an entropy-conservative/stable method is equivalent to the condition derived by Manzanero~\cite[Eq.~(3.106)]{Manzanero2020AHD}.

\begin{theorem}{(Second form)}
The FV scheme
\begin{equation}
    \tag{\ref{FV_form2}}
\partial_t \vec{u}_i + \frac{1}{\Delta x_i} \left  ( \vec{f}^{\mathrm{num}}_{i+1/2} - \vec{f}^{\mathrm{num}}_{i-1/2}\right ) + \vec{H}_i \frac{ \vec{g}^{\mathrm{num}}_{i+1/2}- \vec{g}^{\mathrm{num}}_{i-1/2}}{\Delta x_i}  = \vec{0}
\end{equation}
is entropy-conservative/stable for a given entropy pair $(U, F)$ (see Definition~\ref{def_entropy_conservation}) if and only if
\begin{equation}
\label{TadmorDef_form2}
	\jump{\vec{\omega}} \cdot \vec{f}^\mathrm{num} +\jump{\vec{\omega} \cdot \vec{H}}\vec{g}^{\mathrm{num}} - \jump{\vec{\omega} \cdot \vec{H} \vec{g}} \leq \jump{\psi},
\end{equation}
where $\vec{\omega} = U'$ are the entropy variables and $\psi = \vec{\omega} \cdot \vec{f} - F$ is the entropy potential.
\end{theorem}
\begin{proof}
The proof follows the same steps as in Theorem~\ref{tadmor_theorem} by applying Lemma~\ref{algebraic_lemma_nc_form2}.
\end{proof}

\begin{theorem}{(Third form)}
    The FV scheme
    \begin{equation}
    \tag{\ref{FV_form3}}
\partial_t \vec{u}_i + \frac{1}{\Delta x_i} \left  ( \vec{f}^{\mathrm{num}}_{i+1/2} - \vec{f}^{\mathrm{num}}_{i-1/2}\right ) + \frac{\vec{H}^{\mathrm{num}}_{i+1/2} \jump{\vec{g}}_{i+1/2} + \vec{H}^{\mathrm{num}}_{i-1/2}  \jump{\vec{g}}_{i-1/2}}{2 \Delta x_i}  = \vec{0},
\end{equation}
is entropy-conservative/stable for a given entropy pair $(U, F)$ (see Definition~\ref{def_entropy_conservation}) if and only if
\begin{equation}
\label{TadmorDef_form3}
	\jump{\vec{\omega}} \cdot \vec{f}^\mathrm{num} -\mean{\vec{\omega}} \cdot \vec{H}^{\mathrm{num}}\jump{\vec{g}} \leq \jump{\psi},
\end{equation}
where $\vec{\omega} = U'$ are the entropy variables and $\psi = \vec{\omega} \cdot \vec{f} - F$ is the entropy potential.
\end{theorem}
\begin{proof}
The proof follows the same steps as in Theorem~\ref{tadmor_theorem} by applying Lemma~\ref{algebraic_lemma_nc_form3}.
\end{proof}

\begin{theorem}{(Fourth form)}
    The FV scheme
    \begin{equation}
    \tag{\ref{FV_form4}}
\partial_t \vec{u}_i + \frac{1}{\Delta x_i} \left  ( \vec{f}^{\mathrm{num}}_{i+1/2} - \vec{f}^{\mathrm{num}}_{i-1/2}\right ) + \frac{\left (\vec{H} \vec{g} \right ) ^{\mathrm{num}}_{i+1/2} - \left (\vec{H} \vec{g} \right ) ^{\mathrm{num}}_{i-1/2}}{\Delta x_i} -\vec{g}_i \frac{\vec{H}^{\mathrm{num}}_{i+1/2} - \vec{H}^{\mathrm{num}}_{i-1/2}}{\Delta x_i}  = \vec{0},
\end{equation}
is entropy-conservative/stable for a given entropy pair $(U, F)$ (see Definition~\ref{def_entropy_conservation}) if and only if
\begin{equation}
\label{TadmorDef_form4}
	\jump{\vec{\omega}} \cdot \vec{f}^\mathrm{num} + \jump{\vec{\omega}} \cdot \left (\vec{H} \vec{g} \right )^{\mathrm{num}} - \Bigl (\vec{\omega}_+ \cdot \vec{H}^{\mathrm{num}} \vec{g}_+ - \vec{\omega}_- \cdot \vec{H}^{\mathrm{num}} \vec{g}_-\Bigr ) \leq \jump{\psi},
\end{equation}
where $\vec{\omega} = U'$ are the entropy variables and $\psi = \vec{\omega} \cdot \vec{f} - F$ is the entropy potential.
\end{theorem}
\begin{proof}
The proof follows the same steps as in Theorem~\ref{tadmor_theorem} by applying Lemma~\ref{algebraic_lemma_nc_form4}.
\end{proof}

\subsection{A generalized semi-discretization for nonconservative systems}

It can be useful to consider linear combinations of different forms of the semi-discretizations introduced above, e.g., to obtain generalized split forms.
Here, we focus on a linear combination of the first and third form, i.e.,
\begin{equation}
    \partial_t \vec{u}_i
    + \frac{1}{\Delta x_i} \left  ( \vec{f}^{\mathrm{num}}_{i+1/2} - \vec{f}^{\mathrm{num}}_{i-1/2}\right )
    + \alpha \frac{\vec{H}^{\mathrm{num}}_{i+1/2} \jump{\vec{g}}_{i+1/2} + \vec{H}^{\mathrm{num}}_{i-1/2}  \jump{\vec{g}}_{i-1/2}}{2 \Delta x_i}
    + (1 -\alpha) \vec{H}_i \frac{\jump{\vec{g}}_{i+1/2} + \jump{\vec{g}}_{i-1/2}}{ 2 \Delta x_i} = \vec{0},
\label{FV_convex}
\end{equation}
where $\alpha \in \mathbb{R}$.
In this case, we have the following result.
\begin{theorem}\label{th:EC_linear_combination}
The FV scheme~\eqref{FV_convex} is entropy-conservative/stable for a given entropy pair $(U, F)$ (see Definition~\ref{def_entropy_conservation}) and $\alpha \in \mathbb{R}$ if and only if
\begin{equation}
\label{TadmorDef_convex}
	\jump{\vec{\omega}} \cdot \vec{f}^\mathrm{num}
    - \alpha \mean{\vec{\omega}}\cdot\vec{H}^{\mathrm{num}}\jump{\vec{g}}
    - (1-\alpha) \mean{\vec{\omega} \cdot\vec{H}} \jump{\vec{g}}
    \leq \jump{\psi},
\end{equation}
where $\vec{\omega} = U'$ are the entropy variables and $\psi = \vec{\omega} \cdot \vec{f} - F$ is the entropy potential.
\end{theorem}
\begin{proof}
    The proof follows by applying Lemma~\ref{algebraic_lemma_nc_form1} and~\ref{algebraic_lemma_nc_form3}.
\end{proof}
\begin{remark}
    In general, the parameter $\alpha$ represents a degree of freedom when constructing numerical fluxes.
    In contrast to the case of hyperbolic conservation laws with (strictly) convex entropy, there may be no solution to the characterizing condition~\eqref{TadmorDef_convex} for a given $\alpha \in \mathbb{R}$, or even for any $\alpha$.
    If there is a solution, it may exist for several choices of $\alpha$.
    We will explore this in some examples in Section~\ref{sec:numerical_fluxes}.
\end{remark}

\begin{remark}
    Although the condition can be applied by treating all nonconservative terms in a uniform manner, it is also possible to discretize different nonconservative terms using any of the four options, or a combination of them.
    Thus, each nonconservative term can be associated to some parameter $\alpha_j \in \mathbb{R}$. An example is given in Section~\ref{sec:numerical_fluxes}.
\end{remark}

\begin{remark}
    When $\alpha = 0$, the condition~\eqref{TadmorDef_convex} becomes fully specified for the nonconservative terms; the only remaining degree of freedom is the choice of the numerical flux $\vec{f}^\mathrm{num}$ associated with the conservative part of the system, when present.
    In contrast, the freedom obtained by considering $\vec{H}^\mathrm{num}$ in~\eqref{TadmorDef_form2}, or the more general one in~\eqref{TadmorDef_form4} allows constructing well-balanced methods for nonlinear steady states \cite{WARUSZEWSKI2022111507,artiano2025structurepreservinghighordermethodscompressible}.
\end{remark}

\section{Constructing entropy-conservative/stable fluxes}
\label{sec:numerical_fluxes}
In this section, we provide several examples to derive entropy-conservative/stable numerical fluxes for nonconservative hyperbolic systems of equations.
We demonstrate how existing fluxes can be recovered within our framework, and how novel fluxes can be constructed systematically.

\subsection{Variable-coefficient advection equation}
\label{sec:fluxes_variable_coefficient_advection}
We first consider the variable-coefficient linear advection equation
\begin{equation}
\partial_t u(t,x) + a(x) \partial_x u(t,x) = 0,
\label{eq_adv}
\end{equation}
with positive speed $a(x) \ge a_0 > 0$ and compatible initial and boundary conditions.
This equation can be written in the form \eqref{nonconservative_system} with $f = 0$, $H = a$, and $g = u$.
We consider the entropy pair~\cite{ranocha2018generalised,manzanero2017insights}
\begin{equation}
U = \frac{u^2}{a}, \quad F = u^2.
\end{equation}
Thus, the entropy variable $\omega$ and the entropy potential $\psi$ are
\begin{equation}
\omega = 2\frac{u}{a}, \quad \psi = - F = -u^2.
\end{equation}
The entropy conservation condition~\eqref{TadmorDef_convex} with $\alpha = 0$ leads to
\begin{equation}
     -\mean{\omega H}\jump{g}
    =
    -2 \mean{u}\jump{u}
    =
    - \jump{u^2}
    =
    \jump{\psi} = -2 \jump{u}\mean{u},
\end{equation}
which is always satisfied.
This corresponds effectively to the construction of energy-conservative methods in~\cite{ranocha2018generalised}.
The resulting semi-discretization is
\begin{equation}
\partial_t u_i + a_i \frac{\jump{u}_{i+1/2} + \jump{u}_{i-1/2}}{2 \Delta x_i} = \partial_t u_i + a_i \frac{u_{i+1} - u_{i-1}}{2 \Delta x_i} = 0.
\label{eq_adv_scheme}
\end{equation}

\subsection{Coupled Burgers' equation}
Following Castro et al.~\cite{Castro2013}, we consider the
nonconservative hyperbolic system
\begin{equation}
\label{eq:coupled_burgers}
\begin{aligned}
    & \partial_t u + u \partial_x (u + v) = 0,\\
    & \partial_t v + v \partial_x (u + v) = 0.
\end{aligned}
\end{equation}
This system is of the form~\eqref{nonconservative_system} with
\begin{equation}
    \vec{u} = \begin{pmatrix} u \\ v \end{pmatrix},\quad
    \vec{f} = \vec{0}, \quad
    \vec{H} = \begin{pmatrix} u & 0 \\ 0 & v \end{pmatrix}, \quad
    \vec{g} = \begin{pmatrix} u+v \\ u+v \end{pmatrix}.
\end{equation}
By summing the two equations, we observe that $q := u + v$ satisfies Burgers' equation~\cite{Castro2013}
\begin{equation}\label{eq_burgers}
\partial_t q + \partial_x\!\left(\frac{q^2}{2}\right) = 0.
\end{equation}
Therefore, the coupled Burgers' equation \eqref{eq:coupled_burgers} admits the entropy pair~\cite{Castro2013}
\begin{equation}
U = q^2/2, \quad F = q^3/3.
\end{equation}
Hence, the entropy variables and the entropy potential are
\begin{equation}
\vec{\omega} = (q, q),\quad \psi = - q^3/3.
\end{equation}
The EC condition~\eqref{TadmorDef_convex} is equivalent to
\begin{equation}
    \label{eq:EC_condition_coupled_burgers}
\begin{aligned}
    &\qquad
    \alpha \mean{\vec{\omega}} \cdot \vec{H}^{\mathrm{num}} \jump{\vec{g}}
    + (1-\alpha) \mean{\vec{\omega} \cdot \vec{H}} \jump{\vec{g}}
    \\
    &=
    \alpha \mean{u+v} \left( u^{\mathrm{num}} + v^{\mathrm{num}} \right) \jump{u+v}
    + (1-\alpha)  \mean{(u+v) u} \jump{u+v}
    + (1-\alpha)  \mean{(u+v) v} \jump{u+v}
    \\
    &=
    \frac{\jump{(u+v)^3}}{3}
    =
    -\jump{\psi}.
\end{aligned}
\end{equation}
The right-hand side of~\eqref{eq:EC_condition_coupled_burgers} can be expanded as
\begin{equation}
\frac{\jump{(u+v)^3}}{3} = \frac{\jump{u+v}}{3} \left((u_+ + v_+)^2 + (u_+ + v_+)(u_- + v_-) + (u_-+v_-)^2\right).
\end{equation}
Thus, the EC condition becomes
\begin{multline}
    \alpha \mean{u+v} \left( u^{\mathrm{num}} + v^{\mathrm{num}} \right) + (1-\alpha) \mean{(u+v)^2}
    \\
    = \frac{1}{3}\left((u_+ + v_+)^2 + (u_+ + v_+)(u_- + v_-) + (u_-+v_-)^2\right).
\end{multline}
Choosing arithmetic means $u^{\mathrm{num}} = \mean{u}$ and $v^{\mathrm{num}} = \mean{v}$, we obtain
\begin{equation}
    \alpha \mean{u+v}^2 + (1-\alpha) \mean{(u+v)^2}
    = \frac{1}{3}\left((u_+ + v_+)^2 + (u_+ + v_+)(u_- + v_-) + (u_-+v_-)^2\right).
\end{equation}
This condition is satisfied if and only if $\alpha = 2/3$.
Therefore, the numerical fluxes determined by
\begin{equation}
    u^{\mathrm{num}} = \mean{u}, \quad
    v^{\mathrm{num}} = \mean{v}, \quad
    \alpha = 2/3,
\end{equation}
are entropy-conservative. The resulting finite-volume discretization can be written as
\begin{equation}
\begin{aligned}
    &\partial_t u_i
    + \frac{2}{3}\frac{\mean{u}_{i+1/2}\jump{u+v}_{i+1/2} + \mean{u}_{i-1/2}\jump{u+v}_{i-1/2}}{2 \Delta x}
    + \frac{u_i}{3}\frac{\jump{u+v}_{i+1/2} + \jump{u+v}_{i-1/2}}{2\Delta x}
    = 0,
    \\
    & \partial_t v_i
    + \frac{2}{3}\frac{\mean{v}_{i+1/2}\jump{u+v}_{i+1/2} + \mean{v}_{i-1/2}\jump{u+v}_{i-1/2}}{2 \Delta x}
    + \frac{v_i}{3}\frac{\jump{u+v}_{i+1/2} + \jump{u+v}_{i-1/2}}{2 \Delta x}
    = 0,
\end{aligned}
\end{equation}
which is equivalent to the scheme presented by Castro et al.~\cite{Castro2013}.
It reduces to the standard split-form discretization of Burgers' equations~\eqref{eq_burgers}, which is obtained using the uniquely determined EC flux.

\subsection{General polynomial equation}
\label{sec:general_polynomial_equation}
A generalization of Burgers' equation involves a monomial flux of degree $k \in \mathbb{N}$.
For $k > 2$, we can write the equation in different nonconservative forms as
\begin{equation}
\partial_t u + u^m \partial_x u^n = 0, \quad m, n \in \mathbb{N}, \; m + n = k.
\end{equation}
The ansatz \eqref{FV_convex} becomes
\begin{equation}\label{eq_poly}
\partial_t u_i + \alpha \frac{h^{\mathrm{num}}_{i+1/2} \jump{u^n}_{i+1/2} + h^{\mathrm{num}}_{i-1/2}\jump{u^n}_{i-1/2} }{2 \Delta x} + \left (1- \alpha \right ) \frac{u^m_i \jump{u^n}_{i+1/2} + u^m_i \jump{u^n}_{i-1/2}}{2 \Delta x} = 0,
\end{equation}
where
\begin{equation}
h^{\mathrm{num}}_{i+1/2}= h^{\mathrm{num}}(u_i, u_{i+1} ) = h^{\mathrm{num}}(u_{i+1}, u_{i} ), \quad h^{\mathrm{num}}(u, u) = u^m,
\end{equation}
for consistency and symmetry of the fluxes. We seek numerical fluxes $h^{\mathrm{num}}$ and a parameter $\alpha$ such that the scheme preserves the entropy given by
\begin{equation}
    U = \frac{u^2}{2}, \quad
    F = \frac{n}{m+n+1}u^{m+n+1}, \quad
    \psi = -F, \quad
    \omega = u.
\end{equation}
Hence, the entropy conservation condition~\eqref{TadmorDef_convex} becomes
\begin{equation}
    \alpha \mean{u} h^{\mathrm{num}} \jump{u^n} + (1-\alpha) \mean{u^{m+1}} \jump{u^n} = \frac{n}{m+n+1} \jump{u^{m+n+1}}.
\end{equation}
This shows that there does not always need to be a solution.
Indeed, if $m, n \ge 2$ are both even, we can choose $u_\pm = \pm 1$ to get a contradiction, since the left-hand side is zero while the right-hand side is not.

However, we can interpret the product $h^{\mathrm{num}}\jump{u^n}$ appearing in the semi-discretization~\eqref{eq_poly} as a fluctuation, and derive an expression for it directly.
This reduces to the polynomial division
\begin{equation}
h^{\mathrm{num}} \jump{u^n} = \frac{1}{\alpha \mean{u}} \left ( \frac{n}{m+n+1} \jump{u^{m+n+1}} - (1-\alpha) \mean{u^{m+1}} \jump{u^n} \right ) = \frac{P(u_-, u_+)}{\alpha \mean{u}},
\label{hnum_monomial}
\end{equation}
where we defined the polynomial $P(u_-, u_+)$.
If $n$ is even,~\eqref{hnum_monomial} can only be satisfied for $u_- = u_+ \neq 0$ if $m$ is odd, since otherwise the left-hand side would be zero while the right-hand side is not. On the other hand, if $n$ is odd, the numerical flux is well-behaved and we can always divide~\eqref{hnum_monomial} by $\jump{u^n}$.
In particular, in the case where $n$ is odd and $m$ is even, we obtain the compatibility condition
\begin{equation}
\alpha = \frac{m+1}{m+n+1},
\end{equation}
which we use from now on, for simplicity, to cover the cases where $n$ is odd and $m \in \mathbb{N}$.
Polynomial long division yields
\begin{equation}
h^{\mathrm{num}} \jump{u^n} = \frac{2 n}{m+1} \left(\sum_{k=0}^{m+n} (-1)^k u_+^{m+n-k}u_-^{k} - \mean{u^{m+1}} \sum_{k=0}^{n-1} (-1)^k u_+^{n-1-k}u_-^k\right),
\end{equation}
which is antisymmetric and consistent, and thus we have proved the following theorem.
\begin{theorem}
Given $m \in \mathbb{N}$ and $n$ odd. The semi-discretization~\eqref{eq_poly} is entropy conservative if $\alpha = \frac{m+1}{m+n+1}$ and the fluctuation $h^{\mathrm{num}} \jump{u^n}$ is given by
\begin{equation}\label{ec_1_monomial}
h^{\mathrm{num}} \jump{u^n} = \frac{2 n}{m+1} \left(\sum_{k=0}^{m+n} (-1)^k u_+^{m+n-k}u_-^{k} - \mean{u^{m+1}} \sum_{k=0}^{n-1} (-1)^k u_+^{n-1-k}u_-^k\right).
\end{equation}
\end{theorem}

To analyze the general case including even $n$, we rewrite the continuous equation as
\begin{equation}
\partial_t u + \partial_x (u^{m+n}) - u^{n}\partial_x (u^m) = 0.
\end{equation}
A corresponding semi-discretization reads
\begin{equation}
\partial_t u_i + \frac{f^\mathrm{num}_{i+1/2} - f^\mathrm{num}_{i-1/2}}{\Delta x} - \alpha \frac{h^{\mathrm{num}}_{i+1/2}\jump{u^m} + h^{\mathrm{num}}_{i-1/2}\jump{u^m}}{2 \Delta x} - (1-\alpha) \frac{u^n_i \jump{u^m}_{i+1/2} + u^n_i \jump{u^m}_{i-1/2}}{2 \Delta x} = 0.
\label{eq_poly2}
\end{equation}
Here, the entropy potential is
\begin{equation}
\psi = \omega f - F = u \cdot u^{m+n} - \frac{n}{m+n+1} u^{m+n+1} = \frac{m+1}{m+n+1} u^{m+n+1}
\end{equation}
and the EC condition \eqref{TadmorDef_convex} reads
\begin{equation}
\jump{u} f^{\mathrm{num}} + \alpha \mean{u} h^{\mathrm{num}} \jump{u^m} + (1-\alpha) \mean{u^{n+1}} \jump{u^m} = \frac{m+1}{m+n+1} \jump{u^{m+n+1}}.
\end{equation}
Decomposing the jumps as $\jump{u^k} = (u_+ - u_-) \sum_{l=0}^{k-1} u_+^{k-1-l} u_-^l$ gives
\begin{equation}
f^{\mathrm{num}} = \frac{m+1}{m+n+1} \sum_{k=0}^{m+n} u_+^{m+n-k} u_-^k - \alpha \mean{u} h^{\mathrm{num}} \sum_{k=0}^{m-1} u_+^{m-1-k} u_-^k - (1-\alpha) \mean{u^{n+1}} \sum_{k=0}^{m-1} u_+^{m-1-k} u_-^k.
\end{equation}
Here $h^{\mathrm{num}}$ remains a degree of freedom that may be chosen to be symmetric and consistent with $u^n$. Therefore, we can state the following theorem.
\begin{theorem}
Given $m, n \in \mathbb{N}$, the semi-discretization~\eqref{eq_poly2} conserves the entropy $U = \frac{u^2}{2}$ if
\begin{equation}\label{ec_2_monomial}
f^\mathrm{num} =  \frac{m+1}{m+n+1} \sum_{k=0}^{m+n} u_+^{m+n-k} u_-^k - \alpha \mean{u} h^{\mathrm{num}} \sum_{k=0}^{m-1} u_+^{m-1-k} u_-^k - (1-\alpha) \mean{u^{n+1}} \sum_{k=0}^{m-1} u_+^{m-1-k} u_-^k,
\end{equation}
where $\alpha \in \mathbb{R}$ and $h^\mathrm{num}$ is a symmetric and consistent mean value of $u^n$.
\end{theorem}

This shows that the choice of how to write a system in nonconservative form can affect the construction of entropy-conserving methods.

\subsection{Shallow-water equations}

Considerable effort has been devoted to the semi-discretization of the shallow-water equations, with various entropy-conservative fluxes and split forms~\cite{FJORDHOLM20115587,gassner2016shallowwater,WINTERMEYER2017200,Ranocha2017}.
The 1D equations are
\begin{equation}
\begin{aligned}
    &\partial_t h + \partial_x \left ( hv \right ) = 0,\\
    & \partial_t (hv) + \partial_x \left ( hv^2 + \frac{1}{2}g h^2 \right ) + g h \partial_x b = 0,
\end{aligned}
\end{equation}
where $h$ is the water height, $v$ the velocity, $b$ the bottom topography (bathymetry), and $g$ the gravitational acceleration constant.
The entropy pair is
\begin{equation}
U = \frac{1}{2}hv^2 + \frac{1}{2}gh^2 + ghb, \quad F = \frac{1}{2}h v^3 + ghv(h + b).
\end{equation}
The entropy variables and the entropy potential are
\begin{equation}
\vec{\omega} = \left (-\frac{1}{2}v^2 + gh + gb, v \right ), \quad \psi = \vec{\omega} \cdot \vec{f} - F = \frac{1}{2} g h^2 v.
\end{equation}
The entropy conservation condition~\eqref{TadmorDef_convex} becomes
\begin{multline}
-\frac{1}{2}\jump{v^2}f^h + g \jump{h}f^h + g \jump{b} f^h + \jump{v}f^{hv} \\
- \alpha g \mean{v}h^{\mathrm{num}} \jump{b} -(1 - \alpha)g \mean{hv} \jump{b} = \frac{1}{2}g\jump{h^2 v}.
\end{multline}
We introduce the decomposition $f^{hv} = f^{hv,1} + f^{hv,2}$, where
\begin{equation}
f^{hv,1}(\vec{u},\vec{u}) = h v^2, \quad f^{hv,2}(\vec{u},\vec{u}) = \frac{1}{2} g h^2, \quad \vec{u} = (h, hv)^T.
\end{equation}
Grouping similar terms shows that the EC condition is satisfied if
\begin{equation}
    \begin{aligned}
& f^h = \alpha h^{\mathrm{num}} \mean{v} + (1-\alpha) \mean{hv},\\
& f^{hv,1} = f^h \mean{v}, \\
& g \jump{h} f^h + \jump{v} f^{hv,2} = \frac{1}{2} g \jump{h^2 v}.
    \end{aligned}
    \label{SW_conditions}
\end{equation}
The first two equations determine the numerical fluxes up to $\alpha$. The last condition needs to be satisfied in order to get the last remaining flux $f^{hv,2}$.
Using the discrete product rule $\jump{a b} = \mean{a}\jump{b} + \jump{a}\mean{b}$, we can rewrite the right-hand side using
\begin{equation}
\begin{aligned}
    &\quad
    \frac{1}{2}\jump{h^2 v}
    =
    (1-\alpha) \jump{h^2v} + \left(\alpha - \frac{1}{2}\right) \jump{h^2v}
    \\
    &= ( 1 - \alpha) \jump{h} \mean{hv} + (1 - \alpha) \jump{hv} \mean{h} + \left (\alpha - \frac{1}{2} \right ) 2 \jump{h}\mean{h} \mean{v}+ \left ( \alpha - \frac{1}{2} \right ) \jump{v} \mean{h^2}
    \\
    &= ( 1 - \alpha) \jump{h} \mean{hv} + (1 - \alpha) \jump{h} \mean{h} \mean{v} + (1-\alpha) \jump{v}\mean{h}^2
    \\
    &\quad
    + \left (\alpha - \frac{1}{2} \right ) 2 \jump{h}\mean{h} \mean{v}+ \left ( \alpha - \frac{1}{2} \right ) \jump{v} \mean{h^2}.
\end{aligned}
\end{equation}
Inserting this splitting into the last of~\eqref{SW_conditions} gives
\begin{equation}
\begin{aligned}
    f^h &= \alpha h^{\mathrm{num}} \mean{v} + (1-\alpha) \mean{hv}, \\
    f^{hv} &= f^h \mean{v} + (1-\alpha)g \mean{h}^2 + \Bigl (\alpha - \frac{1}{2} \Bigr ) g \mean{h^2},\\
    h^{\mathrm{num}} &= \mean{h},
\end{aligned}
\label{shallow_water_split}
\end{equation}
which is the same one-parameter family of numerical fluxes for split forms constructed in~\cite{Ranocha2017}, recovering the numerical fluxes of Fjordholm et al.~\cite{FJORDHOLM20115587} and Gassner et al.~\cite{gassner2016shallowwater}.


\subsection{Hyperbolized Sainte-Marie equations}
The Sainte-Marie system~\cite{saintemarie2015,sainte2011vertically} is a depth-averaged model for non-hydrostatic free-surface flows, which accounts for the effects related to dispersive waves (see also~\cite{fernandeznieto:hal-01324012} for more details).
We consider the hyperbolic reformulation derived by Escalante et al.~\cite{ESCALANTE2019385}, where the evolution of the non-hydrostatic pressure is introduced as a prognostic variable, after manipulating the original system of equations in~\cite{saintemarie2015} by coupling the velocity field with the remaining conservation laws.
As we are interested in EC fluxes, we consider the system without friction for $\gamma = 2$, i.e.,
\begin{equation}
    \label{eq:escalante_saintemarie}
    \begin{aligned}
    &\partial_t h + \partial_x (h v) = 0, \\
    &\partial_t (h v) + \partial_x \left(h v^2 + \frac{1}{2} g h^2 + hp\right) + gh \partial_x b + 2p \partial_x b = 0,\\
    &\partial_t (h w) + \partial_x (h w v) = 2p,\\
    &\partial_t (hp) + \partial_x (h p v) + h c^2 \partial_x v - 2 c^2 v \partial_x b = -2 c^2 w,\\
     \end{aligned}
\end{equation}
where $h$ is the water depth, $v$ and $w$ denote the depth-averaged horizontal and vertical velocity, $p$ is the non-hydrostatic pressure, $b$ is the bottom topography, $g$ is the gravitational acceleration constant, and $c = \alpha \sqrt{g b_0}$ is the constant celerity, with $b_0$ being the reference depth and $\alpha >1$. Following~\cite{ESCALANTE2019385}, the system admits an entropy pair
\begin{equation}
U = \frac{1}{2}h v^2 + \frac{1}{2}h w^2 + \frac{1}{2c^2} h p^2 + \frac{1}{2}gh^2 + ghb, \quad F = U v + \frac{1}{2}g h^2 v + h p v.
\end{equation}
The entropy variables and the entropy potential are
\begin{equation}
\vec{\omega} = \left ( -\frac{v^2}{2} - \frac{w^2}{2} - \frac{p^2}{2c^2} + gh + gb, v, w, \frac{p}{c^2} \right ), \quad \psi = \vec{\omega} \cdot \vec{f} - F = \frac{1}{2} g h^2 v.
\end{equation}

For the semi-discretization of the system~\eqref{eq:escalante_saintemarie} we use a linear combination of the first and third form as in Theorem~\ref{th:EC_linear_combination}, which reads
\begin{multline}
    \partial_t \vec{u}_i
    + \frac{\vec{f}^{\mathrm{num}}_{i+1/2} - \vec{f}^{\mathrm{num}}_{i-1/2}}{\Delta x}
    \\
    + \sum_{k=1}^4 \left ( \alpha_k \frac{\vec{H}^{\mathrm{num}}_{k,i+1/2}\jump{g_k}_{i+1/2} + \vec{H}^{\mathrm{num}}_{k,i-1/2}\jump{g_k}_{i-1/2}}{2 \Delta x} + (1-\alpha_k) \vec{H}_{k,i} \frac{{\jump{g_k}_{i+1/2} + \jump{g_k}_{i-1/2}}}{2 \Delta x} \right )
    = \vec{0},
\label{saintmarie_semi}
\end{multline}
where we have introduced for each nonconservative term a parameter $\alpha_k$ to allow for more flexibility in the choice of the numerical fluxes and
\begin{equation}
\vec{u} =
\begin{pmatrix}
h \\
hv \\
hw \\
hp
\end{pmatrix},
\quad
\vec{f}^{\mathrm{num}} =
\begin{pmatrix}
f^{h} \\
f^{hv} \\
f^{hw} \\
f^{hp}
\end{pmatrix},
\quad
g_1 = g_2 = g_4 = b, \quad g_3 = v,
\label{saintmarie_noncons_fluxes}
\end{equation}
\begin{equation}
\vec{H}^{\mathrm{num}}_1 =
\begin{pmatrix}
0  \\
gh^{\mathrm{num}} \\
0 \\
0
\end{pmatrix},
\quad
\vec{H}^{\mathrm{num}}_2 =
\begin{pmatrix}
0  \\
2p^{\mathrm{num}} \\
0 \\
0
\end{pmatrix},
\quad
\vec{H}^{\mathrm{num}}_3 =
\begin{pmatrix}
0  \\
0 \\
0 \\
c^2 h^{\mathrm{num}} \\
\end{pmatrix},
\quad
\vec{H}^{\mathrm{num}}_4 =
\begin{pmatrix}
0  \\
0 \\
0 \\
-2c^2 v^{\mathrm{num}} \\
\end{pmatrix}.
\end{equation}
Following Theorem~\ref{th:EC_linear_combination}, the entropy conservation condition can be written as
\begin{equation}
\jump{\vec{\omega}} \cdot \vec{f}^{\mathrm{num}} -\sum_{k=1}^4 \Bigl ( \alpha_k \mean{\vec{\omega}} \cdot \vec{H}^{\mathrm{num}}_k\jump{g_k} + (1-\alpha_k) \mean{\vec{\omega} \cdot \vec{H}_k} \jump{g_k} \Bigr ) = \jump{\psi}.
\end{equation}
Substituting the above expressions, we have
\begin{equation}
\begin{aligned}
    &
    -\jump{v}\mean{v}f^h - \jump{w}\mean{w}f^{h} - \frac{\jump{p}}{c^2}\mean{p}f^{h} + g\jump{h}f^h + g \jump{b}f^h
    + \jump{v}f^{hv} + \jump{w}f^{hw}
    \\&
    + \frac{\jump{p}}{c^2}f^{hp}
    -\alpha_1 g \mean{v} h^{\mathrm{num}}\jump{b} - (1-\alpha_1) g \mean{hv} \jump{b}
    -2 \alpha_2 \mean{v} p^{\mathrm{num}}\jump{b} - 2 (1- \alpha_2)\mean{pv}\jump{b}
    \\&
- \alpha_3 \mean{p}h^{\mathrm{num}}\jump{v}- (1-\alpha_3) \mean{ph} \jump{v} +2 \alpha_4 \mean{p}v^{\mathrm{num}}\jump{b} + 2 (1- \alpha_4)\mean{pv}\jump{b} = \frac{1}{2}g\jump{h^2v}.
\end{aligned}
\end{equation}
We consider the splitting $f^{hv} = f^{hv,\mathrm{sw}} + f^{hv,\mathrm{nh}}$, where the superscript $\mathrm{sw}$ denotes the classical shallow-water hydrostatic pressure term $\left ( \frac{1}{2}gh^2 \right )$ and $\mathrm{nh}$ denotes the non-hydrostatic pressure term $\left (hp \right )$. We can at this point split the terms and isolate the shallow-water terms from the rest, where the shallow-water part is satisfied by~\eqref{shallow_water_split}.
The remaining terms are
\begin{equation}
\begin{aligned}
    &
    - \jump{w}\mean{w}f^{h} - \frac{\jump{p}}{c^2}\mean{p}f^{h} + g \jump{b}f^h
    + \jump{v}f^{hv,\mathrm{nh}} + \jump{w}f^{hw} + \frac{\jump{p}}{c^2}f^{hp}
    \\&
    -2 \alpha_2 \mean{v} p^{\mathrm{num}}\jump{b} - 2 (1- \alpha_2)\mean{pv}\jump{b}
    - \alpha_3 \mean{p}h^{\mathrm{num}}\jump{v} - (1-\alpha_3) \mean{ph} \jump{v}
    \\&
    +2 \alpha_4 \mean{p}v^{\mathrm{num}}\jump{b} + 2 (1- \alpha_4)\mean{pv}\jump{b} = 0.
\end{aligned}
\end{equation}
The terms with the variable bathymetry are satisfied if $\alpha_2 = \alpha_4$.
The remaining terms are easily satisfied by imposing that each coefficient of the jumps vanishes, which leads to the following result.
\begin{theorem}
\label{th:ec_saintmarie}
   Given $\alpha_1, \alpha_2, \alpha_3 \in \mathbb{R}$, the semi-discretization~\eqref{saintmarie_semi} with $\alpha_4 = \alpha_2$ conserves the total energy $U$ if
\begin{equation}
    \label{eq:ec_saintmarie_flux}
    \begin{aligned}
        f^h &= \alpha_1 \mean{h} \mean{v} + (1-\alpha_1) \mean{hv} \\
        f^{hv} &= f^h \mean{v} + (1-\alpha_1)g \mean{h}^2 + \Bigl (\alpha_1 - \frac{1}{2} \Bigr ) g \mean{h^2} + \alpha_3 \mean{p}\mean{h} + (1-\alpha_3) \mean{ph}\\
        f^{hw} &= f^h \mean{w}, \quad
        f^{hp} = f^h \mean{p}, \quad
        h^{\mathrm{num}} = \mean{h}, \quad
        v^{\mathrm{num}} = \mean{v}, \quad
        p^{\mathrm{num}} = \mean{p}.
    \end{aligned}
\end{equation}
\end{theorem}

\subsubsection{Well-balanced steady states}
The entropy-conservative semi-discretization given in Theorem~\ref{th:ec_saintmarie} is well-balanced for the lake-at-rest steady state, which is given by
\begin{equation}\label{eq:lake_at_rest_saintmarie}
v = w = p = 0, \quad h + b = \text{const}.
\end{equation}
Note that for $w = p = 0$, the system~\eqref{eq:escalante_saintemarie} reduces to the shallow-water equations and for $h+b=\text{const}$ we have the classical lake-at-rest steady state.
\begin{lemma}\label{lemma_wb_saintemarie}
The semi-discretization given in Theorem~\ref{th:ec_saintmarie} is well-balanced for the lake-at-rest steady state~\eqref{eq:lake_at_rest_saintmarie}.
\end{lemma}
\begin{proof}
Since the velocities and the non-hydrostatic pressure $p$ are identically zero, the only terms that can be non-zero are in the momentum equation of $hv$
\begin{multline}
\frac{f^{hv,p}_{i+1/2} - f^{hv,p}_{i-1/2}}{\Delta x} + \frac{\left (\alpha_1 gh^{\mathrm{num}}_{i+1/2} + \alpha_2 2 p^{\mathrm{num}}_{i+1/2}\right )\jump{b}_{i+1/2} + \left (\alpha_1 gh^{\mathrm{num}}_{i-1/2} + \alpha_2 2 p^{\mathrm{num}}_{i-1/2}\right )\jump{b}_{i-1/2}}{2\Delta x}+ \\
\Bigl ((1-\alpha_1)g h_i + (1-\alpha_2) 2 p_i \Bigr ) \frac{\jump{b}_{i+1/2} + \jump{b}_{i-1/2}}{2\Delta x} = 0.
\end{multline}
Again, we consider the splitting $f^{hv} = f^{hv,sw} + f^{hv,nh}$, and the proof for the shallow-water terms follows Lemma~12 in~\cite{Ranocha2017}. Now, considering that $p = 0$, the remaining terms vanish.
\end{proof}

\subsection{Ideal magnetohydrodynamics equations}
We consider the three-dimensional compressible magnetohydrodynamics (MHD) equations augmented with a generalized Lagrange multiplier (GLM) divergence-cleaning technique~\cite{Derigs2018420, bohm2018entropy}.
In this formulation, nonconservative
terms naturally appear if the divergence-free constraint on the magnetic field is
not strictly enforced at the discrete level.
The GLM-MHD equations generalize the classical Euler equations by including the magnetic effects and the divergence-cleaning field.
The primary variables are the fluid density $\varrho$, the momentum $\varrho \vec{v} = \varrho ( v_1, v_2, v_3)^T$, and the total energy $\varrho E$, together with the magnetic field $\vec{B} = (B_1, B_2, B_3)^T$ and the scalar field $\Psi$ responsible for cleaning divergence errors.
The governing equations can be expressed compactly as~\cite{Derigs2018420, bohm2018entropy}
\begin{multline}
    \label{eq:glm-mhd}
\partial_t
\begin{pmatrix} \varrho \\ \varrho \vec{v} \\ \varrho E \\ \vec{B} \\ \Psi
\end{pmatrix}
+
\nabla \cdot
\begin{pmatrix}
\varrho \vec{v} \\
\varrho \vec{v} \otimes \vec{v} + \left (p + \frac{1}{2} \|\vec{B}\|^2 \right )\vec{I} - \vec{B} \otimes \vec{B}
 \\
\vec{v}\left (\frac{1}{2}\varrho \left\|\vec{v}\right\|^2 + \frac{\gamma p}{\gamma -1} + \|\vec{B}\|^2 \right )
- \vec{B}\left (\vec{v}\cdot\vec{B}\right )
+ c_h \Psi \vec{B}
 \\
 \vec{v}\otimes \vec{B} - \vec{B}\otimes \vec{v}
+ c_h \Psi \vec{I} \\
 c_h \vec{B} \\
\end{pmatrix}
\\
 +
(\nabla \cdot \vec{B})
\begin{pmatrix}
 0 \\ \vec{B} \\ \vec{v} \cdot \vec{B} \\  \vec{v}
\\ 0
\end{pmatrix}
 +
\nabla \Psi \cdot
\begin{pmatrix}
 0 \\ 0 \\ \vec{v}\Psi \\  0 \\ \vec{v}
\end{pmatrix}
 =
 \begin{pmatrix}
     0 \\ \vec{0} \\ 0 \\ \vec{0} \\ 0
 \end{pmatrix}.
\end{multline}
Here $c_h$ is the divergence-cleaning speed, $\vec{I}$ is the $3 \times 3$ identity matrix,
and the pressure is
\begin{equation}
\label{eq:pressure_mhd}
p = (\gamma - 1) \left( \varrho E - \frac{1}{2} \varrho \|\vec{v}\|^2
- \frac{1}{2} \|\vec{B}\|^2
- \frac{1}{2} \Psi^2
\right),
\end{equation}
where $\gamma = 5/3$ for a monoatomic ideal gas.
Under the divergence-free constraint $\nabla \cdot \vec{B} = 0$, the ideal MHD equations admit an entropy function~\cite{bohm2018entropy,Derigs2018420,ChandrashekarKlingenberg2016}
\begin{equation}
\partial_t U(\vec{u}) + \nabla \cdot \vec{F}(\vec{u}) = 0
\end{equation}
with
\begin{equation}
    U(\vec{u}) = -\frac{\varrho s}{\gamma -1}, \quad \vec{F}(\vec{u}) = - \frac{\varrho s}{\gamma -1} \vec{v}.
\end{equation}
The associated entropy variables are given by~\cite{ChandrashekarKlingenberg2016}
\begin{equation}
    \vec{\omega} = \left (\frac{\gamma-s}{\gamma-1}-\beta \|\vec{v}\|^2,\ 2 \beta \vec{v},\ - 2\beta,\ 2 \beta \vec{B}, 2 \beta \Psi \right ),
\end{equation}
where $\beta = \frac{\varrho}{2{p}}$. The corresponding entropy potential is given by~\cite{ChandrashekarKlingenberg2016}
\begin{equation}
\psi_j = \vec{\omega} \cdot \vec{f}_j - F_j = \varrho v_j + \beta v_j \|\vec{B}\|^2 + 2 \beta c_h B_j \Psi - \phi B_j, \quad j = 1,2,3,
\end{equation}
where $\phi = 2 \beta \vec{v} \cdot \vec{B}$.
Applying the condition~\eqref{TadmorDef_convex} for $\alpha = 0$ yields
\begin{equation}
\jump{\vec{\omega}} \cdot \vec{f}^{\mathrm{num},j} - \mean{\phi}\jump{B_j} = \jump{\varrho v_j} + \jump{\beta v_j \| \vec{B} \|^2} + \jump{2 \beta c_h B_j \Psi} - \jump{\phi B_j}.
\end{equation}
This is equivalent to the condition in~\cite[Eq. 4.41]{Derigs2018420}, which have been used there to construct EC fluxes, and extended to the multi-ion MHD equations in~\cite{rueda2025entropy}.

\subsection{Compressible Euler equations in nonconservative form}
Renac~\cite{RENAC20191} considered the Euler equations with internal energy $e$ as thermodynamic variable.
We consider the same system but rewrite the pressure work term in the internal energy equation and include a gravity potential $\phi$.
We extend the numerical fluxes presented in~\cite{RENAC20191}, avoiding the influence of the temperature on the pressure discretization in the momentum equation.
The equations are
\begin{equation}
\begin{aligned}
&\partial_t \varrho + \partial_x (\varrho v) = 0,\\
&\partial_t (\varrho v) + \partial_x ( \varrho v^2 + p) = -\varrho \partial_x \phi,\\
&\partial_t ( \varrho e) + \partial_x ( \varrho e v + pv) - v \partial_x p = 0,
\label{euler_noncons}
\end{aligned}
\end{equation}
where $v$ is the velocity, $p$ is the pressure, $e$ is the internal energy, and $\varrho$ is the density.
We assume the ideal gas law $p = (\gamma-1)\varrho e$ with adiabatic constant $\gamma$.
Two quantities are conserved for smooth solutions: the total energy
\begin{equation}
U_{\varrho E} = \varrho E = \varrho e + \frac{1}{2} \varrho v^2+ \varrho \phi
\end{equation}
and the thermodynamic entropy\footnote{In accordance with the mathematical literature, we use a minus sign to obtain a convex entropy. In contrast to~\cite{ismail2009affordable,ECChandra,Ranocha2018}, we do not use a scaling by $\gamma - 1$ to simplify terms when working with the internal energy as thermodynamic variable.}
\begin{equation}
    U_{\varrho s} = -\varrho s,
    \qquad
    s = \log \left ( \frac{p}{\varrho^{\gamma}}\right )
    = \log (\gamma - 1) + \log \left ( \frac{\varrho e }{\varrho^{\gamma}}\right ).
\end{equation}
Our goal is to find a method that conserves both total energy and entropy.

For the semi-discretization of~\eqref{euler_noncons}, we set $\alpha = 1$, which results in the numerical scheme
\begin{equation}
    \partial_t \vec{u}_i
    + \frac{\vec{f}^{\mathrm{num}}_{i+1/2} - \vec{f}^{\mathrm{num}}_{i-1/2}}{\Delta x}
    + \frac{\vec{H}^{\mathrm{num}}_{i+1/2}\jump{\vec{g}}_{i+1/2} + \vec{H}^{\mathrm{num}}_{i-1/2}\jump{\vec{g}}_{i-1/2}}{2 \Delta x}
    = \vec{0},
\label{euler_noncons_semi}
\end{equation}
where
\begin{equation}
\vec{u} =
\begin{pmatrix}
\varrho \\
\varrho v \\
\varrho e
\end{pmatrix},
\quad
\vec{f}^{\mathrm{num}} =
\begin{pmatrix}
f^{\varrho} \\
f^{\varrho v} \\
f^{\varrho e,1} + f^{\varrho e,2}
\end{pmatrix},
\quad
\vec{g} =
\begin{pmatrix}
\phi \\
p
\end{pmatrix},
\quad
\vec{H}^{\mathrm{num}} =
\begin{pmatrix}
0 & 0 \\
\varrho^{\mathrm{num}} & 0 \\
0 & v^{\mathrm{num}}
\end{pmatrix}.
\label{euler_noncons_fluxes}
\end{equation}

\subsubsection{Total energy conservation}

For the total energy $U_{\varrho E}$, the associated flux and entropy variables are
\begin{equation}
F_{\varrho E} = (\varrho E  + p) v, \quad \vec{\omega}_{\varrho E} = \left (-\frac{1}{2}v^2 + \phi, v, 1 \right ).
\end{equation}
Since $\psi_{\varrho E} = \vec{\omega}_{\varrho E} \cdot \vec{f} - F_{\varrho E} = pv$, the entropy conservation condition becomes
\begin{equation}
\begin{aligned}
    -\frac{1}{2}\jump{v^2} f^{\varrho} + \jump{\phi}f^{\varrho} + \jump{v}f^{\varrho v} - \mean{v}\varrho^{\mathrm{num}}\jump{\phi} + v^{\mathrm{num}}\jump{p} = \jump{pv}.
\end{aligned}
\end{equation}
Using the discrete product rule $\jump{pv} = \jump{p}\mean{v} + \mean{p}\jump{v}$, we observe that the numerical fluxes satisfying the conditions for kinetic energy-preserving (KEP) fluxes~\cite{Jameson2008,Kuya2018,Ranocha2020,Ranocha2022, ranocha2018thesis}, i.e.,
\begin{equation}
\label{eq:kep_conditions}
\begin{aligned}
& f^{\varrho v} = \mean{v} f^{\varrho} + \mean{p},\\
& f^{\varrho} = \varrho^{\mathrm{num}} \mean{v},\\
& v^{\mathrm{num}} = \mean{v},
\end{aligned}
\end{equation}
conserve the total energy.
Thus, we have the following result.
\begin{theorem}
\label{thm:total_energy_conservation}
    The semi-discretization~\eqref{euler_noncons_semi} conserves the total energy $\varrho E$ if the numerical fluxes are KEP, i.e., \eqref{eq:kep_conditions} holds.
\end{theorem}
\begin{remark}
Note that the density flux associated with the gravity term, $\varrho^{\mathrm{num}}$, is not uniquely determined and therefore represents a degree of freedom, which can be used to enforce additional structure-preserving properties, e.g., the pressure-equilibrium-preserving (PEP) property~\cite{artiano2025structurepreservinghighordermethodscompressible,Shima2021,Ranocha2022}, or as we will see in the following, entropy conservation.
\end{remark}
\subsubsection{Entropy conservation}

Now we seek conditions for entropy conservation. The associated flux and entropy variables are
\begin{equation}
    F_{\varrho s} = -\varrho s v,
    \quad
    \vec{\omega}_{\varrho s} = \left (
        -\log \left ( \frac{\varrho e (\gamma-1)}{\varrho^{\gamma}}\right ) + \gamma,
         0,
         -\frac{\varrho}{\varrho e}
    \right ),
\end{equation}
and the entropy potential is $\psi_{\varrho s} = 0$.
By definition~\eqref{TadmorDef_form3}, an EC scheme is obtained if
\begin{equation}
    -\jump{\log \left ( \frac{\varrho e (\gamma-1)}{\varrho^{\gamma}}\right )} f^{\varrho}
    - \jump{\frac{\varrho}{\varrho e}} f^{\varrho e}
    - \mean{\frac{\varrho}{\varrho e}} v^{\mathrm{num}} \jump{p} = 0.
\end{equation}
Given that $p = (\gamma-1) \varrho e$, we can rewrite this condition as
\begin{equation}
    \jump{\log ( \varrho)}f^{\varrho}
    + \frac{1}{\gamma - 1} \jump{\log(\varrho/p)} f^{\varrho}
    - \jump{\frac{\varrho}{p}}f^{\varrho e}
    - \mean{\frac{\varrho}{p}}v^{\mathrm{num}}\jump{p} = 0.
\end{equation}
Splitting the contributions of the internal energy flux, we observe that the numerical fluxes
\begin{equation}
\begin{aligned}
& f^{\varrho}= \mean{\varrho}_{\log} \mean{v},\\
& f^{\varrho e} = \frac{1}{\gamma-1}\frac{\mean{\varrho}_{\log}}{\mean{\varrho/p}_{\log}} \mean{v} + \mean{v}\mean{p},\\
& v^{\mathrm{num}} = \mean{v},
\end{aligned}
\end{equation}
are entropy conservative, where $\mean{\varrho}_{\log} = \jump{\varrho} / \jump{\log \varrho}$ is the logarithmic mean~\cite{ismail2009affordable}.
Since the momentum fluxes are a degree of freedom, we can set $f^{\varrho v} = \mean{v} f^{\varrho} + \mean{p}$ to achieve KEP.
\begin{theorem}
The semi-discretization~\eqref{euler_noncons_semi} is entropy-conservative and KEP if
\begin{equation}
\label{eq:ec_and_kep_fluxes}
\begin{aligned}
    & f^{\varrho}= \mean{\varrho}_{\log} \mean{v},\\
    & f^{\varrho v} = \mean{v} f^{\varrho} + \mean{p},\\
    & f^{\varrho e} = \frac{1}{\gamma-1}\frac{f^{\varrho}}{\mean{\varrho/p}_{\log}} + \mean{v}\mean{p},\\
    & v^{\mathrm{num}} = \mean{v}.
\end{aligned}
\end{equation}
If in addition $\varrho^{\mathrm{num}} = \mean{\varrho}_{\log}$, the numerical flux conserves the total energy $\varrho E$.
\end{theorem}

Next, we investigate entropy-stable methods.
To do so, we introduce the interface velocity as
\begin{equation}
V_\mathrm{int} = \mean{v} -\beta \jump{p}, \quad \beta \in \mathbb{R}^+,
\label{interface_velocity}
\end{equation}
which comes from solving the Rankine-Hugoniot condition in the limit of incompressibility~\cite{AControlVolumeModeloftheCompressibleEulerEquationswithaVerticalLagrangianCoordinate}.
In the numerical experiments reported later, we set $\beta = \frac{1}{2 \mean{\varrho}|v|}$ and
\begin{equation}
f^{\varrho v} = \left (\mean{\varrho v} - \frac{\jump{\varrho v}}{2}\text{sign}(V_\mathrm{int}) \right ) V_\mathrm{int} + \mean{p} - \frac{1}{2}|v| \mean{\varrho} \jump{v},
\end{equation}
where $|v| = \max(|v_-|, |v_+|)$.
\begin{theorem}
The semi-discretization~\eqref{euler_noncons_semi} is entropy-stable if
\begin{equation}
    \label{eq:es_flux}
    \begin{aligned}
    & f^{\varrho}= \left (\mean{\varrho}_{\log} - \frac{\jump{\varrho}}{2} \operatorname{sign}(V_{\mathrm{int}}) \right ) V_{\mathrm{int}},\\
    & f^{\varrho v} = \text{any consistent discretization of the momentum flux},\\
    & f^{\varrho e} = \frac{1}{\gamma-1}\frac{f^{\varrho}}{ \mean{\varrho/p}_{\log}}  + \mean{p} V_{\mathrm{int}},\\
    & v^{\mathrm{num}} = V_{\mathrm{int}},\\
    \end{aligned}
\end{equation}
where $V_{\mathrm{int}}$ is given by~\eqref{interface_velocity}.
\end{theorem}
\begin{proof}
Splitting the fluxes into the entropy-conservative part and the dissipative part, we observe that the entropy stability condition~\eqref{TadmorDef_form3} is satisfied since
\begin{multline}
    \jump{\log ( \varrho)} f^{\varrho}
    + \frac{1}{\gamma - 1} \jump{\log(\varrho/p)} f^{\varrho}
    - \jump{\frac{\varrho}{p}} f^{\varrho e}
    - \mean{\frac{\varrho}{p}} v^{\mathrm{num}} \jump{p}
    \\
    =
    - \frac{1}{2} \jump{\log ( \varrho)} \jump{\varrho} \operatorname{sign}(V_{\mathrm{int}}) V_{\mathrm{int}}
    \le 0.
    \qedhere
\end{multline}
\end{proof}

\begin{remark}
    Note that also here the momentum flux $f^{\varrho v}$ is not uniquely determined. Choosing $f^{\varrho v} = f^{\varrho} \mean{v} + \mean{p}$ yields a numerical scheme that is entropy stable and KEP. Moreover, for $\beta = 0$, the numerical fluxes are also total energy conservative.
\end{remark}

\subsubsection{Another interpretation of kinetic energy-preserving and pressure equilibrium-preserving fluxes}

We have seen in Theorem~\ref{thm:total_energy_conservation} that KEP methods conserve the total energy when working with the internal energy as thermodynamic variable.
Next, we will present another interpretation of kinetic energy preservation~\cite{Jameson2008,Kuya2018,Ranocha2020,Ranocha2022, ranocha2018thesis} and pressure equilibrium preservation~\cite{Shima2021,Ranocha2022}.
For simplicity, we set the gravity potential $\phi = 0$.
Since $p = (\gamma - 1) \varrho e$ for an ideal gas, the system is equivalent to
\begin{equation}
\begin{aligned}
    &\partial_t \varrho + \partial_x (\varrho v) = 0,\\
    &\partial_t (\varrho v) + \partial_x ( \varrho v^2 + p) = 0,\\
    &\partial_t p + \gamma \partial_x (p v) - (\gamma - 1) v \partial_x p = 0.
    \label{euler_noncons_pressure}
\end{aligned}
\end{equation}
Pressure equilibrium-preserving (PEP) methods~\cite{Shima2021,Ranocha2022} keep states of the form $v \equiv \mathrm{const}$ and $p \equiv \mathrm{const}$, where the compressible Euler equations reduce to the linear advection equation $\partial_t \varrho + v \partial_x \varrho = 0$.
A consistent discretization of the pressure equation (or equivalently the internal energy equation) yields $\partial_t p = 0$ for these states.
Moreover, the KEP conditions \eqref{eq:kep_conditions} reduce the semi-discrete momentum equation in this case to
\begin{equation}
\begin{aligned}
    -\Delta x \, \varrho_i \partial_t v_i
    &=
    - \Delta x \, \partial_t (\varrho v)_i
    + \Delta x \, v_i \partial_t \varrho
    =
    f^{\varrho v}_{i+1/2} - f^{\varrho v}_{i-1/2}
    - v_i \left ( f^{\varrho}_{i+1/2} - f^{\varrho}_{i-1/2} \right )
    \\
    &=
    \left( v f^{\varrho}_{i + 1/2} + p \right)
    - \left( v f^{\varrho}_{i-1/2} + p  \right)
    - v \left ( f^{\varrho}_{i+1/2} - f^{\varrho}_{i-1/2} \right )
    =
    0,
\end{aligned}
\end{equation}
since $v \equiv \mathrm{const}$ and $p \equiv \mathrm{const}$.
Thus, discretizing the pressure or internal energy equation consistently and using a KEP method automatically results in a PEP method.
Since the KEP method conserves the total energy, the induced total energy flux is given by \eqref{F_num_form1} as
\begin{equation}
\label{eq:total_energy_flux_from_pressure_and_kep}
\begin{aligned}
    F^\mathrm{num}_{\varrho E}
    &=
    \mean{F_{\varrho E}}
    + \mean{\vec{\omega}_{\varrho E}} \cdot \vec{f}^{\mathrm{num}}
    - \mean{\vec{\omega}_{\varrho E} \cdot \vec{f}}
    - \frac{1}{4} \jump{\vec{\omega}_{\varrho E}} \cdot \vec{H}^\mathrm{num} \jump{\vec{g}}
    \\
    &=
    - \frac{1}{2} \mean{v^2} f^{\varrho}
    + \mean{v} f^{\varrho v}
    + f^{\varrho e}
    - \mean{p v}
    \\
    &=
    \left( \mean{v}^2 - \frac{1}{2} \mean{v^2} \right) f^{\varrho}
    + f^{\varrho e}
    + \mean{p} \mean{v}
    - \mean{p v}
\end{aligned}
\end{equation}
because of the KEP conditions \eqref{eq:kep_conditions}.
Using $f^{\varrho e} = f^{\varrho e, 1} + \mean{p} \mean{v}$, where $f^{\varrho e, 1}$ is a consistent discretization of $\varrho e v$, yields
\begin{equation}
    F^\mathrm{num}_{\varrho E}
    =
    \left( \mean{v}^2 - \frac{1}{2} \mean{v^2} \right) f^{\varrho}
    + f^{\varrho e, 1}
    + 2 \mean{p} \mean{v}
    - \mean{p v}.
\end{equation}
Using the product mean
\begin{equation}
    \prodmean{a \cdot b}
    =
    2 \mean{a} \mean{b} - \mean{ab}
    =
    \frac{1}{2} \left( a_- b_+ + a_+ b_- \right)
\end{equation}
and the entropy-conservative fluxes \eqref{eq:ec_and_kep_fluxes} yields the induced total energy flux
\begin{equation}
    F^\mathrm{num}_{\varrho E}
    =
    \frac{1}{2} \mean{\varrho}_{\log} \mean{v} \prodmean{v \cdot v}
    + \frac{1}{\gamma-1}\frac{\mean{\varrho}_{\log}}{\mean{\varrho/p}_{\log}} \mean{v}
    + \prodmean{p \cdot v},
\end{equation}
which is exactly the total energy flux of the KEP and PEP fluxes constructed in~\cite{Ranocha2020,Ranocha2022}.
This fits to the uniqueness result of numerical fluxes that are EC, KEP, PEP, and do not contain an influence of the pressure in the density flux~\cite[Theorem~1]{Ranocha2022}.
Similarly, using $f^{\varrho} = \mean{\varrho} \mean{v}$ and $f^{\varrho e, 1} = \mean{\varrho e} \mean{v}$ yields
\begin{equation}
    F^\mathrm{num}_{\varrho E}
    =
    \frac{1}{2} \mean{\varrho}\mean{v} \prodmean{v \cdot v}
    + \mean{\varrho e} \mean{v}
    + \prodmean{p \cdot v},
\end{equation}
i.e., the total energy flux of the KEP and PEP fluxes of~\cite{Shima2021}.
We will use this derivation of KEP and PEP fluxes for more complicated systems in the future.

\subsection{General procedure to construct affordable entropy conservative/stable fluxes}
The procedure to construct affordable entropy-conservative or entropy-stable fluxes is conceptually similar to the approach presented in~\cite{Ranocha2018}.
Once a semi-discretization is fixed, the corresponding entropy condition can be applied, following the same steps described in~\cite[Procedure 4.1]{Ranocha2018}.

However, as shown in the previous chapter, the choice of semi-discretization is often not unique and may not always guarantee the existence of an EC scheme.
In particular, rewriting the equations at the continuous level using the product rule can facilitate, in some cases, the derivation of numerical fluxes and help avoid issues related to the non-existence of the parameter $\alpha$, as in Section~\ref{sec:general_polynomial_equation}.
This difficulty could in principle be avoided by using the fourth form~\eqref{FV_form4}, which we consider less convenient in practice for the set of equations investigated.
For this reason and some considerations about split forms given in the next section, the fluxes derived in Section~\ref{sec:numerical_fluxes} rely on the linear combination of the first and third form~\eqref{TadmorDef_convex}.


\section{Summation-by-parts operators and nonconservative systems}
\label{sec:sbp_cartesian}
In this section, we first briefly review the concept of SBP operators.
Then, we extend the results obtained for the low-order FV discretization in Section~\ref{sec:nonconservative_systems} to general high-order SBP operators.
The notation closely follows~\cite{Ranocha2018, ranocha2018thesis}.

\subsection{Summation-by-parts operators}

Following~\cite{Hicken2016_SBP,Ranocha2018}, we recall the definition of an SBP operator.
\begin{definition}[SBP operator]
    An SBP operator on a $d$-dimensional element $E$ with order of accuracy $p \in \mathbb{N}$ consists of the following components
    \begin{itemize}
        \item Derivative operators $\mtx{D}^j, j \in \{1, \dots d \}$, approximating the partial derivative in the $j$-th coordinate direction; they are exact for polynomials of degree $\leq p$.
        \item A mass matrix $\mtx{M}$, approximating the $L_2$ scalar product on $E$.
        \item A restriction operator $\mtx{R}$, that projects the solution to the boundary nodes.
        \item A boundary mass matrix $\mtx{B}$ approximating the $L_2$ scalar product on $\partial E$.
        \item Multiplication operators $\mtx{N}^j$ for $j \in \{1, \dots d \}$, performing multiplication of functions on the boundary $\partial E$ by the $j$-th component $n_j$ of the outer normal vector $\vec{n}$ at the boundary $\partial E$.
    \end{itemize}
    Together, these components have to satisfy the SBP property
    \begin{equation}\label{sbp_property}
        \mtx{M} \, \mtx{D}^j + (\mtx{D}^j)^T \mtx{M} = \mtx{R}^T \mtx{B} \, \mtx{N}^j \mtx{R}.
    \end{equation}
\end{definition}

We consider a nonconservative hyperbolic system in $d$ dimensions
\begin{equation}
\partial_t \vec{u} + \sum_{j=1}^d \partial_{x_j} \vec{f}^j(\vec{u}) + \sum_{j=1}^d \mtx{H}^j (\vec{u}) \partial_{x_j} \vec{g}^j(\vec{u}) = 0.
\label{nonconservative_system_d}
\end{equation}
An SBP semi-discretization of~\eqref{nonconservative_system_d} is given by
\begin{equation}
    \partial_t \vec{u} = - \vec{\mathrm{VOL}} - \vec{\mathrm{SURF}},
    \label{semi_SBP}
\end{equation}
where $\vec{\mathrm{VOL}}$ contains the volume terms consistent with $\partial_{x_j} \vec{f}^j(\vec{u}) + \mtx{H}^j \cdot \partial_{x_j} \vec{g}^j(\vec{u})$ and $\vec{\mathrm{SURF}}$ contains additional surface terms coupling neighboring elements.
We assume the nodal basis to always admit sufficient boundary nodes, e.g., tensor product Gauss-Lobatto-Legendre nodes in quadrilateral/hexahedral elements.

\subsection{Nonconservative split forms}
In this section, we present the semi-discretization of the volume terms for high-order SBP operators for the nonconservative terms, as a natural extension of the finite volume methods presented in Section~\ref{sec:nonconservative_systems}.
Gassner et al.~\cite{gassner2016split} showed how the flux-differencing framework introduced by Fisher and Carpenter~\cite{FISHER2013518} generates common split forms formulations such as~\cite{KENNEDY20081676, ducros2000}, given specific choices of the numerical volume flux.
A systematic way to generate and implement split forms is of great interest for high-order discretizations based on SBP operators.
In fact, a simple way to obtain entropy-preserving methods is to prove the same property at the continuous level using only integration by parts, and then replace derivatives by SBP operators (see for example~\cite{RanochaSGN2025,ricardo2024thermodynamicconsistencystructurepreservationsummation,bleecke2025asymptoticpreservingenergyconservingmethodshyperbolic,gassner2016shallowwater}).
This is often far from trivial, or it may require non-trivial splittings~\cite{DEMICHELE2025114262}.
Moreover, split forms can improve the robustness of under-resolved simulations even without provable entropy stability properties~\cite{KENNEDY20081676,gassner2016split, Shima2021, sjogreen2017skew, klose2020assessing}.

The first discretization~\eqref{FV_form1} is based on a pointwise evaluation of the nonconservative product.
A natural SBP generalization is
\begin{equation}
    \vec{H}^j(\vec{u}(x_i)) \cdot \partial_j \vec{g}^j(\vec{u}(x)) \Big |_{x = x_i}
    \approx
    \sum_{k = 1}^N \mtx{D}^j_{i,k} \vec{H}^j(\vec{u}_i) \left (\vec{g}^j(\vec{u}_k) - \vec{g}^j(\vec{u}_i) \right ),
    \label{volume_noncons1}
\end{equation}
where the term $\vec{g}^j(\vec{u}_i)$ does not contribute to the sum, since it is constant with respect to the summation index $k$ and thus mapped to zero by the consistent SBP operator $\mtx{D}^j$. However, this term is included to maintain the jump term in the volume flux, similarly to the approach used in the finite volume scheme for $\vec{H}_i\jump{\vec{g}}$, which we will later justify.
The second discretization~\eqref{FV_form2} is analogous to the flux-differencing form for conservative terms, i.e.,
\begin{equation}
    \vec{H}^j(\vec{u}(x_i)) \cdot \partial_j \vec{g}^j(\vec{u}(x)) \Big |_{x = x_i}
    \approx
    \sum_{k = 1}^N 2 \mtx{D}^j_{i,k} \vec{H}^j(\vec{u}_i) \vec{g}^{\mathrm{vol},j}(\vec{u}_i, \vec{u}_k),
    \label{volume_noncons2}
\end{equation}
where we denoted the numerical flux approximating $\vec{g}^j$ as $\vec{g}^{\mathrm{vol},j}$ instead of $\vec{g}^{\mathrm{num},j}$ to emphasize its role in the volume terms.
The third discretization~\eqref{FV_form3} extends naturally to high-order SBP operators via flux differencing, i.e., the nonconservative volume terms read~\cite{RUEDARAMIREZ2024112607, artiano2025structurepreservinghighordermethodscompressible}
\begin{equation}
    \vec{H}^j(\vec{u}(x_i)) \cdot \partial_j \vec{g}^j(\vec{u}(x)) \Big |_{x = x_i}
    \approx
    \sum_{k = 1}^N \mtx{D}^j_{i,k} \vec{H}^{\mathrm{vol},j}\left (\vec{u}_i, \vec{u}_k \right ) \left (\vec{g}^{j}(\vec{u}_k) - \vec{g}^{j}(\vec{u}_i) \right ).
\label{volume_noncons3}
\end{equation}
The fourth discretization~\eqref{FV_form4} includes a conservative term that we discretize via flux differencing as usual, while the second term is similar to the second form~\eqref{volume_noncons2}, i.e.,
\begin{equation}
    \vec{H}^j(\vec{u}(x_i)) \cdot \partial_j \vec{g}^j(\vec{u}(x)) \big |_{x = x_i}
    \approx
    \sum_{k = 1}^N 2 \mtx{D}^j_{i,k} \left ( \vec{H}\vec{g}\right )^{\mathrm{vol},j}(\vec{u}_i, \vec{u}_k) - 2 \mtx{D}^j_{i,k} \vec{H}^{\mathrm{vol},j}(\vec{u}_i,\vec{u}_k) \vec{g}^j(\vec{u}_i).
\label{volume_noncons4}
\end{equation}
Following~\cite{gassner2016split}, we show these forms presented are naturally related to split forms of the nonconservative product or the discrete equivalent of the product rule for the nonconservative product, i.e,
\begin{equation}
h(u) \partial_x g(u) = \frac{1}{2} \Bigl ( \partial_x \bigl (h(u)g(u) \bigr ) +h(u) \partial_x g(u) - g(u) \partial_x h(u) \Bigr ),
\label{split_form_noncons_continuous}
\end{equation}
and
\begin{equation}
h(u) \partial_x g(u) = \partial_x (h(u)g(u)) - g(u) \partial_x h(u),
\label{discrete_rule_noncons}
\end{equation}
where we dropped the vector-matrix notation, to work with scalar functions $h(u)$ and $g(u)$, i.e. $m = n = 1$, for the sake of simplicity.
We will restrict our analysis to quadratic terms generated by $\partial_x  ( hg  )$.
The first and second discretization for $g^{\mathrm{vol}} = \mean{g}$ are the least interesting as they are equivalent to the strong form.
We encapsulate the different split-forms in the following two lemmas.
\begin{lemma}[Discrete nonconservative split-forms]
    Using the arithmetic mean in the volume fluxes, it is possible to recover the strong form for the first and second forms and the split form~\eqref{split_form_noncons_continuous} for the third volume discretization~\eqref{volume_noncons3}.
    Specifically, we have
    \begin{equation}
    \begin{aligned}
        \sum_{k=1}^N \mtx{D}_{i,k} h_i \jump{g}_{i,k}
        &=
        (\mtx{H} \mtx{D} \vec{g})_i,
        \\
        \sum_{k=1}^N 2 \mtx{D}_{i,k} h_i g^{\mathrm{vol}}(u_i,u_k)
        =
        \sum_{k=1}^N 2\mtx{D}_{i,k}H_i \mean{g}_{i,k}
        &=
        (\mtx{H} \mtx{D} \vec{g})_i,
        \\
        \sum_{k=1}^N  \mtx{D}_{i,k} h^{\mathrm{vol}} (u_i,u_k) \jump{g}_{i,k}
        =
        \sum_{k=1}^N \mtx{D}_{i,k} \mean{h}_{i,k} \jump{g}_{i,k}
        &=
        \frac{1}{2} \left (\mtx{D} \mtx{H} \vec{g} + \mtx{H}  \mtx{D} \vec{g} - \mtx{G} \mtx{D} \vec{h} \right )_i,
    \end{aligned}
    \end{equation}
    where $\vec{a} = (a_1, \dots, a_N)$, $\mtx{A} = \text{diag}(\vec{a})$, and
    \begin{equation}
        \begin{aligned}
        \mean{a}_{i,k} := \frac{a_k + a_i}{2}, \quad \jump{a}_{i,k} := a_k - a_i.
        \end{aligned}
    \end{equation}
\end{lemma}
\begin{proof}
The proof follows from direct calculation and the consistency of the derivative operator with respect to constants, i.e., $\sum_{k=1}^{N} \mtx{D}_{ik} a_i = 0, \forall a_i \in \mathbb{R}$.
\end{proof}

\begin{lemma}
    \label{lemma:split_form_fourth_form}
Using the arithmetic mean, the product of the arithmetic mean and the product mean in the volume flux of the fourth form~\eqref{volume_noncons4}, it is possible to recover the discrete product rule~\eqref{discrete_rule_noncons}, the split form~\eqref{split_form_noncons_continuous} and the strong form, i.e.,
\begin{equation}
    \begin{aligned}
    \sum_{k=1}^N  2 \mtx{D}_{i,k} \mean{hg}_{i,k} - 2 \mtx{D}_{i,k} \mean{h}_{i,k} g_i &= \left ( \mtx{D} \mtx{H} \vec{g} - \mtx{G} \mtx{D} \vec{H} \right )_i,\\
    \sum_{k=1}^N  2 \mtx{D}_{i,k} \mean{h}_{i,k}\mean{g}_{i,k} - 2 \mtx{D}_{i,k} \mean{h}_{i,k} g_i &= \frac{1}{2} \left (\mtx{D} \mtx{H} \vec{g} + \mtx{H}  \mtx{D} \vec{g} - \vec{G} \mtx{D} \mtx{h} \right )_i,\\
    \sum_{k=1}^N  2 \mtx{D}_{i,k} \prodmean{h\cdot g}_{i,k} - 2 \mtx{D}_{i,k} \mean{h}_{i,k} g_i &= \left ( \mtx{H} \mtx{D} \vec{g} \right )_i.
    \end{aligned}
\end{equation}
\end{lemma}
\begin{proof}
The proof follows the same arguments as the previous lemma.
\end{proof}
\begin{remark}
The fourth form~\eqref{volume_noncons4} is also equivalent to the linear combination of the first and the third forms with coefficient $\alpha = 2/3$.
\end{remark}
Lemma~\ref{lemma:split_form_fourth_form} suggests that the fourth form~\eqref{volume_noncons4} is the most general formulation, as it recovers the previous ones for some particular choices of the volume fluxes.
However, in the construction of novel numerical fluxes in Section~\ref{sec:numerical_fluxes}, as shown previously, such a generalization was not required and would only add unnecessary complexity to the their construction.
A second important reason for not adopting this fourth form is that, in all cases considered in Section~\ref{sec:numerical_fluxes}, the numerical fluxes associated to $\vec{g}(\vec{u})$ always reduce to simple arithmetic means.
Consequently, the fourth form reduces to the third form, and therefore a linear combination of the first and third form already covers all the cases studied.
Moreover, in the case where $g = \text{const}$ one must ensure that $(\vec{H}\vec{g})^{\mathrm{vol}} - \vec{H}^{\mathrm{vol}}\vec{g}_i = 0$, which imposes a restriction on the choice of $(\vec{H}\vec{g})^{\mathrm{vol}}$.
We have summarized the different identities in Table~\ref{table_identities_splitforms}.
\begin{table}[htbp]
\centering
  \caption{Equivalence of different SBP semi-discretization of the volume terms for different forms and volume fluxes.}
  \label{table_identities_splitforms}
  \centering
  \footnotesize
    \begin{tabular}{c rll}
      \toprule
      \multicolumn{1}{c}{Form}
      & \multicolumn{1}{c}{Volume fluxes}
      & \multicolumn{1}{c}{Volume terms}
      & \multicolumn{1}{c}{Equivalent strong form}\\
      \midrule
1~\eqref{volume_noncons1} & ---\hspace{0.2cm}
& $\mtx{D}_{i,k} h_i \jump{g}_{i,k}$
& $\vec{H} \vec{D} \vec{g}$
\\
2~\eqref{volume_noncons2} & $g^{\mathrm{vol}} = \mean{g}$
& $\mtx{D}_{i,k} h_i \mean{g}_{i,k}$
& $\vec{H} \mtx{D} \vec{g}$
\\
3~\eqref{volume_noncons3} & $h^{\mathrm{vol}} = \mean{h}$
& $\mtx{D}_{i,k} \mean{h}_{i,k}\jump{g}_{i,k}$
& $\frac{1}{2} \left (\mtx{D} \vec{H} \vec{g} + \vec{H}  \mtx{D} \vec{g} - \vec{G} \mtx{D} \vec{h} \right )$
\\
4~\eqref{volume_noncons4} & $\left ( h g \right )^{\mathrm{vol}} = \mean{hg}$, $h^{\mathrm{vol}} = \mean{h}$
& $2 \mtx{D}_{i,k} \mean{hg}_{i,k} - 2 \mtx{D}_{i,k} \mean{h}_{i,k} g_i$
& $\mtx{D} \vec{H}\vec{g} - \vec{G} \mtx{D} \vec{h}$    \\
4~\eqref{volume_noncons4} & $\left ( h g \right )^{\mathrm{vol}} = \mean{h}\mean{g}$, $h^{\mathrm{vol}} = \mean{h}$
& $2 \mtx{D}_{i,k} \mean{h}\mean{g}_{i,k} - 2 \mtx{D}_{i,k} \mean{h}_{i,k} g_i$
& $\frac{1}{2} \left (\mtx{D} \vec{H} \vec{g} + \vec{H}  \mtx{D} \vec{g} - \vec{G} \mtx{D} \vec{h} \right )$\\
4~\eqref{volume_noncons4} & $\left ( h g \right )^{\mathrm{vol}} = \prodmean{h \cdot g}$, $h^{\mathrm{vol}} = \mean{h}$
& $2 \mtx{D}_{i,k} \prodmean{h \cdot g}_{i,k} - 2 \mtx{D}_{i,k} \mean{h}_{i,k} g_i$
& $\vec{H} \mtx{D} \vec{g}$
\\
      \bottomrule
    \end{tabular}
\end{table}

\subsubsection{Order of accuracy for nonconservative volume terms}
\begin{theorem}
    If the volume fluxes $(\vec{H}\vec{g})^{\mathrm{vol},j}$, $\vec{H}^{\mathrm{vol},j}$, $\vec{g}^{\mathrm{vol},j}$ are smooth and symmetric and the functions $\vec{H}^j$ and $\vec{g}^j$ are smooth, the forms~\eqref{volume_noncons2}, \eqref{volume_noncons3} and \eqref{volume_noncons4} are an approximation of the same order of accuracy as the derivative matrices $\mtx{D}^j$. Here the derivative matrices do not need to be given by SBP operators.
\end{theorem}
\begin{proof}
    It suffices to consider the one-dimensional scalar case. For the sake of brevity and simplicity we will report the proof only for the third form~\eqref{volume_noncons3}.
    Note that the conservative terms appearing in the second and fourth form can be treated as standard conservative flux differencing terms, e.g.,~\cite{chen2017entropy} or~\cite[Lemma 3.17]{ranocha2018thesis}.
    Assuming the derivative operator is of order $p$, then
    \begin{equation}
        \begin{aligned}
        & \sum_{k = 1}^N \mtx{D}_{ik} h^{\mathrm{vol}}\left ( u_i, u_k \right ) \left (g_k - g_i \right ) = \sum_{k=1}^N \mtx{D}_{ik} h^{\mathrm{vol}}\left (u(x_i), u(x_k) \right ) \left (g(x_k) - g(x_i) \right )\\
        =& \frac{\partial h^{\mathrm{vol}}\left (u(x_i), u(x_k) \right ) g\left ( x_k \right )}{\partial x_k} \Bigg |_{x_k = x_i} - g(x_i) \frac{\partial h^{\mathrm{vol}}\left ( u(x_i), u(x_k)\right )}{\partial x_k}\Bigg |_{x_k = x_i} + \mathcal{O}\left (\Delta x^p \right )\\
        =& h^{\mathrm{vol}}\left (u(x_i), u(x_i) \right ) \frac{\partial g\left ( x_k \right )}{\partial x_k} \Bigg |_{x_k = x_i} + \mathcal{O}\left (\Delta x^p \right ) = h\left ( u  (x_i) \right ) \frac{\partial g\left ( u(x_k) \right )}{\partial x_k} \Bigg |_{x_k = x_i} + \mathcal{O}\left (\Delta x^p \right ),
        \end{aligned}
    \end{equation}
    which justifies the absence of the factor 2 compared to the conservative case.
\end{proof}

\subsection{SBP semi-discretization of nonconservative hyperbolic system}
The main components of the SBP discretization can be written as follows.
When nonconservative terms are present, the volume term takes the form

\begin{multline}
    \label{volume_terms}
\vec{\mathrm{VOL}}_i :=  \sum_{j = 1}^d \Biggl ( \sum_{k = 1}^N 2 \mtx{D}^j_{i,k} \vec{f}^{\mathrm{vol},j}\left (\vec{u}_i, \vec{u}_k \right ) + \alpha \sum_{k = 1}^N \mtx{D}^j_{i,k} \mtx{H}^{\mathrm{vol},j}\left (\vec{u}_i, \vec{u}_k \right ) \left (\vec{g}^j(\vec{u}_k) - \vec{g}^j(\vec{u}_i) \right ) \\+(1-\alpha)\sum_{k = 1}^N \mtx{D}^j_{i,k} \mtx{H}^j\left (\vec{u}_i\right ) \left (\vec{g}^j(\vec{u}_k) - \vec{g}^j(\vec{u}_i) \right ) \Biggr )
\end{multline}
and the surface term reads
\begin{equation}
    \vec{\mathrm{SURF}}
    :=
    \mtx{M}^{-1}\mtx{R}^T \left (
        \mtx{B}\ \vec{f}^{\mathrm{num}}
        + \frac{\alpha}{2}|\mtx{B}| \mtx{H}^\mathrm{num} \jump{\vec{g}}
        + \frac{1-\alpha}{2}|\mtx{B}| \mtx{H} \jump{\vec{g}} \right )
    - \sum_{j = 1}^d \mtx{M}^{-1} \mtx{R}^T \mtx{B} \mtx{N_j} \mtx{R}\ \vec{f}^j,
    \label{surface_terms}
\end{equation}
where
\begin{equation}
    \begin{aligned}
\vec{f}^{\mathrm{num}}(\vec{u}_-, \vec{u}_+) &= \sum_{j=1}^{d} \vec{n}_j \vec{f}^{\mathrm{num},j}(\vec{u}_-, \vec{u}_+),\\
\mtx{H}^{\mathrm{num}}(\vec{u}_-, \vec{u}_+)\jump{\vec{g}} &= \sum_{j=1}^{d} \vec{n}_j \mtx{H}^{\mathrm{num},j}(\vec{u}_-, \vec{u}_+) \jump{\vec{g}^j},\\
\mtx{H}(\vec{u}_0)\jump{\vec{g}} &= \sum_{j=1}^{d} \vec{n}_j \mtx{H}^{j}(\vec{u}_0) \jump{\vec{g}^j}.
\end{aligned}
\end{equation}
For the $E$-th element, the flux at the right boundary is computed as $\vec{f}^{\mathrm{num}}(\vec{u}_- = \vec{u}_N^E, \vec{u}_+ = \vec{u}_1^{E+1})$, where the subscripts $1$ and $N$ here correspond to the last and first node of the generic element $E$.
\begin{remark}
For a polynomial degree $p$ of order zero, the volume terms vanish, since the derivative is exact for constants, i.e., $\mtx{D} \vec{1} = \vec{0}$, recovering the FV method presented in the previous section.
\end{remark}

\subsection{Entropy conservation/stability}

Following the results in the previous section, we introduce a multidimensional definition of entropy conservation/stability for nonconservative systems.
\begin{definition}\label{Convex_d}
Two-point numerical fluxes $\vec{f}^{\mathrm{num},j}$ and $\mtx{H}^{\mathrm{num},j}$ in the $j$-th direction for the hyperbolic nonconservative system~\eqref{nonconservative_system_d} with $\alpha \in \mathbb{R}$ are entropy-conservative/stable if
\begin{equation}
\label{TadmorDef_convex_d}
	\jump{\vec{\omega}}\cdot \vec{f}^{\mathrm{num},j} -\alpha \mean{\vec{\omega}}\cdot \vec{H}^{\mathrm{num},j}\jump{\vec{g}^j} -(1-\alpha) \mean{\vec{\omega} \cdot \vec{H}^j}\jump{\vec{g}^j} \leq \jump{\psi^j}, \quad j = 1, \dots, d,
\end{equation}
where $U$ is the entropy function, $\vec{\omega} = U'(\vec{u})$ the entropy variables, and $\psi^j$ is the entropy potential in the $j$-th direction.
\end{definition}

\begin{theorem}
    Consider the semi-discretization~\eqref{semi_SBP} with the volume and surface terms given by \eqref{volume_terms} and \eqref{surface_terms}.
    If $\alpha$, $\vec{f}^{\mathrm{vol}}$, $\vec{H}^{\mathrm{vol}}$, that are consistent with $\vec{f}$ and $\vec{H}$, symmetric and entropy-conservative as in Definition~\ref{Convex_d}, and both the mass matrix $\mtx{M}$ and the boundary operators $\mtx{R}^T \mtx{B} \mtx{N}^j \mtx{R}$ are diagonal, the semi-discretization is entropy-conservative/stable, if $\alpha$ and the numerical fluxes $\vec{f}^{\mathrm{num}}$, $\vec{H}^{\mathrm{num}}$ are entropy-conservative/stable according to Definition~\ref{Convex_d}.
\end{theorem}
\begin{proof}
To derive the semi-discrete rate of change of the entropy within an element $\frac{d}{dt} \int_{\Omega}U$, we contract to the entropy space by multiplying the left-hand side by $\vec{\omega} \mtx{M} \partial_t \vec{u}$.
The proof is an extension of the proof of~\cite{Ranocha2018,FISHER2013518,chen2017entropy}, including here the nonconservative terms, and it represents a particular case of the specified fluctuation form of the proof presented in~\cite{WARUSZEWSKI2022111507}.
We first consider the volume terms, assuming $\alpha = 1$ for simplicity (since the general case follows similarly),
\begin{equation}
\begin{aligned}
  & \sum_{j = 1}^d \left ( \sum_{i, k = 1}^N 2 \vec{\omega}_i \vec{M}_{i,i} \mtx{D}^j_{i,k} \vec{f} ^{\mathrm{vol},j}\left (\vec{u}_i, \vec{u}_k \right )
  + \sum_{i, k = 1}^N \vec{\omega}_i \vec{M}_{i,i} \mtx{D}^j_{i,k} \vec{H}^{\mathrm{vol},j}\left (\vec{u}_i, \vec{u}_k \right ) \left (\vec{g}^j(\vec{u}_k) - \vec{g}^j(\vec{u}_i) \right ) \right ) \\
  & =  \sum_{j = 1}^d \Bigg ( \sum_{i, k = 1}^N \vec{\omega}_i \left [ \mtx{M} \mtx{D}^j + \mtx{R}^T \mtx{B} \mtx{N}^j \mtx{R} \mtx{D}^j - \mtx{D}^j_{i,k} \mtx{M} \right ]_{i,k} \vec{f} ^{\mathrm{vol},j}\left (\vec{u}_i, \vec{u}_k \right ) \\
  &\quad + \sum_{i,k = 1}^N \frac{\vec{\omega}_i}{2} \left [ \mtx{M} \mtx{D}^j + \mtx{R}^T \mtx{B} \mtx{N}^j \mtx{R} \mtx{D}^j - \mtx{D}^j \mtx{M} \right ]_{i,k} \vec{H}^{\mathrm{vol}}\left (\vec{u}_i, \vec{u}_k \right ) \left (\vec{g}^j(\vec{u}_k) - \vec{g}^j(\vec{u}_i) \right ) \Bigg ).
\end{aligned}
\end{equation}
Since the mass matrices are diagonal in each direction $j \in \{1, .., d \}$,
\begin{small}
\begin{equation}
    \begin{aligned}
& \sum_{i,k=1}^N \vec{\omega}_i \left[  \mtx{M} \mtx{D}^j -  (\mtx{D}^j)^T \mtx{M}\right]_{i,k}  \vec{f}^{\mathrm{vol},j}\left (\vec{u}_i, \vec{u}_k \right ) + \sum_{i,k=1}^N\frac{\vec{\omega}_i}{2}  \left[  \mtx{M} \mtx{D}^j -  (\mtx{D}^j)^T \mtx{M}\right]_{i,k} \vec{H}^{\mathrm{vol},j}(\vec{u}_i,\vec{u}_k)\left (\vec{g}^j(\vec{u}_k) - \vec{g}^j(\vec{u}_i) \right ) \\
 = &\sum_{i,k}^N  \left ( \mtx{M}_{i,i} \mtx{D}^j_{i,k} -  \mtx{D}^j_{k,i} \mtx{M}_{k,k} \right ) \vec{\omega}_i \vec{f}^{\mathrm{vol},j}\left (\vec{u}_i, \vec{u}_k \right ) +  \sum_{i,k=1}^N\frac{\vec{\omega}_i}{2} \left ( \mtx{M}_{i,i} \mtx{D}^j_{i,k} -  \mtx{D}^j_{k,i} \mtx{M}_{k,k} \right ) \vec{H}^{\mathrm{vol},j}(\vec{u}_i,\vec{u}_k)\left (\vec{g}^j(\vec{u}_k) - \vec{g}^j(\vec{u}_i) \right ) \\
 = & \sum_{i,k}^N \mtx{M}_{i,i} \mtx{D}^j_{i,k} \left (\vec{\omega}_i - \vec{\omega}_k \right ) \vec{f}^{\mathrm{vol},j}\left (\vec{u}_i, \vec{u}_k \right ) +  \sum_{i,k=1}^N \mtx{M}_{i,i} \mtx{D}^j_{i,k}  \frac{\vec{\omega}_i + \vec{\omega}_k}{2}  \vec{H}^{\mathrm{vol},j}(\vec{u}_i,\vec{u}_k)\left (\vec{g}^j(\vec{u}_k) - \vec{g}^j(\vec{u}_i) \right ),
    \end{aligned}
\end{equation}
\end{small}%
where we exchanged the indices $i,k$ in the last step using the symmetry of the volume fluxes $\vec{f}^{\mathrm{vol},j}$ and $\vec{H}^{\mathrm{vol},j}$ and the anti-symmetry of the $\vec{H}^{\mathrm{vol},j} \jump{\vec{g}^j}$.
Furthermore, with diagonal matrices $\mtx{R}^T \mtx{B} \mtx{N}^j \mtx{R}$ and the anti-symmetry of $\jump{\vec{g}^j}$,
\begin{equation}
    \sum_{i,k = 1}^N \frac{\vec{\omega}_i}{2} \left [ \mtx{R}^T\mtx{B} \mtx{N}^j \mtx{R} \right ]_{i,k} \vec{H}^{\mathrm{vol},j}\left(\vec{u}_i, \vec{u}_k \right) \left (\vec{g}^j(\vec{u}_k) - \vec{g}^j(\vec{u}_i) \right )  = 0
\end{equation}
and
\begin{equation}
\sum_{i,k = 1}^N \vec{\omega}_i \left [ \mtx{R}^T\mtx{B} \mtx{N}^j \mtx{R} \right ]_{i,k} \vec{f}^{\mathrm{vol},j}\left(\vec{u}_i, \vec{u}_k \right) = \sum_{k = 1}^N \left [ \mtx{R}^T\mtx{B} \mtx{N}^j \mtx{R} \right ]_{k,k} \vec{\omega}_k \vec{f}^{\mathrm{vol},j}\left(\vec{u}_k, \vec{u}_k \right).
\end{equation}
Thus, the volume contribution to the entropy rate within an element is
\begin{multline}
\vec{\omega}^T \mtx{M} \mtx{\mathrm{VOL}}^j
= \sum_{i,k}^N \mtx{M}_{i,i} \mtx{D}^j_{i,k} \left (\vec{\omega}_i - \vec{\omega}_k \right ) \vec{f}^{\mathrm{vol},j}\left (\vec{u}_i, \vec{u}_k \right )
\\
+  \sum_{i,k=1}^N \mtx{M}_{i,i} \mtx{D}^j_{i,k}  \frac{\vec{\omega}_i + \vec{\omega}_k}{2}  \vec{H}^{\mathrm{vol},j}(\vec{u}_i,\vec{u}_k)\left (g(\vec{u}_k) - g(\vec{u}_i) \right )
+ \sum_{k = 1}^N\left [ \mtx{R}^T\mtx{B} \mtx{N}^j \mtx{R} \right ]_{k,k} \vec{\omega}_k \vec{f}^{\mathrm{vol},j}\left(\vec{u}_k, \vec{u}_k \right).
    \label{step_cartesian}
\end{multline}
Using the EC Definition~\ref{Convex_d} for $\alpha = 1$, and $j \in \{1,\ldots,d\}$, the volume terms become
\begin{equation}
\begin{aligned}
\vec{\omega} \mtx{M}\  \mtx{\mathrm{VOL}}^j &= \sum_{i,k = 1}^N \mtx{M}_{i,i} \mtx{D}^j_{i,k} \left (\psi^j_i - \psi^j_k \right ) + \sum_{k = 1}^N\left [ \mtx{R}^T\mtx{B} \mtx{N}^j \mtx{R} \right ]_{k,k} \vec{\omega}_k \vec{f}^{\mathrm{vol},j}\left(\vec{u}_k, \vec{u}_k \right) \\
 &= - \sum_{i,k = 1}^N \left [ \mtx{M}_{i,i} \mtx{D}^j\right ]_{i,k} \psi^j_k + \sum_{k = 1}^N\left [ \mtx{R}^T\mtx{B} \mtx{N}^j \mtx{R} \right ]_{k,k} \vec{\omega}_k \vec{f}^{\mathrm{vol},j}\left(\vec{u}_k, \vec{u}_k \right) \\
&= - \sum_{i,k = 1}^N \left [ \mtx{R}^T \mtx{B} \mtx{N}^j \mtx{R} - (\mtx{D}^j)^T \mtx{M} \right ]_{i,k} \psi^j_k + \sum_{k = 1}^N\left [ \mtx{R}^T\mtx{B} \mtx{N_j} \mtx{R} \right ]_{k,k} \vec{\omega}_k \vec{f}^{\mathrm{vol},j}\left(\vec{u}_k, \vec{u}_k \right) \\
& = - \sum_{i,k = 1}^N \left [ \mtx{R}^T\mtx{B} \mtx{N}^j \mtx{R} \right ]_{i,k} \psi^j_k + \sum_{k = 1}^N\left [ \mtx{R}^T\mtx{B} \mtx{N}^j \mtx{R} \right ]_{k,k} \vec{\omega}_k \vec{f}^{\mathrm{vol},j}\left(\vec{u}_k, \vec{u}_k \right).
\end{aligned}
\end{equation}
Hence, the semi-discrete rate of change of the entropy in one element becomes
\begin{equation}
\vec{\omega} \cdot \mtx{M} \partial_t \vec{u} = \sum_{j=1}^{d}  \sum_{i,k = 1}^N \left [ \mtx{R}^T\mtx{B} \mtx{N}^j \mtx{R} \right ]_{i,k} \psi^j_k -\vec{\omega} \cdot \mtx{R}^T \left (\mtx{B} \vec{f}^{\mathrm{num}} +  \frac{\alpha}{2}|\mtx{B}| \mtx{H}^{num} \jump{\vec{g}} + \frac{1-\alpha}{2}|\mtx{B}| \mtx{H} \jump{\vec{g}} \right ).
\end{equation}
These surface terms can be treated in each direction $j$ independently, by adding and subtracting $\psi^j_N$ from the left element and $\psi^j_1$ from the right element, similarly to Lemma~\ref{algebraic_lemma_nc_form1}. Because of the entropy-conservative/stable Definition~\ref{Convex_d}, we finally obtain, without the assumption of $\alpha =1$
\begin{equation}
\vec{\omega} \cdot \mtx{M} \partial_t \vec{u} \leq -\mtx{1}^T \mtx{R}^T \mtx{B} \vec{F}^{\mathrm{num}}
\end{equation}
where
\begin{equation}
    \begin{aligned}
\vec{F}^{\mathrm{num}}(\vec{u}_+, \vec{u}_-) &= \sum_{j=1}^d \vec{n}_j \vec{F}^{\mathrm{num},j}\\
\vec{F}^{\mathrm{num},j}(\vec{u}_+, \vec{u}_-) &= \mean{\vec{\omega}} \cdot \vec{f}^{\mathrm{num},j}  -\mean{\vec{\omega} \cdot \vec{f}^j} +  \mean{F^j} - \frac{\alpha}{4}\jump{\vec{\omega}} \cdot \vec{H}^{\mathrm{num},j}\jump{\vec{g}^j} - \frac{1-\alpha}{4}\jump{\vec{\omega} \cdot \vec{H}^j} \jump{\vec{g}^j}.
    \end{aligned}
\end{equation}
\end{proof}

\begin{remark}
    From the implementation point of view, the presence of antisymmetric terms in the volume nonconservative terms (due to the presence of $\jump{\vec{g}}$) can be exploited to reduce the computational cost by including the action of $\sum_{j = 1}^d \mtx{M}^{-1} \mtx{R}^T \mtx{B} \mtx{N}^j \mtx{R}\ \vec{f}^j$ into the derivative operator, as already shown for the conservative terms in~\cite{ranocha2023efficient}, by defining
    \begin{equation}
    \tilde{\mtx{D}}^j = 2 \mtx{D}^j - \mtx{M}^{-1} \mtx{R}^T \mtx{B} \mtx{N}^j \mtx{R}.
    \end{equation}
    Since the nonconservative terms are written in anti-symmetric form, their contribution to the diagonal of the derivative operator vanishes, i.e., they are consistent with zero. Therefore, it is sufficient to compute only either the upper or lower triangle part of the derivative operator.
\end{remark}

\section{Summation-by-parts operators and curvilinear coordinates}
\label{sec:sbp_curvilinear}
To extend the method to curved meshes within the SBP framework, we follow the classical approach described by Kopriva~\cite{Kopriva2006,Kopriva2009}.
For the sake of clarity, we restrict the presentation to 2D.
The extension to 3D is conceptually straightforward, requiring only additional care in the discretization of the metric terms in order to ensure the free-stream preservation property~\cite{vinokur2002extension, Kopriva2006, Kopriva2009}.

In two space dimensions a generic hyperbolic nonconservative system reads
\begin{equation}
\partial_t \vec{u} + \partial_x \vec{f}^1(\vec{u}) + \partial_y \vec{f}^2(\vec{u}) + \vec{H}^1(\vec{u}) \partial_x \vec{g}^1(\vec{u}) + \vec{H}^2(\vec{u}) \partial_y \vec{g}^2(\vec{u}) = 0.
\end{equation}
The computational domain $\Omega$ is divided into $K$ non-overlapping quadrilateral elements $E_k$, $k = 1, \dots, K$.
Each physical element is obtained from the reference element $E_0 = [-1,1]^2$ through a mapping
\begin{equation}
x = X(\xi, \eta), \quad y = Y(\xi, \eta), \quad (\xi, \eta) \in E_0.
\end{equation}
If the edges of an element are straight, the mapping corresponds to a bilinear transformation and is linear with respect to each reference coordinate.
When the elements have curved boundaries, the geometry is described using high-order polynomials~\cite{Kopriva2006, Kopriva2009}.
As a result, the mapping itself becomes a polynomial in each direction.
In such cases, we model the geometry through a polynomial mapping of degree $p_{\mathrm{geo}}$ in each coordinate direction.

Under this transformation, the nonconservative hyperbolic system reads
\begin{equation}
J \partial_t \vec{u} + \partial_\xi \tilde{\vec{f}}^1 + \partial_\eta \tilde{\vec{f}}^2 + \tilde{\vec{H}} \partial_\xi \tilde{\vec{g}}^1 + \tilde{\vec{H}} \partial_\eta \tilde{\vec{g}}^2 = \vec{0},
\end{equation}
where $J = X_\xi Y_\eta - X_\eta Y_\xi$ is the Jacobian of the transformation, and the $\tilde{\cdot}$ quantities are the contravariant counterparts, i.e.,
\begin{equation}
    \tilde{\vec{f}}^j = \hat{\vec{n}}^j \cdot \vec{f}, \quad
    \tilde{\vec{g}}^j = \hat{\vec{n}}^j \circ \vec{g}, \quad
    \tilde{\vec{H}} = \left (\vec{H}^1, \vec{H}^2 \right ),
\end{equation}
where $\circ$ denotes the Hadamard (element-wise) product, $\vec{f} = (\vec{f}^1, \vec{f}^2)^T$, $\vec{g} = (\vec{g}^1, \vec{g}^2)^T$, and $\hat{\vec{n}}^j$ are the contravariant vectors (normal directions to the coordinate lines) defined as
\begin{equation}
    \hat{\vec{n}}^1 = \begin{pmatrix} Y_\eta \\ -X_\eta \end{pmatrix}, \quad
    \hat{\vec{n}}^2 = \begin{pmatrix} -Y_\xi \\ X_\xi \end{pmatrix},
\end{equation}
where $X_\xi, X_\eta, Y_\xi, Y_\eta$ are the metric terms.
They satisfy the metric identities~\cite{vinokur2002extension, Kopriva2006, Kopriva2009}
\begin{equation}
\partial_\xi Y_\eta - \partial_\eta Y_\xi = 0, \quad -\partial_\xi X_\eta + \partial_\eta X_\xi = 0,
\label{metric_identities}
\end{equation}
ensuring that conservation laws are preserved under the coordinate transformation.
The SBP operators are all defined on the reference element $E_0$ and the metric terms are approximated as
\begin{equation}
X_\xi \approx \mtx{D}^1 \vec{X} = \vec{X}_\xi, \quad X_\eta \approx \vec{X} \left ( \mtx{D}^2  \right )^T = \vec{X}_\eta, \quad Y_\xi \approx \mtx{D}^1 \vec{Y} = \vec{Y}_\xi, \quad Y_\eta \approx \vec{Y} \left ( \mtx{D}^2 \right )^T = \vec{Y}_\eta,
\end{equation}
which guarantees that the free-stream preservation (FSP)~\cite{VisbalGaitonde1999, Kopriva2006} property holds also at the discrete level, i.e.,
\begin{equation}
    \begin{aligned}
    \partial_\xi Y_\eta - \partial_\eta Y_\xi &\approx \mtx{D}^1 \vec{Y}_\eta - \vec{Y}_\xi \left (\mtx{D}^2 \right )^T  = \vec{0}, \\
    \partial_\xi X_\eta - \partial_\eta X_\xi &\approx \mtx{D}^1 \vec{X}_\eta - \vec{X}_\xi \left (\mtx{D}^2 \right )^T  = \vec{0}.
    \end{aligned}
    \label{discrete_metric_identities}
\end{equation}
Essentially, due to the metric identities being satisfied on the reference element, the SBP operators are applied to the reference element, which is then mapped to the physical element.
Thus, the volume term reads
\begin{equation}
    \begin{aligned}
\vec{\mathrm{VOL}}_i :=&  \sum_{j = 1}^d \left ( \sum_{k = 1}^N 2 \mtx{D}^j_{i,k} \tilde{\vec{f}} ^{\mathrm{vol},j}\left (\vec{u}_i, \vec{u}_k \right ) + \alpha \sum_{k = 1}^N \mtx{D}^j_{i,k} \tilde{\vec{H}}^{\mathrm{vol}}\left (\vec{u}_i, \vec{u}_k \right ) \jump{\tilde{\vec{g}}^j}_{i,k} +(1-\alpha)\sum_{k = 1}^N \mtx{D}^j_{i,k} \tilde{\vec{H}}\left (\vec{u}_i\right ) \jump{\tilde{\vec{g}}^j}_{i,k} \right ),
\label{volume_terms_curvilinear}
    \end{aligned}
\end{equation}
and the surface term reads
\begin{equation}
    \vec{\mathrm{SURF}} := \mtx{M}^{-1}\mtx{R}^T \left (\mtx{B}  \vec{\tilde{f}}^{\mathrm{num}} + \alpha \frac{1}{2}|\mtx{B}| \tilde{\mtx{H}}^{\mathrm{num}} \jump{\tilde{\vec{g}}} + (1-\alpha)\frac{1}{2}|\mtx{B}| \tilde{\mtx{H}} \jump{\tilde{\vec{g}}} \right ) - \sum_{j = 1}^d \mtx{M}^{-1} \mtx{R}^T \mtx{B} \mtx{N_j} \mtx{R} \vec{\tilde{f}}^j,
    \label{surface_terms_curvilinear}
\end{equation}
where
\begin{equation}
    \begin{aligned}
\tilde{\vec{f}}^{\mathrm{num}}(\vec{u}_-, \vec{u}_+) &= \sum_{j=1}^{d} \tilde{\vec{f}}^{\mathrm{num},j}(\vec{u}_-, \vec{u}_+), \quad \tilde{\vec{f}}^{\mathrm{num},j} = \hat{\vec{n}}^j \cdot \left (\vec{f}^{\mathrm{num},1}, \vec{f}^{\mathrm{num},2} \right ),\\
\mtx{H}^{\mathrm{num}}(\vec{u}_-, \vec{u}_+)\cdot \jump{\vec{g}} &= \sum_{j=1}^{d} \tilde{\vec{H}}^{\mathrm{num}}(\vec{u}_-, \vec{u}_+) \cdot \jump{\tilde{\vec{g}}^j},\\
\mtx{H}(\vec{u}_0)\cdot \jump{\vec{g}} &= \sum_{j=1}^{d} \tilde{\vec{H}}(\vec{u}_0) \cdot \jump{\tilde{\vec{g}}^j},\\
\tilde{\vec{f}}^{\mathrm{vol},j}(\vec{u}_i, \vec{u}_k) &= \mean{\hat{\vec{n}}^j}_{i,k} \cdot \left (\vec{f}^{\mathrm{vol},1}(\vec{u}_i, \vec{u}_k), \vec{f}^{\mathrm{vol},2}(\vec{u}_i, \vec{u}_k) \right ),\\
\tilde{\vec{g}}^{j} &= \mean{\hat{\vec{n}}^j} \circ \left (\vec{g}^{1}, \vec{g}^{2} \right ).
\end{aligned}
\end{equation}
An extension of entropy-preserving methods for nonconservative hyperbolic systems to curved meshes has already been presented by Waruszewski et al.~\cite{WARUSZEWSKI2022111507} for DG methods, which are a special case of SBP methods.
In fact, any SBP method can be extended to curvilinear meshes by formulating the operators on the reference element $[-1,1]^d$ and incorporating the geometric mapping through the corresponding metric identities.
Therefore, the entropy analysis carries over directly once the SBP property is available in the reference configuration.
For this reason, we do not repeat the proof of the following theorem.
It follows the same arguments as in the Cartesian case, while the treatment of the metric terms for a general fluctuation formulation has already been presented in~\cite{WARUSZEWSKI2022111507}.
The only modification required is to replace the DGSEM differentiation and quadrature operators with general SBP operators, since the derivation relies solely on the SBP property and not on the specific structure of DGSEM.
\begin{theorem}
    Consider the semi-discretization~\eqref{semi_SBP} with the volume and surface terms given by~\eqref{volume_terms_curvilinear} and~\eqref{surface_terms_curvilinear}.
    If $\alpha \in \mathbb{R}$ and the numerical fluxes $\vec{f}^{\mathrm{vol}}$, $\vec{H}^{\mathrm{vol}}$ are consistent with $\vec{f}$ and $\vec{H}$, symmetric, entropy-conservative as in Definition~\ref{Convex_d}, both the mass matrix $\mtx{M}$ and the boundary operators $\mtx{R}^T \mtx{B} \mtx{N_j} \mtx{R}$ are diagonal, and the metric terms $\vec{X}_\xi, \vec{X}_\eta, \vec{Y}_\xi, \vec{Y}_\eta$ satisfy the metric identities~\eqref{metric_identities} at the discrete level, the semi-discretization is entropy-conservative/stable, if $\alpha$ and the numerical fluxes $\vec{f}^{\mathrm{num}}$, $\vec{H}^{\mathrm{num}}$ are entropy-conservative/stable as in Definition~\ref{Convex_d}.
\end{theorem}

\subsection{Well-balanced methods using curvilinear coordinates}
\label{sec:wb_curvilinear}
In this section, we present a characterization for steady states conditions of nonconservative hyperbolic systems. The motivation behind this characterization is that it allows extending well-balanced schemes from two-point finite volume fluxes to high-order methods on curved meshes.

\begin{definition}[Steady states]\label{steady_state_definition}
    A solution $\vec{u}(x,t) \in Y \subset \mathbb{R}^n$ of~\eqref{nonconservative_system} is a steady state if
    \begin{equation}
   \partial_t \vec{u} = -\partial_x \vec{f}(\vec{u}) - \vec{H}(\vec{u}) \partial_x \vec{g}(\vec{u}) = 0, \quad \forall (x,t) \in \mathbb{R}^d \times \mathbb{R}^+.
    \end{equation}
\end{definition}
Therefore, given steady-state values $\vec{u}_i$, we seek for a semi-discretization that satisfies
\begin{equation}
\frac{\vec{f}^{\mathrm{num}}_{i+1/2} - \vec{f}^{\mathrm{num}}_{i-1/2}}{\Delta x} + \alpha \frac{\vec{H}^{\mathrm{num}}_{i+1/2} \jump{\vec{g}}_{i+1/2} + \vec{H}^{\mathrm{num}}_{i-1/2}\jump{\vec{g}}_{i-1/2}} {2 \Delta x} + (1-\alpha) \vec{H}_i \frac{\jump{\vec{g}}_{i+1/2} + \jump{\vec{g}}_{i-1/2}} {2 \Delta x}= \vec{0},
\label{semidiscrete_wb}
\end{equation}
where we have explicitly introduced again a combination of the first and third form for the discretization of the nonconservative term.
We use the following well-known definition.
\begin{definition}
    \label{finite_volume_wb}
Given $\vec{u}_i$ a steady state as in Definition~\ref{steady_state_definition}, a finite volume semi-discretization of~\eqref{nonconservative_system} is said to be a well-balanced scheme if $\partial_t \vec{u}_i = \vec{0}$.
\end{definition}
For sake of brevity, we will proceed with deriving an algebraic condition, similarly to the previous section, only for the semi-discretization~\eqref{semidiscrete_wb}, while the conditions and results for the other semi-discretization will be only given in a tabular form.

\begin{lemma}{\label{algebraic_wb}}
Let $Y \subset \mathbb{R}^n$ be open, and $\vec{f}\colon Y \to \mathbb{R}^n$,
$\vec{H}\colon Y \to \mathbb{R}^{n \times m}$, $\vec{g}\colon Y \to \mathbb{R}^m$,
$\vec{f}^{\mathrm{num}}\colon Y \times Y \to \mathbb{R}^{n}$ and
$\vec{H}^{\mathrm{num}}\colon Y \times Y \to \mathbb{R}^{n \times m}$ satisfying
\begin{equation}
\forall \vec{u} \in Y\colon \quad
\vec{f}^\mathrm{num}(\vec{u},\vec{u}) = \vec{f}(\vec{u}),
\quad
\vec{H}^{\mathrm{num}}(\vec{u},\vec{u}) = \vec{H}(\vec{u})
\end{equation}
be given. Then,
\begin{multline}
    \forall \vec{u}_-, \vec{u}_0, \vec{u}_+ \in Y\colon
    \\
    \vec{f}^{\mathrm{num}} (\vec{u}_0, \vec{u}_+)
    - \vec{f}^{\mathrm{num}}(\vec{u}_-, \vec{u}_0)
    + \alpha \frac{
        \vec{H}^{\mathrm{num}}(\vec{u}_0, \vec{u}_+) \bigl(\vec{g}(\vec{u}_+) - \vec{g}(\vec{u}_0)\bigr)
        + \vec{H}^{\mathrm{num}}(\vec{u}_-, \vec{u}_0) \bigl(\vec{g}(\vec{u}_0) - \vec{g}(\vec{u}_-)\bigr)
    }{2}\\
   + (1-\alpha) \frac{\vec{H}(\vec{u}_0)\bigl(\vec{g}(\vec{u}_+) - \vec{g}(\vec{u}_0)\bigr)
        + \vec{H}(\vec{u}_0) \bigl(\vec{g}(\vec{u}_0) - \vec{g}(\vec{u}_-)\bigr)
    }{2} = \vec{0}
    \label{algebraic_condition_wb}
\end{multline}
if and only if
\begin{multline}
    \forall \vec{u}_-, \vec{u}_+ \in Y\colon \\
    \vec{f}^{\mathrm{num}} (\vec{u}_-, \vec{u}_+)
    + \frac{\alpha}{2} \vec{H}^{\mathrm{num}}(\vec{u}_-, \vec{u}_+)
    \bigl(\vec{g}(\vec{u}_+) - \vec{g}(\vec{u}_-)\bigr) + \frac{(1-\alpha)}{2}\vec{H}(\vec{u}_-) \bigl( \vec{g}(\vec{u}_+) - \vec{g}(\vec{u}_-) \bigr )
    = \vec{f}(\vec{u}_-),
\end{multline}
\begin{multline}
    \forall \vec{u}_-, \vec{u}_+ \in Y\colon \\
    {-}\vec{f}^{\mathrm{num}} (\vec{u}_-, \vec{u}_+)
    + \frac{\alpha}{2}\vec{H}^{\mathrm{num}}(\vec{u}_-, \vec{u}_+)
    \bigl(\vec{g}(\vec{u}_+) - \vec{g}(\vec{u}_-)\bigr) + \frac{(1-\alpha)}{2}\vec{H}(\vec{u}_+) \bigl( \vec{g}(\vec{u}_+) - \vec{g}(\vec{u}_-) \bigr )
    = -\vec{f}(\vec{u}_+).
\end{multline}
\end{lemma}
\begin{proof}
The proof follows immediately by applying the same arguments as in the algebraic characterizations showed in Sections~\ref{sec:conservation_laws} and~\ref{sec:nonconservative_systems}.
\end{proof}
We immediately obtain the following result.
\begin{theorem}
    \label{well_balanced_fluxes}
A finite volume scheme~\eqref{FV_convex} is well-balanced (see Definition~\ref{finite_volume_wb}) if and only if
\begin{equation}
    \begin{aligned}
    &\vec{f}^{\mathrm{num}} (\vec{u}_i, \vec{u}_{i+1})
    + \frac{\alpha}{2}\vec{H}^{\mathrm{num}}(\vec{u}_i, \vec{u}_{i+1})
    \bigl(\vec{g}(\vec{u}_{i+1}) - \vec{g}(\vec{u}_{i})\bigr) + \frac{(1-\alpha)}{2}\vec{H}(\vec{u}_i) \bigl( \vec{g}(\vec{u}_{i+1}) - \vec{g}(\vec{u}_{i}) \bigr ),
    = \vec{f}(\vec{u}_i) \\
    -&\vec{f}^{\mathrm{num}} (\vec{u}_{i-1}, \vec{u}_i)
    + \frac{\alpha}{2}\vec{H}^{\mathrm{num}}(\vec{u}_{i-1}, \vec{u}_{i})
    \bigl(\vec{g}(\vec{u}_{i}) - \vec{g}(\vec{u}_{i-1})\bigr) + \frac{(1-\alpha)}{2}\vec{H}(\vec{u}_i) \bigl( \vec{g}(\vec{u}_{i}) - \vec{g}(\vec{u}_{i-1}) \bigr )
    = -\vec{f}(\vec{u}_{i}).
    \end{aligned}
\end{equation}
\end{theorem}
As a consequence, every well-balanced steady state is preserved by an SBP method on curved meshes if the numerical fluxes satisfy~\eqref{algebraic_condition_wb}. Indeed, we can formulate the following result.
\begin{theorem}
    \label{th:wb_curvilinear}
    Consider the semi-discretization~\eqref{semi_SBP} with the volume and surface terms given by~\eqref{volume_terms_curvilinear} and~\eqref{surface_terms_curvilinear}.
    If the numerical fluxes $\vec{f}^{\mathrm{vol}}$, $\vec{H}^{\mathrm{vol}}$ are consistent with $\vec{f}$ and $\vec{H}$, symmetric, well-balanced as in Theorem~\ref{well_balanced_fluxes}, both the mass matrix $\mtx{M}$ and the boundary operators $\mtx{R}^T \mtx{B} \mtx{N_j} \mtx{R}$ are diagonal, and the metric terms $\vec{X}_\xi, \vec{X}_\eta, \vec{Y}_\xi, \vec{Y}_\eta$ satisfy the metric identities~\eqref{metric_identities} at the discrete level, the semi-discretization is well-balanced, if the numerical fluxes $\vec{f}^{\mathrm{num}}$, $\vec{H}^{\mathrm{num}}$ are well-balanced as in Theorem~\ref{well_balanced_fluxes}.
\end{theorem}
\begin{proof}
For the proof, it is sufficient to show that the volume terms vanish, as the surface term can be treated as in the finite volume scheme.
\begin{equation}
\vec{\mathrm{VOL}}_i := \sum_{j = 1}^d \left ( \sum_{k = 1}^N 2 \mtx{D}^j_{i,k} \tilde{\vec{f}} ^{\mathrm{vol},j}\left (\vec{u}_i, \vec{u}_k \right ) + \sum_{k = 1}^N \mtx{D}^j_{i,k} \tilde{\vec{H}}^{\mathrm{vol}}\left (\vec{u}_i, \vec{u}_k \right ) \jump{\tilde{\vec{g}}^j}_{i,k} \right ).
\end{equation}
We first expand the first components, i.e., $j=1$, yielding
\begin{multline}
\sum_{k=1}^N \biggl( 2 \mtx{D}_{i,k}^1 \left (\mean{n^1_1}_{i,k} \vec{f}^{\mathrm{vol},1}(\vec{u}_i, \vec{u}_k) + \mean{n^1_2}_{i,k} \vec{f}^{\mathrm{vol},2}(\vec{u}_i, \vec{u}_k) \right )\\
+ \alpha \mtx{D}_{i,k}^1 \left ( \mean{n^1_1}\vec{H}^{\mathrm{vol},1}(\vec{u}_i, \vec{u}_k) \jump{\vec{g}^1}_{i,k} + \mean{n^1_2}\vec{H}^{\mathrm{vol},2}(\vec{u}_i, \vec{u}_k) \jump{\vec{g}^2}_{i,k} \right ) \\
+(1-\alpha) \mtx{D}_{i,k}^1 \left ( \mean{n^1_1}\vec{H}^{1}(\vec{u}_i) \jump{\vec{g}^1}_{i,k} + \mean{n^1_2}\vec{H}^{2}(\vec{u}_i) \jump{\vec{g}^2}_{i,k} \right ) \biggr) \\
 = \sum_{k=1}^N  \mtx{D}_{i,k}^1 \left (\mean{n^1_1}_{i,k} \vec{f}^1(\vec{u}_i)  + \mean{n^1_2}_{i,k} \vec{f}^2(\vec{u}_i) \right ) =
  \sum_{k=1}^N \left( \vec{f}^1(\vec{u}_i) \mtx{D}_{i,k}^1 \left ( \vec{Y}_{\eta} \right )_k - \vec{f}^2(\vec{u}_i)\mtx{D}_{i,k}^1 \left ( \vec{X}_{\eta} \right )_k \right),
\end{multline}
where we have used in the last row the well-balanced condition of the fluxes given in Theorem~\ref{well_balanced_fluxes}. Repeating the same steps also for the second direction we obtain the additional contribution
\begin{equation}
\sum_{k=1}^N \left( \vec{f}^2(\vec{u}_i)\mtx{D}_{i,k}^2 \left ( \vec{X}_{\xi} \right )_k - \vec{f}^1(\vec{u}_i)\mtx{D}_{i,k}^2 \left ( \vec{Y}_{\xi} \right )_k \right).
\end{equation}
Summing up the two contributions and considering the metric identities~\eqref{discrete_metric_identities}, the volume term vanishes.
\end{proof}

Theorem~\ref{well_balanced_fluxes} provides conditions to construct well-balanced finite volume methods.
Theorem~\ref{th:wb_curvilinear} shows that the same property extends directly to high-order SBP methods on curved meshes.
Thus, it suffices to derive the well-balanced methods within the simple framework of two-point finite volume fluxes.
Some cases covered by Theorem~\ref{well_balanced_fluxes} are lake-at-rest steady states for shallow-water equations~\cite{FJORDHOLM20115587, gassner2016shallowwater, Ranocha2017} and isothermal steady states of the compressible Euler equations with gravity~\cite{WARUSZEWSKI2022111507, artiano2025structurepreservinghighordermethodscompressible,ChandrashekarWB}.
In particular, the well-balancedness of the dispersive shallow-water system analyzed in Lemma~\ref{lemma_wb_saintemarie} extends to curvilinear meshes.
Table~\ref{table_wb} summarizes the equivalent well-balanced conditions for each semi-discretization of the nonconservative product.

\begin{table}[htbp]
    \sisetup{
  output-exponent-marker=\text{e},
  round-mode=places,
  round-precision=2
}
\centering
  \caption{Equivalence of well-balanced schemes for different semi-discretization of the nonconservative product.}
  \label{table_wb}
  \centering
  \small
    \begin{tabular}{ll}
      \toprule
       \multicolumn{1}{c}{Semi-discretization}
      & \multicolumn{1}{c}{Well-balanced condition on the right interface}\\
      \midrule
 $\vec{H}_i\frac{\jump{\vec{g}}_{i+1/2} + \jump{\vec{g}}_{i-1/2}}{2 \Delta x_i}$
& $ \vec{f}^\mathrm{num}(\vec{u}_i,\vec{u}_{i+1}) +\frac{1}{2}\vec{H}(\vec{u}_i) \jump{\vec{g}}_{i+1/2} = \vec{f}(\vec{u}_i)$
\\
 $\vec{H}_i \frac{\vec{g}_{i+1/2}^{\mathrm{num}} - \vec{g}^{\mathrm{num}}_{i-1/2}}{\Delta x_i}$
& $\vec{f}^{\mathrm{num}}(\vec{u}_i, \vec{u}_{i+1}) + \vec{H}(\vec{u}_i)\vec{g}^{\mathrm{num}}(\vec{u}_i, \vec{u}_{i+1}) = \vec{f}(\vec{u}_i)$
\\
 $\frac{\vec{H}^{\mathrm{num}}_{i+1/2} \jump{\vec{g}}_{i+1/2} + \vec{H}^{\mathrm{num}}_{i-1/2} \jump{\vec{g}}_{i-1/2}}{2 \Delta x_i}$
& $\vec{f}^\mathrm{num}(\vec{u}_i,\vec{u}_{i+1}) +\frac{1}{2}\vec{H}^{\mathrm{num}}(\vec{u}_i,\vec{u}_{i+1}) \jump{\vec{g}}_{i+1/2} = \vec{f}(\vec{u}_i)$
\\
 $\frac{\left (\vec{H} \vec{g} \right ) ^{\mathrm{num}}_{i+1/2} - \left (\vec{H} \vec{g} \right ) ^{\mathrm{num}}_{i-1/2}}{\Delta x_i} - \frac{\vec{H}^{\mathrm{num}}_{i+1/2} - \vec{H}^{\mathrm{num}}_{i-1/2}}{\Delta x_i} \vec{g}_i$
& $\vec{f}^\mathrm{num}(\vec{u}_i,\vec{u}_{i+1}) +\left ( \vec{H} \vec{g} \right ) ^{\mathrm{num}}(\vec{u}_i,\vec{u}_{i+1}) - \vec{H}^{\mathrm{num}}(\vec{u}_i, \vec{u}_{i+1}) \vec{g}(\vec{u}_i) = \vec{f}(\vec{u}_i)$ \\
$\frac{\vec{D}^-(\vec{u}_{i-1},\vec{u}_i) + \vec{D}^+(\vec{u}_i, \vec{u}_{i+1})}{\Delta x_i}$ & $\vec{D}^+(\vec{u}_i, \vec{u}_{i+1}) = \vec{0}$\\ 
      \bottomrule
    \end{tabular}
\end{table}

\section{Numerical experiments}
\label{sec:numerical_experiments}
We test the fluxes described in Section~\ref{sec:numerical_fluxes}, verifying their entropy conservation/stability properties in numerical tests.
We have implemented the methods and performed the numerical tests in Julia~\cite{bezanson2017julia} using Trixi.jl~\cite{ranocha2022adaptive,schlottkelakemper2021purely}.
Unless stated otherwise, we employ discontinuous Galerkin spectral element methods (DGSEMs).
The solution is approximated by Lagrange polynomials of degree $p$ defined at the Gauss-Lobatto-Legendre nodes.
The corresponding quadrature rule defines the mass matrix $\mtx{M}$ and the differentiation matrix $\mtx{D}_{ik} = l'_k(\xi_i)$ on the reference element $\xi \in [-1, 1]$, where $l_k$ is the $k$-th Lagrange polynomial.
This differentiation matrix satisfies the SBP property~\eqref{sbp_property} (see~\cite{gassner2013skew}), thus fits into the proposed framework.
Time integration is carried out with the fourth-order strong stability-preserving Runge-Kutta method of~\cite{kraaijevanger1991contractivity} implemented in OrdinaryDiffEq.jl~\cite{rackauckas2017differentialequations}.

\subsection{Variable-coefficient advection equation}
In this section, we assess the entropy-conservation property of the semi-discretization for the one-dimensional linear advection equation with variable coefficient.
Simulations are performed on a uniform mesh with $32$ elements, on the domain $[-1, 1]$ over a time interval $[0, 1]$, with periodic boundary conditions and with a fixed CFL number of $0.01$.
The initial condition is given by
\begin{equation}\label{initial_condition_adv}
u(0, x) = 2 + \sin(\pi (x-0.7)), \quad a(x) = 2 + \cos(\pi x).
\end{equation}
The equation is discretized using the first form~\eqref{FV_form1},  which is entropy conservative, as seen in Section~\ref{sec:fluxes_variable_coefficient_advection}.
In Table~\ref{tab:advection_variable_coefficient} we report the norm of the entropy $||U - U_0||_M$, normalized with respect to $U_0$, and the entropy residual $\vec{\omega}^T \mtx{M} \partial_t \vec{u}$ for different polynomial degrees at the final time $T = 1$, where $U = \frac{u^2}{a}$ and $\omega = 2\frac{u}{a}$. As shown in the table, the entropy is conserved up to machine precision for all polynomial degrees.
\begin{table}[htbp]
    \sisetup{
  output-exponent-marker=\text{e},
  round-mode=places,
  round-precision=2
}
\centering
  \caption{Entropy errors and entropy residual for the variable-coefficient advection scheme~\eqref{eq_adv_scheme} with initial condition~\eqref{initial_condition_adv} at the final time $T = 1$, discretized with DGSEM.}
  \label{tab:advection_variable_coefficient}
  \centering
    \begin{tabular}{r rrrrr}
      \toprule
      & \multicolumn{5}{c}{Polynomial degree $p$} \\
          & \multicolumn{1}{c}{1}
          & \multicolumn{1}{c}{2}
          & \multicolumn{1}{c}{3}
          & \multicolumn{1}{c}{4}
          & \multicolumn{1}{c}{5}\\
      \midrule
$||U - U_0||_{\vec{M}}$   & \num{9.51993e-12}  & \num{3.70763e-12} & \num{2.53573e-12}       & \num{2.35739e-12} & \num{2.47094e-12}       \\
$\vec{\omega}^T \mtx{M} \partial_t \vec{u}$  & \num{-8.88178e-16} & \num{4.44089e-16} & \num{3.33067e-16}       & \num{5.55112e-16} & \num{5.55112e-16} \\
      \bottomrule
    \end{tabular}
\end{table}

\subsection{General polynomial equation}
Here, we propose a simple numerical example to show the conservation property of the nonconservative equation with a general monomial flux of degree $k > 2$ as in~\eqref{eq_poly} and periodic boundary conditions. The initial condition is given by
\begin{equation}
u(0, x) = \sin(\pi x), \quad x \in [-1, 1].
\end{equation}
The domain is discretized with $32$ elements and the time interval is $[0, T]$, with $T = \frac{T_{\text{max}}}{2}$, where $T_{\text{max}}$ is given by~\cite{ranocha2018generalised}
\begin{equation}\label{tmax}
    T_{\text{max}} = -\Bigl (\min_{x} n(m+n-1) \pi \sin(\pi x)^{m+n-2} \cos(\pi x) \Bigr  )^{-1},
\end{equation}
and with a fixed CFL number $0.001$.
The two schemes~\eqref{ec_1_monomial} and~\eqref{ec_2_monomial}, that here we denote with EC1 and EC2 respectively, are both investigated for different pairs of $m, n \in \mathbb{N}$, and the entropy rate residual $\vec{\omega}^T \mtx{M} \partial_t \vec{u}$, the entropy error $||U - U_0||_{\mtx{M}}$, and the conservation of the mass $\vec{1}^T \mtx{M} \partial_t \vec{u}$ are reported in Tables~\ref{tab:polynomial}, \ref{tab:polynomial_2} and \ref{tab:polynomial_3}.
Note that when $n$ is even, the scheme $EC1$ is not well-behaved when $u_- = -u_+ \neq 0$, as discussed in Section~\ref{sec:general_polynomial_equation}. However, in the present test the solution remains smooth for the entire simulation and no discontinuities develop.
\begin{table}[htbp]
      \sisetup{
  output-exponent-marker=\text{e},
  round-mode=places,
  round-precision=2
}
\centering
  \caption{Entropy error $\| U - U_0 \|_M$, entropy residual $\vec{\omega}^T \mtx{M} \partial_t \vec{u}$, and conservation of mass $\vec{1}^T \mtx{M} \partial_t \vec{u}$ for the general polynomial equation~\eqref{eq_poly} for different polynomial degrees $p$ and the pairs $(m,n) = (4,4)$ and $(5,5)$ at the final time $T = \frac{T_{\text{max}}}{2}$ as in~\eqref{tmax}.}
  \label{tab:polynomial}
  \begin{subtable}{0.47\textwidth}
  \centering
    \caption{Scheme EC1~\eqref{ec_1_monomial} for $m,n = 4$.}
            \begin{tabular}{r rrr}
\toprule
\multicolumn{1}{c}{$p$} &
\multicolumn{1}{c}{$||U - U_0||_{\vec{M}}$} &
\multicolumn{1}{c}{$\vec{\omega}^T \mtx{M} \partial_t \vec{u}$} &
\multicolumn{1}{c}{$\vec{1}^T \mtx{M} \partial_t \vec{u}$} \\
\midrule
\multicolumn{1}{c}{1} & \num{2.01005e-3}  & \num{-3.49881e-2}  & 0 \\
\multicolumn{1}{c}{2} & \num{6.17886e-4} & \num{1.33115e-2}   & \num{1.38778e-17} \\
\multicolumn{1}{c}{3} & \num{1.2652e-4}  & \num{1.53745e-3}   & \num{-3.46945e-17} \\
\multicolumn{1}{c}{4} & \num{1.27113e-4} & \num{2.27842e-3}   & \num{-2.77556e-17} \\
\multicolumn{1}{c}{5} & \num{1.41035e-4} & \num{2.18205e-3}   & \num{-6.93889e-17} \\

\bottomrule
\end{tabular}

  \end{subtable}%
  \hspace{\fill}
  \begin{subtable}{0.49\textwidth}
  \centering
    \caption{Scheme EC2~\eqref{ec_2_monomial} for $m,n = 4$.}
    \begin{tabular}{r rrr}
\toprule
\multicolumn{1}{c}{$p$} &
\multicolumn{1}{c}{$||U - U_0||_{\vec{M}}$} &
\multicolumn{1}{c}{$\vec{\omega}^T \mtx{M} \partial_t \vec{u}$} &
\multicolumn{1}{c}{$\vec{1}^T \mtx{M} \partial_t \vec{u}$} \\
\midrule
\multicolumn{1}{c}{1} & \num{4.02156e-12} & \num{-9.71445e-17} & \num{6.93889e-18} \\
\multicolumn{1}{c}{2} & \num{9.59351e-12} & \num{-2.22045e-16} & \num{-2.77556e-17} \\
\multicolumn{1}{c}{3} & \num{5.09838e-11} & \num{2.60902e-15}  & \num{-9.02056e-17} \\
\multicolumn{1}{c}{4} & \num{8.23856e-11} & \num{-8.98587e-16} & \num{8.32667e-17} \\
\multicolumn{1}{c}{5} & \num{9.74345e-11} & \num{-1.249e-16}   & \num{-6.245e-17} \\
\bottomrule
\end{tabular}
  \end{subtable}%
  \\
  \begin{subtable}{0.49\textwidth}
  \centering
    \caption{Scheme EC1~\eqref{ec_1_monomial} for $m,n = 5$.}
               \begin{tabular}{r rrr}
\toprule
\multicolumn{1}{c}{$p$} &
\multicolumn{1}{c}{$||U - U_0||_{\vec{M}}$} &
\multicolumn{1}{c}{$\vec{\omega}^T \mtx{M} \partial_t \vec{u}$} &
\multicolumn{1}{c}{$\vec{1}^T \mtx{M} \partial_t \vec{u}$} \\
\midrule
\multicolumn{1}{c}{1} & \num{2.56126e-13}  & \num{-2.08167e-17}  & \num{-6.93889e-18} \\
\multicolumn{1}{c}{2} & \num{7.63255e-13}  & \num{-3.81639e-17}  & \num{-6.93889e-18} \\
\multicolumn{1}{c}{3} & \num{4.84241e-13}  & \num{1.04083e-17}   & \num{3.46945e-17}  \\
\multicolumn{1}{c}{4} & \num{3.18744e-13}  & \num{1.73472e-17}   & \num{2.08167e-17}  \\
\multicolumn{1}{c}{5} & \num{2.82917e-13}  & \num{2.08167e-16}   & \num{-2.08167e-17} \\
\bottomrule
\end{tabular}
  \end{subtable}%
  \hspace{\fill}
  \begin{subtable}{0.49\textwidth}
  \centering
    \caption{Scheme EC2~\eqref{ec_2_monomial} for $m,n = 5$.}
              \begin{tabular}{r rrr}
\toprule
\multicolumn{1}{c}{$p$} &
\multicolumn{1}{c}{$||U - U_0||_{\vec{M}}$} &
\multicolumn{1}{c}{$\vec{\omega}^T \mtx{M} \partial_t \vec{u}$} &
\multicolumn{1}{c}{$\vec{1}^T \mtx{M} \partial_t \vec{u}$} \\
\midrule
\multicolumn{1}{c}{1} & \num{3.86327e-12}  & \num{2.42861e-16}   & \num{1.38778e-17} \\
\multicolumn{1}{c}{2} & \num{1.25259e-11}  & \num{-2.84495e-16}  & \num{-2.08167e-17} \\
\multicolumn{1}{c}{3} & \num{7.74809e-12}  & \num{-5.6552e-16}   & \num{3.46945e-17} \\
\multicolumn{1}{c}{4} & \num{6.48016e-12}  & \num{-7.11237e-16}  & \num{4.85723e-17} \\
\multicolumn{1}{c}{5} & \num{7.51926e-12}  & \num{-5.96745e-16}  & \num{3.46945e-17} \\
\bottomrule
\end{tabular}
  \end{subtable}%
\end{table}

\begin{table}[htbp]
      \sisetup{
  output-exponent-marker=\text{e},
  round-mode=places,
  round-precision=2
}
\centering
  \caption{Entropy error $\|U - U_0 \|_{\vec{M}}$, entropy residual $\vec{\omega}^T \mtx{M} \partial_t \vec{u}$, and conservation of mass $\vec{1}^T \mtx{M} \partial_t \vec{u}$ for the general polynomial equation~\eqref{eq_poly} for different polynomial degrees $p$ and the pairs $(m,n) = (4,5)$ and $(5,4)$ at the final time $T = \frac{T_{\text{max}}}{2}$ as in~\eqref{tmax}.}
  \label{tab:polynomial_2}
  \begin{subtable}{0.47\textwidth}
  \centering
    \caption{Scheme EC1~\eqref{ec_1_monomial} for $(m,n) = (4,5)$.}
            \begin{tabular}{r rrr}
\toprule
\multicolumn{1}{c}{$p$} &
\multicolumn{1}{c}{$||U - U_0||_{\vec{M}}$} &
\multicolumn{1}{c}{$\vec{\omega}^T \mtx{M} \partial_t \vec{u}$} &
\multicolumn{1}{c}{$\vec{1}^T \mtx{M} \partial_t \vec{u}$} \\
\midrule
\multicolumn{1}{c}{1} & \num{3.00959e-13} & \num{-1.04083e-16} & \num{-1.29956e-5} \\
\multicolumn{1}{c}{2} & \num{8.81719e-13} & \num{-4.51028e-17} & \num{-6.00284e-6} \\
\multicolumn{1}{c}{3} & \num{5.29314e-13} & \num{-3.46945e-18} & \num{-1.1202e-5} \\
\multicolumn{1}{c}{4} & \num{3.49948e-13} & \num{1.38778e-17}  & \num{-1.14851e-5} \\
\multicolumn{1}{c}{5} & \num{3.08227e-13} & \num{1.70003e-16}  & \num{-1.11113e-5} \\
\bottomrule
\end{tabular}
  \end{subtable}%
  \hspace{\fill}
  \begin{subtable}{0.49\textwidth}
  \centering
    \caption{Scheme EC2~\eqref{ec_2_monomial} for $(m,n) = (4,5)$.}
    \begin{tabular}{r rrr}
\toprule
\multicolumn{1}{c}{$p$} &
\multicolumn{1}{c}{$||U - U_0||_{\vec{M}}$} &
\multicolumn{1}{c}{$\vec{\omega}^T \mtx{M} \partial_t \vec{u}$} &
\multicolumn{1}{c}{$\vec{1}^T \mtx{M} \partial_t \vec{u}$} \\
\midrule
\multicolumn{1}{c}{1} & \num{2.29881e-12} & \num{-2.84495e-16} & \num{5.42825e-5} \\
\multicolumn{1}{c}{2} & \num{9.84176e-12} & \num{-2.22045e-16} & \num{-3.6116e-5} \\
\multicolumn{1}{c}{3} & \num{7.2754e-12}  & \num{8.32667e-17}  & \num{2.50735e-5} \\
\multicolumn{1}{c}{4} & \num{8.6639e-11}  & \num{-2.22045e-16} & \num{-3.49375e-5} \\
\multicolumn{1}{c}{5} & \num{1.72526e-10} & \num{-3.05311e-15} & \num{4.15068e-6} \\
\bottomrule
\end{tabular}
  \end{subtable}%
  \\
  \begin{subtable}{0.49\textwidth}
  \centering
    \caption{Scheme EC1~\eqref{ec_1_monomial} for $(m,n) = (5,4)$.}
               \begin{tabular}{r rrr}
\toprule
\multicolumn{1}{c}{$p$} &
\multicolumn{1}{c}{$||U - U_0||_{\vec{M}}$} &
\multicolumn{1}{c}{$\vec{\omega}^T \mtx{M} \partial_t \vec{u}$} &
\multicolumn{1}{c}{$\vec{1}^T \mtx{M} \partial_t \vec{u}$} \\
\midrule
\multicolumn{1}{c}{1} & \num{2.0405e-13}  & \num{-3.46945e-17} & \num{-1.29956e-5} \\
\multicolumn{1}{c}{2} & \num{5.24593e-13} & \num{3.46945e-18}  & \num{-6.00284e-6} \\
\multicolumn{1}{c}{3} & \num{3.67284e-13} & \num{8.67362e-17}  & \num{-1.1202e-5} \\
\multicolumn{1}{c}{4} & \num{2.98981e-13} & \num{-6.245e-17}   & \num{-1.14851e-5} \\
\multicolumn{1}{c}{5} & \num{2.97595e-13} & \num{-2.91434e-16} & \num{-1.11113e-5} \\
\bottomrule
\end{tabular}
  \end{subtable}%
  \hspace{\fill}
  \begin{subtable}{0.49\textwidth}
  \centering
    \caption{Scheme EC2~\eqref{ec_2_monomial} for $(m,n) = (5,4)$.}
              \begin{tabular}{r rrr}
\toprule
\multicolumn{1}{c}{$p$} &
\multicolumn{1}{c}{$||U - U_0||_{\vec{M}}$} &
\multicolumn{1}{c}{$\vec{\omega}^T \mtx{M} \partial_t \vec{u}$} &
\multicolumn{1}{c}{$\vec{1}^T \mtx{M} \partial_t \vec{u}$} \\
\midrule
\multicolumn{1}{c}{1} & \num{3.12685e-12}  & \num{-3.33067e-16} & \num{-7.2775e-5} \\
\multicolumn{1}{c}{2} & \num{8.48688e-12}  & \num{4.16334e-17}  & \num{3.60184e-5} \\
\multicolumn{1}{c}{3} & \num{8.21199e-12}  & \num{-8.32667e-16} & \num{-8.9289e-6} \\
\multicolumn{1}{c}{4} & \num{4.91768e-11}  & \num{5.55112e-16}  & \num{5.49647e-5} \\
\multicolumn{1}{c}{5} & \num{1.32025e-10}  & \num{-4.30211e-16} & \num{-5.46927e-5} \\
\bottomrule
\end{tabular}
  \end{subtable}%
\end{table}

\begin{table}[htbp]
\sisetup{
  output-exponent-marker=\text{e},
  round-mode=places,
  round-precision=2
}
\centering
\caption{Entropy error $\| U - U_0 \|_{\vec{M}}$, entropy residual $\vec{\omega}^T \mtx{M} \partial_t \vec{u}$, and conservation of mass $\vec{1}^T \mtx{M} \partial_t \vec{u}$ for the general polynomial equation~\eqref{eq_poly} for different polynomial degrees $p$ and the pairs $(m,n) = (3,5)$ and $(5,3)$ at the final time $T = \frac{T_{\text{max}}}{2}$ as in~\eqref{tmax}.}
\label{tab:polynomial_3}

\begin{subtable}{0.47\textwidth}
\centering
\caption{Scheme EC1~\eqref{ec_1_monomial} for $(m,n) = (3,5)$.}
\begin{tabular}{r rrr}
\toprule
\multicolumn{1}{c}{$p$} &
\multicolumn{1}{c}{$||U - U_0||_{\vec{M}}$} &
\multicolumn{1}{c}{$\vec{\omega}^T \mtx{M} \partial_t \vec{u}$} &
\multicolumn{1}{c}{$\vec{1}^T \mtx{M} \partial_t \vec{u}$} \\
\midrule
\multicolumn{1}{c}{1} & \num{4.37299e-13} & \num{1.17961e-16}  & \num{4.16334e-17} \\
\multicolumn{1}{c}{2} & \num{1.02769e-12} & \num{3.1225e-17}   & \num{-1.38778e-17} \\
\multicolumn{1}{c}{3} & \num{5.51041e-13} & \num{-3.1572e-16}  & \num{6.93889e-18} \\
\multicolumn{1}{c}{4} & \num{3.75951e-13} & \num{-3.29597e-16} & \num{-6.93889e-18} \\
\multicolumn{1}{c}{5} & \num{3.31572e-13} & \num{-3.29597e-16} & \num{0e0} \\
\bottomrule
\end{tabular}
\end{subtable}%
\hspace{\fill}
\begin{subtable}{0.49\textwidth}
\centering
\caption{Scheme EC1~\eqref{ec_1_monomial} for $(m,n) = (3,5)$.}
\begin{tabular}{r rrr}
\toprule
\multicolumn{1}{c}{$p$} &
\multicolumn{1}{c}{$||U - U_0||_{\vec{M}}$} &
\multicolumn{1}{c}{$\vec{\omega}^T \mtx{M} \partial_t \vec{u}$} &
\multicolumn{1}{c}{$\vec{1}^T \mtx{M} \partial_t \vec{u}$} \\
\midrule
\multicolumn{1}{c}{1} & \num{1.97037e-13} & \num{6.93889e-17}  & \num{-6.93889e-18} \\
\multicolumn{1}{c}{2} & \num{3.74578e-13} & \num{-4.16334e-17} & \num{4.85723e-17} \\
\multicolumn{1}{c}{3} & \num{3.18282e-13} & \num{-2.42861e-17} & \num{-6.93889e-18} \\
\multicolumn{1}{c}{4} & \num{3.28336e-13} & \num{5.89806e-17}  & \num{-2.77556e-17} \\
\multicolumn{1}{c}{5} & \num{3.65666e-13} & \num{-5.20417e-17} & \num{1.38778e-17} \\
\bottomrule
\end{tabular}
\end{subtable}%
\\
\begin{subtable}{0.47\textwidth}
\centering
\caption{Scheme EC2~\eqref{ec_2_monomial} for $(m,n) = (5,3)$.}
\begin{tabular}{r rrr}
\toprule
\multicolumn{1}{c}{$p$} &
\multicolumn{1}{c}{$||U - U_0||_{\vec{M}}$} &
\multicolumn{1}{c}{$\vec{\omega}^T \mtx{M} \partial_t \vec{u}$} &
\multicolumn{1}{c}{$\vec{1}^T \mtx{M} \partial_t \vec{u}$} \\
\midrule
\multicolumn{1}{c}{1} & \num{1.50434e-12} & \num{-9.71445e-17}  & \num{-4.16334e-17} \\
\multicolumn{1}{c}{2} & \num{4.34746e-12} & \num{-2.63678e-16}  & \num{-4.85723e-17} \\
\multicolumn{1}{c}{3} & \num{7.68244e-12} & \num{1.94289e-16}   & \num{6.93889e-17} \\
\multicolumn{1}{c}{4} & \num{1.2421e-11}  & \num{-1.05471e-15}  & \num{2.08167e-17} \\
\multicolumn{1}{c}{5} & \num{3.48745e-11} & \num{1.02696e-15}   & \num{-9.71445e-17} \\
\bottomrule
\end{tabular}
\end{subtable}%
\hspace{\fill}
\begin{subtable}{0.49\textwidth}
\centering
\caption{Scheme EC1~\eqref{ec_1_monomial} for $(m,n) = (5,3)$.}
\begin{tabular}{r rrr}
\toprule
\multicolumn{1}{c}{$p$} &
\multicolumn{1}{c}{$||U - U_0||_{\vec{M}}$} &
\multicolumn{1}{c}{$\vec{\omega}^T \mtx{M} \partial_t \vec{u}$} &
\multicolumn{1}{c}{$\vec{1}^T \mtx{M} \partial_t \vec{u}$} \\
\midrule
\multicolumn{1}{c}{1} & \num{1.21269e-12} & \num{-7.21645e-16}  & \num{2.08167e-17} \\
\multicolumn{1}{c}{2} & \num{2.05723e-12} & \num{-2.45637e-15}  & \num{-2.08167e-17} \\
\multicolumn{1}{c}{3} & \num{1.09285e-11} & \num{3.94823e-15}   & \num{3.46945e-17} \\
\multicolumn{1}{c}{4} & \num{5.54011e-12} & \num{2.30371e-15}   & \num{1.59595e-16} \\
\multicolumn{1}{c}{5} & \num{4.36042e-11} & \num{-7.70217e-16}  & \num{-6.93889e-18} \\
\bottomrule
\end{tabular}
\end{subtable}%
\end{table}

The conservation of the entropy and the entropy residual reported in Tables~\ref{tab:polynomial}--\ref{tab:polynomial_3} show that the entropy is conserved up to machine precision of all tested pairs $(m,n)$ except scheme EC1 for $m = n = 4$, in accordance with the analysis.
In particular, although the two schemes are written in a nonconservative form, they also conserve the mass of $u$, except for the scheme EC1~\eqref{ec_1_monomial} for $m,n$ both even, where the mass is not conserved up to machine precision.
Indeed, the scheme EC1~\eqref{ec_1_monomial} conserves the mass also for the pairs $(m,n)$ equal to $(5,4)$, $(4,5)$, $(5,5)$ and $(3,5)$, as shown in the Tables~\ref{tab:polynomial_2}--\ref{tab:polynomial_3}.

\subsection{Hyperbolized Sainte-Marie equations}
We test the entropy conservation property of the numerical flux in Theorem~\ref{th:ec_saintmarie} for the hyperbolized Sainte-Marie equations~\eqref{eq:escalante_saintemarie}.
We consider the initial condition
\begin{equation}
  \label{ic:escalante_saintemarie}
\begin{aligned}
h(0,x) = 1 + \exp(\sin(2 \pi x)) \quad \text{and} \quad v(0,x) = w(0,x) = 1, \quad p(0,x) = 10
\end{aligned}
\end{equation}
with the bottom topography given by
\begin{equation}
b(x) = 0.1 \exp(\sin(2\pi x )),
\end{equation}
with $g = 1$, $c = 2\sqrt{g b_0}$ and periodic boundary conditions on the time interval $[0, 0.1]$.
We consider a domain $[0, 1]$ with 128 elements.
The CFL number is fixed to $\text{CFL}= 0.1$.
The energy-conservative numerical flux in Theorem~\ref{th:ec_saintmarie} is used with $\alpha_1 = 1/2$, $\alpha_2 = 1$ and $\alpha_3 = 2/3$. The energy error and residual are shown in Table~\ref{tab:escalante_saintemarie} for polynomial degree from 1 to 5. The energy is conserved up to machine precision for all the simulations as expected.

\begin{table}[htbp]
    \sisetup{
  output-exponent-marker=\text{e},
  round-mode=places,
  round-precision=2
}
\centering
  \caption{Entropy errors and entropy residual for the
Sainte-Marie entropy conservative scheme~\eqref{eq:ec_saintmarie_flux} with initial condition~\eqref{ic:escalante_saintemarie} at the final time $T = 0.1$, with DGSEM and varying polynomial degrees. }
  \label{tab:escalante_saintemarie}
  \centering
    \begin{tabular}{r rrrrr}
      \toprule
      & \multicolumn{5}{c}{Polynomial degree $p$} \\
          & \multicolumn{1}{c}{1}
          & \multicolumn{1}{c}{2}
          & \multicolumn{1}{c}{3}
          & \multicolumn{1}{c}{4}
          & \multicolumn{1}{c}{5}\\
      \midrule
$||U - U_0||_{\vec{M}}$   & \num{6.31117e-12}  & \num{2.00636e-12} & \num{9.70995e-13}       & \num{6.22201e-13} & \num{5.03705e-13}       \\
$\vec{\omega}^T \mtx{M} \partial_t \vec{u}$  & \num{-2.47802e-13} & \num{2.22378e-13} & \num{4.89747e-13} & \num{1.71551e-11} & \num{-1.48706e-11} \\
      \bottomrule
    \end{tabular}
\end{table}

\subsubsection{Well-balancedness on curvilinear coordinates}
Next, we assess the well-balanced property of the hyperbolized Sainte-Marie equations scheme defined in Theorem~\ref{th:ec_saintmarie} on curvilinear mesh with a discontinuous bottom topography. The computational domain is discretized with 16 quadrilateral elements, and we adopt the same warping as in~\cite{RANOCHA2025113471}. The domain size is $[ 0, \sqrt{2}]^2$ and the mesh is reported in Figure~\ref{warped_mesh_2d}. The initial condition is given by the following lake-at-rest steady state

\begin{figure}[htb]
    \centering
    \includegraphics[width=0.4\textwidth]{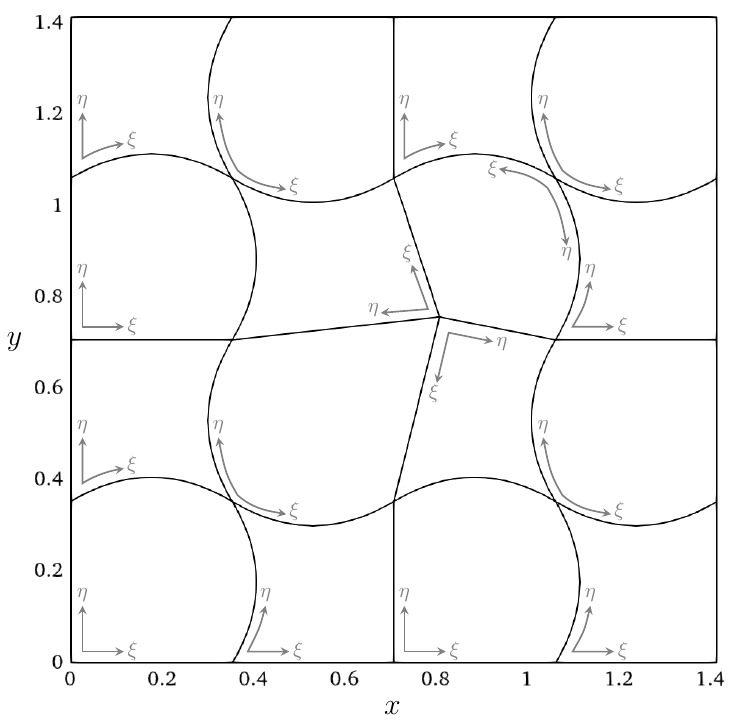}
    \caption{Warped mesh used for the well-balancedness test of the 2D hyperbolized Sainte-Marie system and for the convergence test of the 2D compressible Euler equations in nonconservative form~\cite{RANOCHA2025113471}.}
    \label{warped_mesh_2d}
\end{figure}

\begin{equation}
H_0 = h + b = 3, \quad v_1 = v_2 = w = p = 0,
\end{equation}
with $g = 9.81$, $c = 1.98$, and wall boundary conditions on the time interval $[0, 100]$. The CFL is fixed to CFL$=1$.
The bottom topography is defined as
\begin{equation}
b(x,y) =
\frac{1.5}{\exp\!\left(\tfrac12\big((x-1)^2+(y-1)^2\big)\right)}
+
\frac{0.75}{\exp\!\left(\tfrac12\big((x+1)^2+(y+1)^2\big)\right)} .
\end{equation}
To introduce a discontinuity, the bottom topography is modified in mesh
element $7$ as
\begin{equation}
b(x,y) =
2 + \frac{1}{2}\sin(2\pi x) + \frac{1}{2}\cos(2\pi y),
\qquad \text{for element } 7.
\end{equation}
The cell center of the unwarped quadrilateral corresponding to element 7 is located at $(0.9,0.5)$.
In Table~\ref{tab:escalante_saintemarie_wb}, the well-balanced errors are computed with respect to the initial condition. The lake-at-rest steady state is preserved up to machine precision for all polynomial degrees considered up to time $T = 100$.

\begin{table}[htbp]
    \sisetup{
  output-exponent-marker=\text{e},
  round-mode=places,
  round-precision=2
}
\centering
  \caption{Lake-at-rest errors for the 2D
Sainte-Marie entropy conservative and well-balanced scheme~\eqref{eq:ec_saintmarie_flux} with initial condition~\eqref{ic:escalante_saintemarie} at the final time $T = 100$, with DGSEM and varying polynomial degrees on a warped mesh.}
  \label{tab:escalante_saintemarie_wb}
  \centering
    \begin{tabular}{r rrrr}
      \toprule
      & \multicolumn{4}{c}{Polynomial degree $p$} \\
          & \multicolumn{1}{c}{2}
          & \multicolumn{1}{c}{3}
          & \multicolumn{1}{c}{4}
          & \multicolumn{1}{c}{5}\\
      \midrule

$|| H - H_0||_{\vec{M}}$   & \num{ 1.32258e-12}  & \num{1.93538e-12} & \num{2.55023e-12}       & \num{3.2363e-12} \\

$\| v_1 \|_{\vec{M}}$  & \num{5.27483e-14} & \num{1.90664e-14} & \num{7.97098e-13} & \num{9.02045e-13} \\

$\| v_2 \|_{\vec{M}}$  & \num{3.83536e-14} & \num{1.22192e-14} & \num{1.17551e-12} & \num{1.20126e-12} \\

$\| w \|_{\vec{M}}$  & \num{1.53637e-15} & \num{6.19936e-14} & \num{1.72065e-14} & \num{1.06214e-13} \\

$\| p \|_{\vec{M}}$  & \num{5.45196e-16} & \num{4.65493e-16} & \num{8.13869e-16} & \num{2.37651e-15} \\
      \bottomrule
    \end{tabular}
\end{table}

\subsection{Compressible Euler equations in nonconservative form}
In this section, we consider the two-dimensional compressible Euler equations in nonconservative form~\eqref{euler_noncons}. Following~\cite{RANOCHA2025113471}, we use the manufactured solution therein to assess the accuracy of the entropy and total energy conservative numerical flux~\eqref{eq:ec_and_kep_fluxes}. The manufactured solution in this case is given by
\begin{equation}
\varrho(t, x) = h(t,x), \quad v_1(t, x) = v_2(t, x) = 0, \quad \varrho e(t, x) = h(t,x)^2 - h(t,x)
\end{equation}
where
\begin{equation}
h(t,x) = 2 + 0.1 \sin (\sqrt{2} \pi ( x - t))
\end{equation}
on the time interval $[0, 2]$ and periodic boundary conditions, with a fixed CFL number $0.01$. The source terms can be found in the reproducibility repository~\cite{artiano2026nonconservativeRepro}. We consider again a domain $[0, \sqrt{2}]^2$ with 16 quadrilateral elements, and same warping as the previous test case (see Figure~\ref{warped_mesh_2d}~\cite{RANOCHA2025113471}).

For this numerical test case we use the SBP operators introduced by~\cite{MATTSSON20181261}, where $N$ denotes the number of nodes per element, using the entropy- and total energy conservative scheme~\eqref{eq:ec_and_kep_fluxes}. The operators were built using the library~\cite{ranocha2021sbp}.
The convergence results reported in Table~\ref{convergence_euler_warped} show that the scheme is converging to the exact solution with the expected order of accuracy.

\begin{table}[htbp]
  \sisetup{
  output-exponent-marker=\text{e},
  round-mode=places,
  round-precision=2
}
\centering
  \caption{Convergence results for the 2D compressible Euler equations
           with $K$ elements and $N$ nodes per element on the warped mesh shown in Figure~\ref{warped_mesh_2d} at the final time $T = 2$.}
  \label{convergence_euler_warped}
  \begin{subtable}{0.29\textwidth}
  \centering
    \caption{Interior order of accuracy 4.}
    \begin{tabular}{rrrr}
      \toprule
      $K$ & $N$ & $L^2$ error & EOC \\
      \midrule
   16 & 20 & \num{2.06613e-5} &  \\
   16 & 30 & \num{5.15452e-6} & 3.42 \\
   16 & 40 & \num{1.86367e-6} & 3.54 \\
   16 & 50 & \num{9.4108e-7}  & 3.06 \\
      \bottomrule
    \end{tabular}
  \end{subtable}%
  \hspace{\fill}
  \begin{subtable}{0.29\textwidth}
  \centering
    \caption{Interior order of accuracy 6.}
    \begin{tabular}{rrrr}
      \toprule
      $K$ & $N$ & $L^2$ error & EOC \\
      \midrule
   16 & 20 & \num{9.70097e-7} &  \\
   16 & 30 & \num{1.62983e-7} & 4.40 \\
   16 & 40 & \num{5.00037e-8} & 4.11 \\
   16 & 50 & \num{2.02615e-8} & 4.05 \\
      \bottomrule
    \end{tabular}
  \end{subtable}%
  \hspace{\fill}
  \begin{subtable}{0.29\textwidth}
  \centering
    \caption{Interior order of accuracy 8.}
    \begin{tabular}{rrrr}
      \toprule
      $K$ & $N$ & $L^2$ error & EOC \\
      \midrule
   16 & 20 & \num{2.86431e-8} &  \\
   16 & 30 & \num{2.73902e-9} & 5.79 \\
   16 & 40 & \num{5.58782e-10} & 5.53 \\
   16 & 50 & \num{1.83601e-10} & 4.99 \\
   \bottomrule
    \end{tabular}
  \end{subtable}%
\end{table}

\subsubsection{Baroclinic instability}
The entropy-stable numerical flux for the compressible Euler equations with the internal energy as prognostic variable~\eqref{eq:es_flux} has been tested for the 3D baroclinic instability benchmark~\cite{ullrich2014} (see also~\cite{ullrich2016}). This configuration has become a standard test for assessing the ability of atmospheric models to reproduce midlatitude baroclinic instability.
The initial condition corresponds to a balanced state, axisymmetric solution of the deep-atmosphere system (see Appendix A of Ullrich et al.~\cite{ullrich2014}). Baroclinic growth is initiated by superimposing a localized Gaussian perturbation on the zonal wind in the northern midlatitudes. The disturbance undergoes an approximately linear amplification during the first week of simulation, after which nonlinear processes lead to front formation, wave steepening, and eventual breaking.
In our case the simulation is run for 20 days, so that both the quasi-linear phase and the transition to fully nonlinear dynamics are represented. The time integration is done with the fourth-order, nine-stage, low-storage FSAL explicit Runge-Kutta method with adaptive time-stepping developed in~\cite{Ranocha2022Optimal} implemented in OrdinaryDiffEq.jl~\cite{rackauckas2017differentialequations}. The spatial discretization employs an equiangular cubed–sphere grid with $16 \times 16$ elements on the horizontal and $8$ elements in the vertical, with polynomial degree $p = 5$ on each of the six panels. As volume flux we employ the entropy and energy conservative flux \eqref{eq:ec_and_kep_fluxes}. The long-term entropy stable behavior is illustrated in Fig.~\ref{baroclinic_entropy}, which reports the evolution of the total entropy measured relative to its initial value, and shows that the entropy stability is maintained throughout the long simulation. The surface pressure at day 10 is shown in Fig.~\ref{baroclinic_plot}, which is in good agreement with results presented in~\cite{ullrich2016, WARUSZEWSKI2022111507}.

\begin{figure}[htb]
    \centering
        \includegraphics[width=\textwidth]{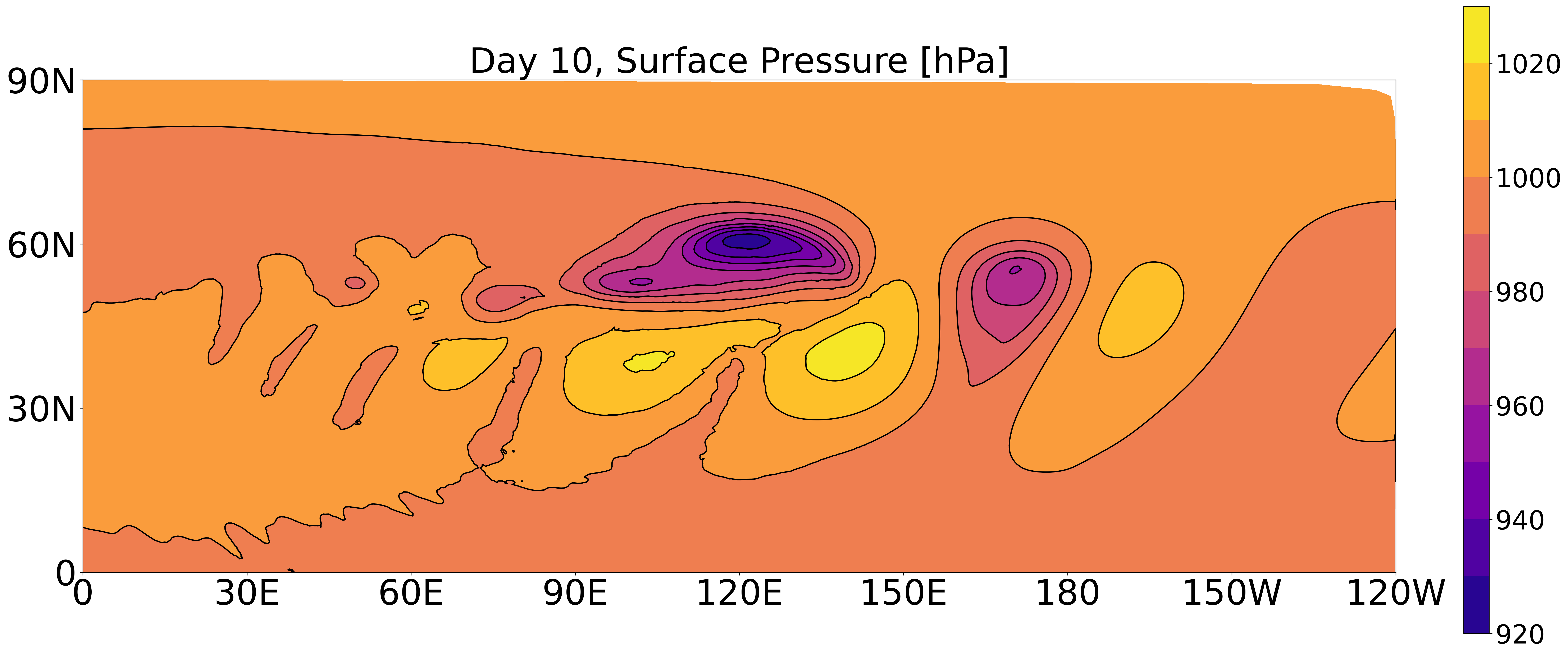}
    \caption{Contour of the surface pressure at day 10 for the baroclinic instability test case, with polynomial degree 5.}
    \label{baroclinic_plot}
\end{figure}

\begin{figure}[htb]
    \centering
        \includegraphics[width=\textwidth]{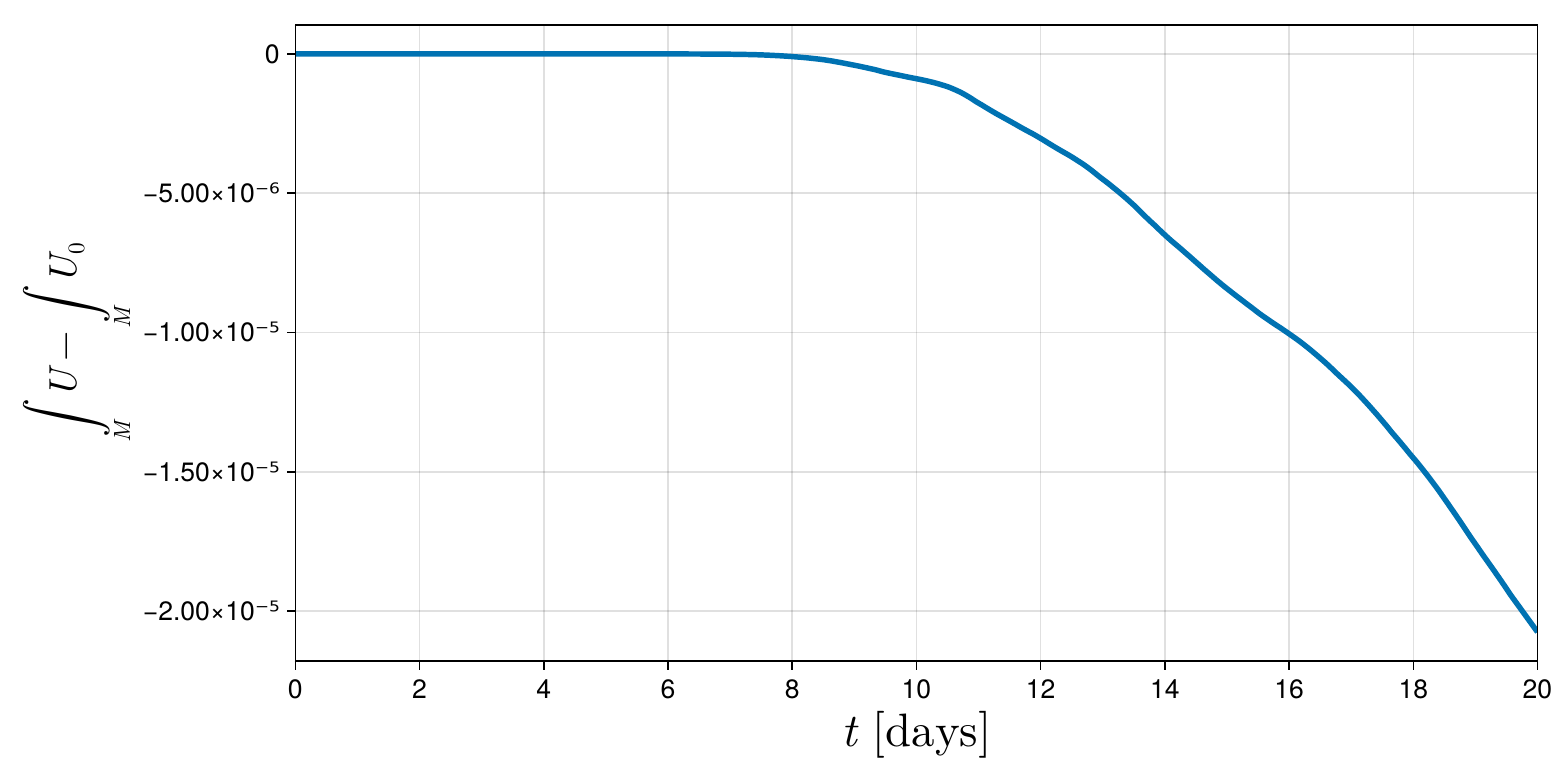}
    \caption{Evolution of the entropy over 20 days for the baroclinic instability test case, with polynomial degree 5 and the entropy stable numerical scheme~\eqref{eq:es_flux}.}
    \label{baroclinic_entropy}
\end{figure}
\section{Summary and conclusions}
\label{sec:summary_conclusions}
In this work, we have provided a framework that enables algorithmic construction of entropy-conservative/stable numerical methods for nonconservative hyperbolic systems.
Based on a purely algebraic characterization of Tadmor's condition, we have proven a necessary and sufficient condition for entropy conservation and stability for finite volume methods in Section~\ref{sec:conservation_laws}.
This is important since it can be used for any functional, even the preservation of pressure equilibria for the compressible Euler equations with real gas law, as discussed in Section~\ref{tadmor_non_convex_entropies}.
An extension to finite volume methods in fluctuation form for nonconservative systems is proposed in Section~\ref{sec:nonconservative_systems}, where we show that the entropy conservation condition of Castro et al.~\cite{Castro2013} is also a necessary condition, with its corresponding algebraic characterization.
Specific semi-discretizations of the nonconservative term are proposed which are amenable for algorithmic construction of affordable entropy-conservative and entropy-stable numerical fluxes for nonconservative hyperbolic systems.
In the second part of Section~\ref{sec:nonconservative_systems} we provide a full characterization of four specific semi-discrete forms for the nonconservative product $\vec{H}(\vec{u}) \partial_x \vec{g}(\vec{u})$, including their corresponding algebraic characterization and entropy-conservative/stable conditions.
A linear combination of two of the forms presented is then used to construct entropy-conservative and entropy-stable numerical fluxes for nonconservative hyperbolic systems in Section~\ref{sec:numerical_fluxes}, where several new numerical fluxes for different nonconservative hyperbolic systems are developed, including the Euler equations with internal energy as prognostic variable and a dispersive shallow-water model.
Furthermore, the framework is extended to high-order summation-by-parts (SBP) operators on both Cartesian and curved meshes in Sections~\ref{sec:sbp_cartesian} and~\ref{sec:sbp_curvilinear}, respectively.
This last extension also includes a discussion about split forms generated by SBP operators for specific choices of some the volume fluxes and a sufficient condition for the extension of well-balanced schemes to SBP operators on curvilinear coordinates.
Numerical experiments reported in Section~\ref{sec:numerical_experiments} show the robustness and accuracy of the novel numerical fluxes.

\appendix
\section*{Acknowledgments}

MA and HR were supported by the Deutsche Forschungsgemeinschaft
(DFG, German Research Foundation, project number 528753982
as well as within the DFG priority program SPP~2410 with project number 526031774).
This work was also supported by the Max Planck Graduate Center with the
Johannes Gutenberg University of Mainz (MPGC) and
the Mainz Institute of Multiscale Modeling (M\textsuperscript{3}ODEL).

\section*{Data availability statement}

All Julia source code and data needed to reproduce the numerical results
presented in this paper are available in our reproducibility repository~\cite{artiano2026nonconservativeRepro}.

\printbibliography

\end{document}